\documentclass{compositio}
\usepackage{amsmath}
 
\usepackage{amssymb} % for \varnothing
\usepackage[cmtip,all]{xy}
\newdir{ >}{{}*!/-5pt/@{>}} % cf xyguide exercise 14
\usepackage{array}  
\usepackage{graphicx} % for \includegraphics
\usepackage[alphabetic,nobysame]{amsrefs}
\usepackage[bbgreekl]{mathbbol}
\usepackage{amsfonts} 
\DeclareSymbolFontAlphabet{\mathbb}{AMSb} % to ensure that the meaning of \mathbb does not change
 \usepackage[mathlines]{lineno} % for line numbering
\usepackage[dvipsnames]{xcolor}
\usepackage{tikz-cd}  
  
\usepackage[T1]{fontenc}

\newcommand{\Ad}{\mathrm{Ad}}
\newcommand{\alg}{\mathrm{alg}} 
\newcommand{\CE}{\mathrm{CE}}
\newcommand{\Diff}{\mathrm{Diff}}
\newcommand{\ev}{\mathrm{ev}}
\newcommand{\id}{\mathrm{id}}
\newcommand{\mot}{\mathrm{mot}}
\newcommand{\Nov}{\mathrm{Nov}} 
\newcommand{\op}{\mathrm{op}}
\newcommand{\Sp}{\mathrm{Sp}}
\newcommand{\trc}{\mathrm{trc}}

\DeclareMathOperator{\Mor}{Mor}
\DeclareMathOperator{\cok}{cok}
\DeclareMathOperator{\End}{End}
\DeclareMathOperator{\eq}{eq}
\DeclareMathOperator{\Ext}{Ext}
\DeclareMathOperator{\fib}{fib} 
\DeclareMathOperator{\cof}{cof} 
\DeclareMathOperator*{\colim}{colim}
\DeclareMathOperator{\fil}{fil}
\DeclareMathOperator{\Fil}{Fil}
\DeclareMathOperator{\Fun}{Fun}
\DeclareMathOperator{\gr}{gr}
\DeclareMathOperator{\Gr}{Gr}
\DeclareMathOperator{\HH}{HH}
\DeclareMathOperator{\im}{im}
\DeclareMathOperator{\K}{K}
\DeclareMathOperator{\SH}{SH}
\DeclareMathOperator{\TC}{TC}
\DeclareMathOperator{\TF}{TF}
\DeclareMathOperator{\THH}{THH}
\DeclareMathOperator{\Tor}{Tor}
\DeclareMathOperator{\Tot}{Tot}
\DeclareMathOperator{\TP}{TP}

\DeclareMathOperator{\Wh}{Wh}
\newcommand{\CP}{\mathbb{C}\mathrm{P}}
\newcommand{\GM}{\mathrm{GM}}
\newcommand{\bC}{\mathbb{C}}
\newcommand{\bR}{\mathbb{R}}
\newcommand{\bE}{\mathbb{E}}
\newcommand{\bF}{\mathbb{F}} 
\newcommand{\BP}{\mathrm{BP}}
\newcommand{\bS}{\mathbb{S}}
\newcommand{\bT}{\mathbb{T}}
\newcommand{\BU}{\mathrm{BU}}
\newcommand{\bZ}{\mathbb{Z}}
\newcommand{\cA}{\mathcal{A}}
\newcommand{\can}{\mathrm{can}}
\newcommand{\cC}{\mathcal{C}}
\newcommand{\cotensor}{\mathbin{\square}}

\newcommand{\ko}{\mathrm{ko}}
\newcommand{\ksp}{\mathrm{ksp}}
\newcommand{\ku}{\mathrm{ku}}
\newcommand{\KU}{\mathrm{KU}}
\newcommand{\Lalg}{L^{\alg}}
\newcommand{\<}{\langle}
\newcommand{\longto}{\longrightarrow}
\newcommand{\MU}{\mathrm{MU}}
\newcommand{\MUP}{\mathrm{MUP}}
\newcommand{\olA}{\overline{A}}
\newcommand{\olC}{\overline{C}}
\newcommand{\olSigma}{\overline{\Sigma}}
\newcommand{\olV}{\overline{V}}
\newcommand{\tmf}{\mathrm{tmf}}
\renewcommand{\:}{\colon}
\renewcommand{\>}{\rangle}

\graphicspath{ {./EU-symbol.jpg} }

\newtheorem{theorem}{Theorem}[section]
\newtheorem{thmx}{Theorem}

\newtheorem{proposition}[theorem]{Proposition}
\newtheorem{lemma}[theorem]{Lemma}
\newtheorem{corollary}[theorem]{Corollary}

\theoremstyle{definition}
\newtheorem{definition}[theorem]{Definition}
\newtheorem{convention}[theorem]{Convention}
\newtheorem{construction}[theorem]{Construction}
\newtheorem{notation}[theorem]{Notation}

\theoremstyle{remark}
\newtheorem{remark}[theorem]{Remark}
\newtheorem{example}[theorem]{Example}

\numberwithin{equation}{section}
\numberwithin{figure}{section}
\numberwithin{table}{section}

\usepackage{hyperref}
\hypersetup{
    colorlinks=true,
    citecolor=Green,
    linkcolor=BrickRed,
    filecolor=Mulberry,
    urlcolor=NavyBlue,
}
\usepackage{caption}
\begin{document}

\title[Algebraic $K$-theory of real topological $K$-theory]
	{Algebraic $K$-theory \\ of real topological $K$-theory}

\author{Gabriel Angelini-Knoll}
\email{gja39@case.edu}
\address{Department of Mathematics, Applied Mathematics and Statistics, Case Western Reserve University, 2145 Adelbert Rd, Cleveland, OH 44106, USA}

\author{Christian Ausoni}
\email{ausoni@math.univ-paris13.fr}
\address{Laboratoire de G\'eom\'etrie, Analyse et Applications, LAGA, UMR 7539, Universit\'e Sorbonne Paris Nord, CNRS, Villetaneuse 93430, France}

\author{John Rognes}
\email{rognes@math.uio.no}
\address{Department of Mathematics, University of Oslo, P.O. Box 1053 Blindern, NO-0316 Oslo, Norway}

\shortauthors{G. Angelini-Knoll, Ch. Ausoni and J. Rognes} 
    
% ’MSC classification, keywords and grant acknowledgements’
\subjclass{Primary
14F30, % $p$-adic cohomology, crystalline cohomology
19D50, % Computations of higher $K$-theory of rings
19D55, % $K$-theory and homology, cyclic homology and cohomology
55Q51, % $v_n$-periodicity
55P43; % Spectra with additional structure
Secondary
13D03, % (Co)homology of commutative rings and algebras
19E20, % Relations of $K$-theory with cohomology theories
55N15, % Topological $K$-theory
55Q10, % Stable homotopy groups
55T25. % Generalized cohomology and spectral sequences in AT
}
\keywords{Topological cyclic homology, algebraic K-theory, real topological K-theory, chromatic
redshift, telescope conjecture, prismatic cohomology, syntomic cohomology,
even filtration, motivic filtration, evenly free descent,
homotopy fixed points, Tate construction, Adams spectral sequence, Novikov
spectral sequence, $v_2$-periodicity.}
\thanks{This project has received funding from
the European Union's Horizon 2020 research and innovation programme
under the Marie Sk\l{}odowska-Curie grant agreement No 1010342555.
\thinspace \includegraphics[scale=0.1]{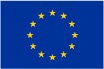}}

%\linenumbers

\begin{abstract}
We determine the $A(1)$-homotopy of the topological cyclic homology
of the connective real $K$-theory spectrum~$\ko$.  The answer has an
associated graded that is a free $\bF_2[v_2^4]$-module of rank~$52$,
on explicit generators in stems $-1 \le * \le 30$.  The calculation is
achieved by using prismatic and syntomic cohomology of~$\ko$ as introduced
by Hahn--Raksit--Wilson, extending work of Bhatt--Morrow--Scholze from the
case of classical commutative rings to $\bE_\infty$-rings.  A new feature
in our case is that there are nonzero differentials in the motivic
spectral sequence from syntomic cohomology to topological cyclic homology.
\end{abstract}

\maketitle
%\tableofcontents

\section{Introduction} \label{sec:intro}

Work of Hahn--Raksit--Wilson~\cite{HRW} extends the notions of prismatic
cohomology and syntomic cohomology to the setting of $\bE_\infty$-rings.
This produces a new tool for computing topological cyclic homology and,
consequently, algebraic $K$-theory.  In the present paper, we use this
tool to compute the $A(1)$-homotopy (cf.~Notation~\ref{not: A(1) def})
of the algebraic $K$-theory of the $\bE_\infty$-ring~$\ko$, known as
connective real $K$-theory.  Throughout, we work at the
prime $p=2$.

This paper continues the program from~\cite{AR02}, examining
the arithmetic of ring spectra through the lens of telescopically
localized algebraic $K$-theory.  In particular, the second and third
authors put forward, in a bundle of predictions known as the redshift
conjectures~\cite{AR08}, the assertion that algebraic $K$-theory
increases chromatic complexity by one.  This has now been proven in a
qualitative form, for all $\bE_\infty$-rings, in a \emph{tour de force}
by Burklund--Schlank--Yuan~\cite{BSY}, building on~\cite{Yua26},
\cite{CMNN24} and~\cite{LMMT24}.

\begin{figure}
\resizebox{\textwidth}{!}{ \input{A1TCko.inp} }
\caption{$\bF_2[v_2]$-basis for the syntomic cohomology modulo
	$(2, \eta, v_1)$ of~$\ko$, with lines of slope $-1$, $1$
	and~$1/3$ indicating multiplication by $\partial$, $\eta$
	and~$\nu$, respectively \label{fig:A1TCko}}
\end{figure}

To better understand the arithmetic of $\bE_\infty$-rings, it is still
desirable to prove more quantitative forms of the redshift conjectures,
such as the one originally appearing as the ``chromatic redshift problem''
in~\cite{Rog00}.  In the present paper, we solve this problem in the
case of~$\ko$.  Explicitly, we prove the following theorem.

\begin{thmx}[(Theorem~\ref{thm:A1htpyTCko})] \label{thm:A}
The $A(1)$-homotopy $A(1)_* \TC(\ko)$ of the topological cyclic homology
of~$\ko$ is a $\bZ/4[v_2^{32}]$-module.  The associated graded
\begin{align*}
\Gr_{\mot}^* A(1)_* \TC(\ko)
	&\cong \bF_2\{ v_2^i \mid 0 \le i \equiv 0, 1 \mod 4 \} \\
	&\qquad\oplus \bF_2[v_2] \{\partial,
		\varsigma, \nu, \lambda'_1, w, \lambda_2 \} \\
	&\qquad\oplus \bF_2[v_2] \{ \varsigma \nu, \nu^2, \partial \lambda_2,
		\nu w, \nu \lambda_2, \lambda'_1 \lambda_2 \} \\
	&\qquad\oplus \bF_2\{ v_2^j \nu^2 w \mid 0 \le j \equiv 2, 3 \mod 4 \}
\end{align*}
of its descending motivic filtration $\Fil_{\mot}^{\star} A(1)_* \TC(\ko)$ 
is a finitely generated and free $\bF_2[v_2^4]$-module of rank~$52$.
Here $|\partial| = -1$, $|\varsigma| = 1$, $|\nu| = 3$, $|\lambda'_1|
= |w| = 5$, $|v_2| = 6$ and $|\lambda_2| = 7$.
\end{thmx}

This calculation is carried out by first computing the syntomic
cohomology with $A(1)$-co\-ef\-fi\-cients, \emph{alias} modulo $(2, \eta,
v_1)$, of connective real $K$-theory.

\begin{thmx}[(Theorem~\ref{thm: syntomic A(1)})] %% \label{thm:B}
The syntomic cohomology modulo $(2, \eta, v_1)$ of~$\ko$ is
\begin{multline*}
\olA(1)_* \gr_{\mot}^* \TC(\ko)
	:= \pi_* ( \olA(1) \otimes \gr_{\mot}^* \TC(\ko) ) \\
\cong \bF_2[v_2] \{1, \partial, \varsigma, \nu, \lambda'_1, w,
	\lambda_2, \varsigma \nu, \nu^2, \partial \lambda_2, \nu w,
	\nu \lambda_2, \lambda'_1 \lambda_2, \nu^2 w \}  \,,
\end{multline*}
where the (stem, motivic filtration) bidegrees of the $\bF_2[v_2]$-module
generators are as displayed in Figure~\ref{fig:A1TCko} and
Table~\ref{tab:A1TCkogens}.
\end{thmx}

See Notation~\ref{not:FilGr} for the algebraic
filtration~$\Fil_{\mot}^{\star} (-)$ on $A(1)_* \TC(\ko)$ and its
associated graded $\Gr_{\mot}^* (-)$, see Definition~\ref{def:syntomic}
for the spectrum level filtration $\fil_{\mot}^{\star} (-)$ on
$\TC(\ko)$ and its associated graded $\gr_{\mot}^* (-)$, and see
Construction~\ref{con:A1V1Ceta} for the meaning of $\olA(1) \otimes (-)
:= \gr_{\ev/\bS}^* A(1) \otimes_{\gr_{\ev}^* \bS} (-)$.

As a consequence, we determine the $A(1)$-homotopy of the algebraic
$K$-theory of~$\ko$ and of its $2$-completion $\ko^\wedge_2$.

\begin{thmx}[(Theorems~\ref{thm:k2-complete} and~\ref{thm:k})] \label{thm:C}
There are exact sequences of $\bZ/4[v_2^{32}]$-modules
\begin{align*}
0 \to \Sigma^3 \bF_2
	\longto A(1)_* \K(\ko)
	&\overset{\trc}\longto A(1)_* \TC(\ko)
	\longto \bF_2\{\partial, \varsigma\} \to 0 \\
\intertext{and}
0 \to \Sigma^1 \bF_2 \oplus \Sigma^3 \bF_2
        \longto A(1)_* \K(\ko^\wedge_2)
        &\overset{\trc}\longto A(1)_* \TC(\ko)
        \longto \bF_2\{\partial, \varsigma\} \to 0 \,,
\end{align*}
with $|\partial| = -1$ and $|\varsigma| = 1$.
\end{thmx}

Using this, we show in Corollary~\ref{cor:nomaptotmf} that there is no
$\bE_0$-ring map $\K(\ko) \to \tmf$, i.e., no spectrum map inducing an
isomorphism on $\pi_0$.  In contrast, an $\bE_\infty$-ring map $\K(\ko)
\to E_2$ is guaranteed to exist by~\cite{BSY}.  Here $\tmf$ denotes the
topological modular forms spectrum, and $E_2$ is a suitable height~$2$
Lubin--Tate theory.

As usual, we let $K(n)$ denote the ($2$-primary) height~$n$ Morava
$K$-theory, with coefficient ring $K(n)_* = \bF_2[v_n^{\pm 1}]$, and
let~$L_n$ and~$L_n^f$ denote Bousfield localization at $K(0) \oplus
\dots \oplus K(n)$ and at $T(0) \oplus \dots \oplus T(n)$, respectively,
where $T(m) = v_m^{-1} F(m)$ is the telescope of any $v_m$ self-map of
a type~$m$ finite $2$-local spectrum $F(m)$.

We say that a spectrum~$X$ satisfies the (strong, $2$-primary)
\emph{height~$n$ telescope conjecture} if the canonical map $L_n^f
X \to L_n X$ is an equivalence.  By recent groundbreaking work of
Burklund--Hahn--Levy--Schlank~\cite{BHLS}, we know for each $n\ge2$ that
not all spectra have this property.  However, it is still interesting to
consider the question of which spectra do satisfy this conjecture, see for
example \cite{MR99}*{Conjecture~7.3}.  In Theorem~\ref{thm:telescope} we
deduce from~\cite{HRW} and~\cite{CMNN20} that the spectra $\K(\ko)$,
$\K(\ko^\wedge_2)$ and $\TC(\ko)$ all satisfy the height~$2$ telescope
conjecture.

We say that a spectrum~$X$ has \emph{telescopic complexity~$n$} if the
map $X \to L_n^f X$ is an equivalence in all sufficiently large degrees.
In~\cite{AR08}, the redshift conjecture was phrased in terms of this
property.  In Theorem~\ref{thm:LQ} we apply our Theorems~\ref{thm:A}
and~\ref{thm:C} to confirm that the spectra $\K(\ko)_{(2)}$,
$\K(\ko^\wedge_2)_{(2)}$ and $\TC(\ko)_{(2)}$ (as well as their
$2$-completions) have telescopic complexity~$2$, as predicted.

\medskip
\noindent {\bf Organization.}
In Section~\ref{sec:THHdescent} we give a presentation for the associated
graded $\gr_{\mot}^* \THH(\ko)$ of the motivic filtration on~$\THH(\ko)$,
and calculate its bigraded homotopy modulo~$(2, \eta, v_1)$ and
modulo~$(2, v_1)$.  In Section~\ref{sec:filtering} we compare two filtrations leading to
spectral sequences calculating the associated graded objects $\gr_\mot^*
\TC^{-}(\ko)$ and $\gr_\mot^* \TP(\ko)$ of the motivic filtrations on
$\TC^{-}(\ko)$ and $\TP(\ko)$, with appropriate finite coefficients. In Section~\ref{sec:detection} we show how some
homotopy classes from~$A(1)$ are detected in homotopy modulo~$(2, \eta,
v_1)$ of $\gr_{\mot}^* \TC^{-}(\ko)$.  In Section~\ref{sec:prismatic} we
study the $\bT$-Tate spectral sequence calculating prismatic cohomology,
i.e., the homotopy modulo~$(2, v_1, v_2)$ and~$(2, \eta, v_1, v_2)$
of $\gr_{\mot}^* \TP(\ko)$.  In Section~\ref{sec:syntomic} we study the
$\bT$-homotopy fixed point spectral sequence calculating the homotopy
of $\gr_{\mot}^* \TC^{-}(\ko)$ modulo~$(2, v_1, v_2)$ and~$(2, \eta,
v_1, v_2)$, and use this to calculate syntomic cohomology, i.e., the
homotopy of $\gr_{\mot}^* \TC(\ko)$ modulo~$(2, v_1)$ and~$(2, \eta,
v_1)$.  Finally, in Section~\ref{sec:tc} we determine the differential
pattern in the motivic spectral sequence calculating $A(1)_* \TC(\ko)$,
and use known facts about the cyclotomic trace map to deduce corresponding
results for $A(1)_* \K(\ko^\wedge_2)$ and~$A(1)_* \K(\ko)$.

\medskip
\noindent {\bf Conventions.}
We let $\cA$ and~$\cA^\vee$ denote the mod~$2$ Steenrod algebra and
its dual, respectively.  We set $H_* (X) := H_* (X; \bF_2)$ and write
\[
\nu \: H_* (X) \longto \cA^\vee \otimes H_* (X)
\]
for the $\cA^\vee$-coaction.  We write $\cA(1)$ for the subalgebra
of~$\cA$ generated by $Sq^1$ and~$Sq^2$, and note that $\cA(1)^\vee
= \bF_2[\xi_1, \xi_2]/(\xi_1^4, \xi_2^2)$.

We say that a spectrum~$X$ is \emph{even} if its homotopy groups $\pi_*(X)
= X_*$ are concentrated in even degrees. 
 
Let $\bZ^{\op}$ be the category whose objects are integers, such that
$\Mor_{\bZ^{\op}}(n,m) = *$ if $n\ge m$ and $= \varnothing$ otherwise.
Let $\bZ^{\delta}$ be the integers as a discrete category.  Given a
presentably symmetric monoidal stable $\infty$-category~$\cC$, we write
\[
\cC^{\fil} := \Fun(\bZ^{\op}, \cC)
	\qquad\text{and}\qquad
\cC^{\gr} := \Fun(\bZ^{\delta}, \cC) \,.
\]
Given $I^{\star} \in \cC^{\fil}$ and $J^* \in \cC^{\gr}$, we write
$I^w := I(w)$ and $J^w := J(w)$.  We recall that there is a symmetric monoidal
functor
\[
	\gr^*\: \cC^{\fil} \longto \cC^{\gr}
\] 
defined on an object~$I^{\star}$ by $\gr^w I = I^w /I^{w+1}$. 
As in the
notation above, for consistency we use the superscript~$\star$ as in
$I^{\star}$ for a filtered object, the superscript~$*$ as in $J^*$
for a graded object, and a superscript~$\bullet$ as in $K^{\bullet}$
for a cosimplicial object.
In Section 3 we also use increasing filtrations and
homological grading, in which case $\fil_w := \fil^{-w}$ and $\gr_w := \gr^{-w}$.

We use the terminology from \cite{Isa19}*{Definition~4.1.2} for (hidden)
extensions in spectral sequences. We use the term algebra spectral sequence in the sense of~\cite{HR24}*{Definition~4.9}, and
we will also consider module spectral sequences over such algebra
spectral sequences. 

In Sections~\ref{sec:THHdescent}--\ref{sec:syntomic}, we use the following
conventions:
We write $\THH(A/R) = |[q] \mapsto A \otimes_R \dots \otimes_R A|$ for
the relative topological Hochschild homology of an~$R$-algebra~$A$, with
$q+1$ copies of~$A$ in simplicial degree~$q$, and simply write~$\THH(A)$
when $R = \bS$.
We implicitly $2$-complete each of the following invariants:
$\THH := \THH(-)^\wedge_2$, $\TC^{-} := ( \THH(-)^{h\bT}
)^\wedge_2$, $\TP := ( \THH(-)^{t\bT} )^\wedge_2$,
$\TC := \TC(-)^\wedge_2$ and $\K := \K(-)^\wedge_2$.
We write $\ko$, $\ku$, $\KU$, $\MU$, $\MUP$, $\bS$ and~$\BP$ for the
$2$-completions of connective real $K$-theory, connective complex
$K$-theory, periodic complex $K$-theory, complex cobordism, periodic
complex cobordism, the sphere spectrum and the Brown--Peterson spectrum,
respectively.
We will simply write~$\bZ$ for the $2$-adic integers.
We write $\Sp_2$ for the $\infty$-category of $2$-complete
spectra with symmetric monoidal product $\otimes$.  Note
that our smash product is implicitly $2$-completed in
Sections~\ref{sec:THHdescent}--\ref{sec:syntomic}, so that $\otimes_R :=
(- \otimes_R - )^\wedge_2$ for any $2$-complete $\bE_\infty$-ring~$R$, and we omit~$R$ from the notation when $R$ is the $2$-complete
sphere spectrum.
We also write $\otimes$ for the (underived) tensor product over the
$2$-adic integers and expect the intended meaning to be clear from
context.
We write~$\bT$ for the circle regarded as the group of complex numbers
of modulus~$1$, and set $\cC^{B\bT} := \Fun(B\bT, \cC)$. We write $C_2 = \{\pm 1\}\subset \bT$ for the subgroup of order $2$.

In Section~\ref{sec:tc}, as in the present Section~\ref{sec:intro},
we explicitly include notation for $2$-completion, especially
in the argument of $\K(-)$.  Note that the canonical map
$\TC(\ko)^\wedge_2\to \TC(\ko^\wedge_2)^\wedge_2$ is an equivalence by
\cite{Mad94}*{pp.~274--275} (cf.~\cite{NS18}*{pp.~351--352}), so we can
omit $2$-completion in the argument of $\TC(-)^\wedge_2$.

\acknowledgements{
The first author would like to thank Jeremy Hahn and Dylan Wilson
for helpful conversations during the course of this project.
We thank Robert Bruner for freely sharing his minimal resolution
calculator~{\tt ext}, which is used in Section~\ref{sec:detection}.
We thank Ishan Levy for pointing out the proof by descent of
Theorem~\ref{thm:telescope}. 
We also thank the anonymous referees for useful comments, which have improved the paper. 
}
%\medskip
%{\bf Acknowledgments.} 
%The first author would like to thank Jeremy Hahn and Dylan Wilson
%for helpful conversations during the course of this project.
%We thank Robert Bruner for freely sharing his minimal resolution
%calculator~{\tt ext}, which is used in Section~\ref{sec:detection}.
%We thank Ishan Levy for pointing out the proof by descent of
%Theorem~\ref{thm:telescope}. 
%We also thank the anonymous referees for useful comments, which have improved the paper. 
%This project has received funding from
%the European Union's Horizon 2020 research and innovation programme
%under the Marie Sk\l{}odowska-Curie grant agreement No 1010342555.
%\thinspace \includegraphics[scale=0.1]{EU-symbol}

\section{Hochschild homology and motivic filtrations} \label{sec:THHdescent}

We first introduce filtrations on $\THH$, $\TC^{-}$, $\TP$ and~$\TC$.
The reader is encouraged to read \cites{BHM93, BM94, BM95, HN20} for
background on these invariants and~\cite{HRW} for a thorough account
of the filtrations that we use in this paper.

\begin{definition}
Given a map $f \: A \to B$ of $2$-complete $\bE_\infty$-rings, we write
$C^{\bullet}(B/A)$ for the associated cosimplicial \emph{Amitsur complex}
with $C^q(B/A) = B^{\otimes_A q+1}$.
\end{definition}
 
Recall that there is a map of $\bE_\infty$-rings $c \: \ko \to \ku$ called
the \emph{complexification map}, where $\ko_* = \bZ[\eta, A, B]/(2\eta,
\eta^3, \eta A, A^2-4B)$, $\ku_* = \bZ[u]$, $A \mapsto 2 u^2$ and $B
\mapsto u^4$, with $|\eta| = 1$, $|A| = 4$, $|B| = 8$ and $|u| = 2$.

\begin{lemma}\label{lem:generators-choice}
There exists an $\bE_\infty$-ring map $\MU \to \ku$.  We can choose the
generators of $\MU_* = \pi_*(\MU) = \bZ[x_i \mid i\ge1]$ so that
$x_1 \mapsto u$ and $x_i \mapsto 0$ for each $i\ge2$.  Here $|x_i| = 2i$.
\end{lemma}

\begin{proof}
The proof of \cite{HY20}*{Theorem~4.3} shows the existence of $\bE_\infty$-ring maps $\MU \to \MUP \to \KU$.  Passing to connective covers gives
the factorization through~$\ku$.  Recall our running convention that $\bZ$ denotes the $2$-adic integers.  Both $\MU$ and~$\ku$ have first
Postnikov $k$-invariant the generator of $H^3(H\bZ; \bZ) \cong \bZ/2$,
so $\bZ\{x_1\} = \pi_2(\MU) \to \pi_2(\ku) = \bZ\{u\}$ must be an
isomorphism.  
Dividing the first choice of~$x_1$ by a unit, we can assume
that $x_1 \mapsto u$.  For each $i\ge2$ we can then subtract a multiple
of $x_1^i$ from the first choice of~$x_i$ to ensure that $x_i \mapsto 0$.
\end{proof}

We hereafter fix $\MU \to \ku$ and the~$x_i$ as in the lemma above.
This provides $\bE_\infty$-ring maps $C^q(\MU/\bS) \to C^q(\ku/\ko)$,
compatibly for all $q \ge 0$.  We write
\[ 
\tau_{\ge \star} \: \Sp_2\longto \Sp_2^{\fil}
\]
for the monoidal Whitehead filtration \cite{Lur17}*{Proposition~1.4.3.6,
Example~2.2.1.10}. Note that the circle group~$\bT$ acts freely on $S^1 \subset S^3 \subset \dots \subset S^\infty = E\bT$. We write $F(S^{2\ell+1}_+, \THH)^{\bT}$ for the $\ell$-th
order approximate fixed points functor, taking $A/R$ to $F(S^{2\ell+1}_+,
\THH(A/R))^{\bT} \simeq \lim_{\CP^\ell} \THH(A/R)$, as in~\cite{BR05}*{\S2}
and~\cite{AKACHR}*{\S7}.  Here $\CP^\ell \subset \CP^\infty = B\bT$.
We also write $\THH^{hC_2}$ and~$\THH^{tC_2}$ for the functors obtained
by applying $C_2$-homotopy fixed points and the $C_2$-Tate construction
to $\THH$, respectively.

\begin{definition} \label{def: motivic filtrations on THH TN and TP}
For $F \in \{\THH, \TC^{-}, \TP, F(S^{2\ell+1}_+,\THH)^{\bT},\THH^{hC_2},\THH^{tC_2}\}$ we define $\bE_\infty$
algebras
\begin{align*}
\fil_{\mot}^{\star} F(\ko) & := \Tot \left(
	\tau_{\ge 2\star} F(C^{\bullet}(\ku/\ko)/C^{\bullet}(\MU/\bS))
	\right) \\
\gr_{\mot}^* F(\ko) & := \gr^* \left(
	\fil_{\mot}^{\star} F(\ko) \right)
\end{align*}
in $\Sp_2^{\fil}$ and $\Sp_2^{\gr}$, respectively.  Here $\Tot$ denotes
the totalization of a cosimplicial object.  We refer to $\gr_{\mot}^w
\TP(\ko)$ as the weight~$w$ \emph{prismatic cohomology} spectrum
of~$\ko$, and $\pi_* \gr_{\mot}^w \TP(\ko)$ as its the weight~$w$
prismatic cohomology groups.
\end{definition}

\begin{remark}\label{rem-well-behaved} 
We will show in Proposition~\ref{prop: equiv of gr modules} that 
\[ \THH(C^q(\ku/\ko) / C^q(\MU/\bS))=C^{q}(\THH(\ku/\MU)/\THH(\ko))\] 
is even and $2$-torsion free. It follows that $F(C^q(\ku/\ko) / C^q(\MU/\bS))$ is even for each functor
$F \in \{\THH, \TC^{-}, \TP, F(S^{2\ell+1}_+,\THH)^{\bT},\THH^{hC_2},\THH^{tC_2}\}$ 
and $q\ge 0$. 
Hence the double-speed Whitehead filtration in
the definition of $\fil_{\mot}^{\star} F(\ko)$ will be well-behaved for each $F$ in the list above. 
\end{remark}
\begin{remark}\label{rem-agreement} 
We will verify in Proposition~\ref{prop: relation to
motivic filtration} that $\pi_*\gr_{\mot}^* \TP(\ko)$ agrees
with prismatic cohomology in the sense of \cite{HRW}*{Definition~1.3.4},
building on~\cite{BMS19}. 
\end{remark}
\begin{remark}\label{rem-class-t} 
Note that when $\THH(A/R)$ is even, we are considering the filtered spectrum defined by $\tau_{\ge
2\star} \TP(A/R) = \tau_{\ge 2\star} (\THH(A/R)^{t\bT})$, which is not
the same as the filtered spectrum $(\tau_{\ge 2\star} \THH(A/R))^{t\bT}$.  In particular,
the class~$t$ in the associated graded of the $\bT$-Tate filtration,
lifting the generator of $\pi_{-2} (H\bF_2)^{t\bT}$, has weight~$-1$
(and motivic filtration~$0$, see below).  The same applies for~$t$
lifting the generator of $\pi_{-2} (H\bF_2)^{h\bT}$, in the associated
graded of the $\bT$-homotopy fixed point filtration for $\TC^{-}(A/R)
= \THH(A/R)^{h\bT}$.
\end{remark}

We also fix terminology for gradings.

\begin{definition} \label{def:gradings}
Given $M^* \in \Sp_2^{\gr}$ and $x \in \pi_n M^w$ we say that $x$ has
\emph{stem}~$n$, \emph{weight}~$w$ and \emph{motivic filtration} $2w-n$,
and write $\|x\| = (n,2w-n)$.  In \cite{HRW}*{Definition~1.4.2}, stem
and motivic filtration are called degree and Adams weight, respectively.
We refer to the (Adams indexed, Bousfield--Kan) spectral sequence
\[
E_2^{n,2w-n} = \pi_n \gr_{\mot}^w F(\ko) \Longrightarrow \pi_n F(\ko)
\]
for each $F \in \{\THH, \TC^{-}, \TP,F(S^{2\ell+1}_+,\THH)^{\bT},\THH^{hC_2}, \THH^{tC_2}\}$ as the \emph{motivic spectral
sequence}. We will show in Lemma~\ref{convergence} that the motivic spectral sequence converges in each case. 

We plot each motivic spectral sequence with stem on the horizontal axis
and motivic filtration on the vertical axis.  Given this convention,
if $\|x\| = (n,2w-n)$ then $\|d_r(x)\| = (n-1,2w-n+r)$.  Note that we
write $E_r^{n,2w-n}$ where it is also standard to write $E_r^{2w-n,2w}$
in the literature.  With these conventions, the motivic spectral sequence
$E_r^{n,s}$-term is concentrated in internal degrees $n+s = 2w$ and,
consequently, $d_r = 0$ for all even integers $r\ge 2$.
\end{definition}

Our first aim is to show that the filtrations from Definition~\ref{def:
motivic filtrations on THH TN and TP} agree with the motivic filtrations
considered in~\cite{HRW}*{Definition~4.2.1}.
%%%% 
\begin{definition}[\cite{HRW}*{Definition~2.2.15(4)}]
A map $A \to B$ of $\bE_\infty$-rings is  \emph{$2$-completely evenly free}
if for each nonzero even $\bE_\infty$ $A$-algebra $C$ of bounded $2$-power torsion the pushout $B
\otimes_A C$ is equivalent as a $C$-module to a nonzero wedge sum of
even suspensions of~$C$. In the future, we will omit ``$2$-completely'' and simply refer to such maps of $\bE_{\infty}$-rings as evenly free maps, to be consistent with our conventions. 
\end{definition}

\begin{remark}
This condition on $B \otimes_A C$ is equivalent to $\pi_*(B \otimes_A C)$
being a nonzero free $C_*$-module on generators in even stems.  An evenly
free map $A \to B$ of $\bE_\infty$-rings is $2$-completely eff (= evenly
faithfully flat), in the sense of~\cite{HRW}*{Definition~2.2.15(2)},
as per~\cite{HRW}*{Remark 2.2.16}. 
\end{remark}

\begin{lemma} \label{lem: complexification is an evenly free map}
The complexification map $\ko \to \ku$ is evenly free.
Moreover, for each even $\bE_\infty$ $\ko$-algebra~$C$
there is an isomorphism of $C_*$-algebras
\[
\pi_* (\ku \otimes_{\ko} C) \cong C_*[b_1]/(\hat b_1^2) \,,
\]
where $\hat b_1^2 = b_1^2+c_2b_1 +c_4$ for some $c_2, c_4 \in C_*$
with $|c_2| = 2$ and $|c_4| = 4$.
\end{lemma}

\begin{proof}
Recall the Wood (homotopy) cofiber sequence
\begin{equation} \label{eq:Wood}
\Sigma \ko \overset{\eta}\longto \ko
	\overset{c}\longto \ku
	\overset{R}\longto \Sigma^2\ko
\end{equation}
of $\ko$-modules, where we write~$R$ for the map satisfying $R \circ
(u\cdot -) = \Sigma^2 r$, with $r \: \ku \to \ko$ the realification map.
The resulting cofiber sequence
\[
\Sigma C\overset{\eta}\longto C
	\longto \ku \otimes_{\ko} C
	\overset{R}\longto \Sigma^2 C
\]
induces a short exact sequence of $C_*$-modules
\begin{equation} \label{eq: example ses}
0 \to C_* \longto \pi_* (\ku \otimes_{\ko} C)
	\overset{R}\longto \Sigma^2 C_* \to 0 \,,
\end{equation}
since $C$ is even so that $\eta \: \Sigma C_* \to C_*$ is zero.
We conclude that $ \pi_* (\ku \otimes_{\ko} C)$ is concentrated in
even degrees.

A choice of a class $b_1 \in \pi_2(\ku \otimes_{\ko} C)$ with $R(b_1)
= \Sigma^2 1$ is equivalent to a choice of a splitting of~\eqref{eq:
example ses}, so that $\pi_* (\ku \otimes_{\ko} C) \cong C_* \{1, b_1 \}$
is free over~$C_*$ and nonzero if $C \ne 0$.  Writing $\hat b_1^2 = b_1^2
+ c_2 b_1 + c_4$ for the monic polynomial that vanishes in $\pi_4(\ku
\otimes_{\ko} C)$, there is an isomorphism
\[
\pi_* (\ku \otimes_{\ko} C)\cong C_*[b_1]/(\hat b_1^2)
\]
of $C_*$-algebras.
\end{proof}

\begin{remark}
Note that $c_2$ and~$c_4$ in the statement of Lemma~\ref{lem:
complexification is an evenly free map} need not be zero.  For example, when $C
= \ku$ then $b_1$ can be chosen as in \cite{DLR22}*{Lemma~5.1} so that
$\hat b_1^2 = b_1^2-u b_1$. 
\end{remark}

\begin{lemma}\label{criteria-for-eff}
Let $A \to B$ be a map of connective $\bE_\infty$-rings with bounded $2$-power torsion.  If $\pi_*(B
\otimes_A H\pi_0 A)$ is a nonzero, free and even $\pi_0 A$-module,
then $A \to B$ is evenly free.
\end{lemma}
\begin{proof} 
Let $C$ be a nonzero even $\bE_\infty$ $A$-algebra of bounded $2$-power torsion.  Associated to the
double-speed Whitehead filtration $\tau_{\ge2\star} C$ we have an algebra
spectral sequence $E_1(C) \Longrightarrow \pi_*(C)$ with $E_1^{*,w}(C) =
\pi_{*-2w}(H \pi_{2w} C)$.  Similarly, the induced filtration $B \otimes_A
\tau_{\ge2\star} C$ yields a spectral sequence $E_1(B \otimes_A C)
\Longrightarrow \pi_*(B \otimes_A C)$ with $E_1^{*,w}(B \otimes_A C) =
\pi_{*-2w}(B \otimes_A H\pi_{2w} C)$, which we view as a module spectral
sequence over the first one.  Both spectral sequences are conditionally
convergent, since $A$ and~$B$ are assumed to be connective.
A choice of $\pi_0 A$-module basis for $\pi_*(B \otimes_A H\pi_0 A)$,
and the equivalence 
$$
B \otimes_A H\pi_{2w} C
	\simeq (B \otimes_A H\pi_0 A) \otimes_{H\pi_0 A} H\pi_{2w} C \,,
$$
show that~$E_1(B \otimes_A C) \cong \pi_*(B \otimes_A H\pi_0 A)
\otimes_{\pi_0 A} E_1(C)$ is a nonzero free $E_1(C)$-module on generators
in (filtration~$0$ and) even stems.  Hence both spectral sequences are
concentrated in even stems, collapse at $E_1 = E_\infty$, and are strongly
convergent by~\cite{Boa99}*{Theorem~8.10}.  It follows that the $E_1(C)$-module
generators lift to the abutment, giving a $\pi_*(C)$-module basis for
$\pi_*(B \otimes_A C)$.
\end{proof}

\begin{remark}
Without our implicit $2$-completion assumptions and the bounded $2$-power torsion assumption, the same argument as in the proof of the lemma above provides a criterion for a map of connective $\bE_{\infty}$-rings to be evenly free in the sense of~\cite{HRW}*{Definition~2.2.15(3)}.
\end{remark}

\begin{proposition} \label{prop: evenly free map}
The map $\THH(\ko) \to \THH(\ku/\MU)$, induced by the complexification
map $c \: \ko \to \ku$ and the unit map $\bS \to \MU$, is evenly free.
\end{proposition}

\begin{proof}
Since $\ko$, $\ku$ and $\MU$ are connective and have finite type, it follows that $\THH(\ko)$ and $\THH(\ku/\MU)$ have finite type, and therefore have bounded $2$-power
torsion. Hence, by Lemma~\ref{criteria-for-eff}, it suffices to show that
\[
\pi_*(\THH(\ku/\MU) \otimes_{\THH(\ko)} H\bZ)
\]
is a nonzero, free and even $\bZ$-module. The $\bE_\infty$-ring
map $\THH(\ko) \to H\bZ$ can be rewritten as $\THH(\ko/\bS) \to
\THH(H\bZ/H\bZ)$, which is induced by $\ko \to H\bZ$ over $\bS \to H\bZ$.
Then
\begin{align*}
\THH(\ku/\MU) \otimes_{\THH(\ko)} H\bZ
	&\cong \THH(\ku/\MU) \otimes_{\THH(\ko)} \THH(H\bZ/H\bZ) \\
	&\simeq \THH(\ku \otimes_{\ko} H\bZ / \MU \otimes H\bZ) \,,
\end{align*}
where the last equivalence holds because $\THH$ commutes with
pushouts of $\bE_\infty$-rings.

Recall that $\pi_*(\MU \otimes H\bZ) \cong H_*(\MU; \bZ) = \bZ[b_k \mid
k\ge1]$ as $\bZ$-algebras, with $|b_k| = 2k$.  Since $\MU \to \ku$ is
$4$-connected by Lemma~\ref{lem:generators-choice}, and $\bS \to \ko$ is $3$-connected, it follows
that $\MU \otimes H\bZ \to \ku \otimes_{\ko} H\bZ$ is $4$-connected.
Hence the generators~$b_k$ may be chosen so that $\pi_*(\MU \otimes
H\bZ) \to \pi_*(\ku \otimes_{\ko} H\bZ) = \bZ[b_1]/(b_1^2)$ satisfies
$b_1 \mapsto b_1$ and $b_k \mapsto 0$ for $k\ge2$, where $b_1$ in the
target is fixed as in Lemma~\ref{lem: complexification is an evenly free map}.

To determine $\THH$ of $\ku \otimes_{\ko} H\bZ$ relative to $\MU \otimes H\bZ$,
we first study the K{\"u}nneth (or $\Tor$) spectral sequence
\begin{align*}
E^2 &= \Tor_*^{\bZ[b_k \mid k\ge1]} (\bZ[b_1]/(b_1^2), \bZ[b_1]/(b_1^2)) \\
	&\Longrightarrow \pi_*( (\ku \otimes_{\ko} H\bZ) \otimes_{(\MU \otimes H\bZ)}
		(\ku \otimes_{\ko} H\bZ)) \,.
\end{align*}
Since $b_1^2$ and~$b_k$ for $k\ge 2$ act trivially on $\bZ[b_1]/(
b_1^2)$, we compute that
\[
E^2 \cong \bZ[b_1]/(b_1^2)
	\otimes \Lambda(\sigma b_1^2, \sigma b_k \mid k\ge 2)
\]
with $\sigma b_1^2, \sigma b_k \in \Tor_1$, where the triviality of
the square $(\sigma b_1^2)^2$ follows from the fact that the product
in $\Tor$ is given by the shuffle product.  Since the algebra generators
are all in filtration $\le 1$, the K{\"u}nneth spectral sequence collapses
at the $E^2$-term.  Note that this is a homological spectral sequence
associated to an increasing filtration.  Since the $E^{\infty}$-term is
a nonzero free $\bZ$-module and the abutment is also a $\bZ$-module,
we conclude that the abutment is a free and nonzero $\bZ$-module.

We can rule out all potential hidden multiplicative extensions, because
$\Tor_0$ splits off from the abutment and all the classes in K{\"u}nneth
filtration~$1$ are in odd total degree.  Therefore, we have an algebra
isomorphism
\[
D_* := \pi_*((\ku \otimes_{\ko} H\bZ) \otimes_{(\MU \otimes H\bZ)}
	(\ku \otimes_{\ko} H\bZ))
	\cong \bZ[b_1]/( b_1^2) 
	\otimes \Lambda(\sigma b_1^2, \sigma b_k \mid k\ge2) \,.
\]

Second, we apply the K{\"u}nneth spectral sequence
\[
E^2 = \Tor_*^{D_*} (\bZ[b_1]/( b_1^2), \bZ[b_1]/( b_1^2))
	\Longrightarrow \pi_* \THH(\ku \otimes_{\ko} H\bZ/ \MU \otimes H\bZ) \,.
\]
The $E^2$-term of this spectral sequence, namely
\[ 
E^2 = \bZ[b_1]/( b_1^2)
	\otimes \Gamma(\sigma^2 b_1^2, \sigma^2 b_k \mid k\ge2) \,,
\]
is concentrated in even degrees, so the spectral sequence collapses at
the $E^2$-term and the abutment is also nonzero, free and concentrated in even degrees.
\end{proof}

\begin{proposition} \label{prop: equiv of gr modules}
For each $q\ge0$, the spectrum $C^q(\THH(\ku/\MU)/\THH(\ko))$ is
an even $\bE_\infty$ $\ku$-algebra whose homotopy groups are free as
a $\ku_*$-module.
\end{proposition}

\begin{proof}
By Proposition~\ref{prop: evenly free map} and induction, it suffices
to show that $\pi_* \THH(\ku/\MU)$ is free as a $\ku_*$-module and
concentrated in even degrees.  This follows from~\cite{HW22}*{Theorem~E},
but we recall the argument in our context. The K{\"u}nneth spectral sequence
\begin{align*}
E^2 &= \Tor_*^{\MU_*}(\ku_*, \ku_*)
	= \Tor_*^{\bZ[x_i \mid i\ge1]}(\bZ[u], \bZ[u]) \\
	&= \bZ[u] \otimes \Lambda(\sigma x_i \mid i\ge2)
	\Longrightarrow \pi_*(\ku \otimes_{\MU} \ku)
\end{align*}
collapses at the $E^2$-term.  Again, we can rule out all hidden
multiplicative extensions, because $\Tor_0$ splits off from the abutment
and all classes in K{\"u}nneth filtration~$1$ are in odd total degrees.
Next we apply the K{\"u}nneth spectral sequence
\begin{align*}
E^2 &= \Tor_*^{\pi_*(\ku \otimes_{\MU} \ku)}(\ku_*, \ku_*)
	= \Tor_*^{\bZ[u] \otimes \Lambda(\sigma x_i \mid i\ge2)}(\bZ[u], \bZ[u]) \\
	&= \bZ[u] \otimes \Gamma(\sigma^2 x_i \mid i\ge2)
	\Longrightarrow \pi_* \THH(\ku/\MU) \,.
\end{align*}
This $E^2$-term is also concentrated in even degrees, and is free over
$\bZ[u] = \ku_*$.  Hence $E^2 = E^\infty$, and the abutment is free
over $\ku_*$ on even degree generators.
\end{proof}

The following fact is stated without proof
in~\cite{HRW}*{Example~1.3.3}.  We follow this reference and write
$\Lalg_{B/A}$ for the algebraic cotangent complex, with homological
grading, of a homomorphism $A \to B$ of (ungraded) commutative rings.

\begin{proposition} \label{prop:koMUpolynomial}
$\ko_* \MU \cong \bZ[u, b_1^2, \bar b_k \mid k\ge2]$ is a polynomial
ring on generators in dimensions~$2$, $4$ and $2k$ for $k\ge2$.  Hence
$\Lalg_{\ko_* \MU/\MU_*}$ has $2$-complete $\Tor$-amplitude contained
in $[0, 1]$, so that $\MU \to \ko \otimes \MU$ is $2$-quasi-lci.
\end{proposition}

\begin{proof}
Since $\MU_*$ is free over~$\BP_*$, we have a ring isomorphism $\ko_*
\BP \otimes_{\BP_*} \MU_* \cong \ko_* \MU$.  Here $H_*(\ko) = \cA^\vee
\cotensor_{\cA(1)^\vee} \bF_2$, so $H_*(\ko \otimes \BP) \cong \cA^\vee
\cotensor_{\cA(1)^\vee} H_*(\BP)$ and the $E_2$-term of the Adams
spectral sequence
\[
{}^{\Ad} E_2 = \Ext_{\cA^\vee}(\bF_2, H_*(\ko \otimes \BP))
	\Longrightarrow \pi_*(\ko \otimes \BP)
\]
can be rewritten as $\Ext_{\cA(1)^\vee}(\bF_2, H_*(\BP))$.  Since
$H_*(\BP) = \bF_2[\xi_k^2 \mid k\ge1]$ is concentrated in even degrees,
its
$\cA(1)^\vee$-coaction factors as
\begin{align*}
H_*(\BP) &\overset{\nu}\longto \Lambda(\xi_1^2) \otimes H_*(\BP)
	\subset \cA(1)^\vee \otimes H_*(\BP) \\
\xi_k^2 &\longmapsto
	\begin{cases}
	1 \otimes \xi_1^2 + \xi_1^2 \otimes 1 & \text{for $k=1$,} \\
	1 \otimes \xi_k^2 & \text{for $k\ge2$.}
	\end{cases}
\end{align*}
To calculate the Adams $E_2$-term we use the Cartan--Eilenberg spectral
sequence
\[
{}^{\CE} E_2 = \Ext_{\Lambda(\xi_1^2)}(\bF_2,
	\Ext_{\Lambda(\xi_1, \xi_2)}(\bF_2, H_*(\BP)))
	\Longrightarrow \Ext_{\cA(1)^\vee}(\bF_2, H_*(\BP))
\]
(cf.~\cite{CE56}*{Theorem~XVI.6.1})
for the Hopf algebra extension $\Lambda(\xi_1^2) \to \cA(1)^\vee
\to \Lambda(\xi_1, \xi_2)$.  Its $E_2$-term
\[
{}^{\CE} E_2 \cong \Ext_{\Lambda(\xi_1^2)}(\bF_2, \bF_2[v_0, v_1]
		\otimes H_*(\BP))
	= \bF_2[v_0, v_1] \otimes \bF_2[\xi_1^4, \xi_k^2 \mid k\ge2]
\]
is concentrated in filtration degree~$0$, so ${}^{\CE} E_2 = {}^{\CE}
E_\infty$.  Hence the Adams $E_2$-term is concentrated in even stems, so
also ${}^{\Ad} E_2 = {}^{\Ad} E_\infty$.  Thus $\pi_*(\ko \otimes \BP)
\cong \bZ[u, t_1^2, \bar t_k \mid k\ge2]$, with $2$ detected by~$v_0$,
$u$ detected by~$v_1$, $t_1^2$ detected by~$\xi_1^4$ and $\bar t_k$
detected by~$\xi_k^2$.  Base change along $\BP \to \MU$ replaces the
degree~$2(2^k-1)$ generators $\bar t_k$ for $k \ge 2$ with degree~$2k$
polynomial generators $\bar b_k$ for $k \ge 2$, while $t_1^2$ is replaced
by~$b_1^2$.

Since $\MU_*$ and $\ko_* \MU$ are polynomial rings, the algebraic
cotangent complexes $\Lalg_{\MU_*/\bZ}$ and $\Lalg_{\ko_* \MU/\bZ}$
are free modules over~$\MU_*$ and $\ko_* \MU$, respectively.  By the
transitivity cofiber sequence
\[
\ko_* \MU \otimes_{\MU_*} \Lalg_{\MU_*/\bZ}
	\longto \Lalg_{\ko_* \MU/\bZ}
	\longto \Lalg_{\ko_* \MU/\MU_*}
\]
associated to the ring homomorphisms $\bZ \to \MU_* \to \ko_* \MU$,
cf.~\cite{Qui70}*{5.1}, it follows that
$\Lalg_{\ko_* \MU/\MU_*}$ has $2$-complete $\Tor$-amplitude
concentrated in homological degrees $[0,1]$.
By definition~\cite{HRW}*{Definitions~4.1.1(2)
and~4.1.6}, this means
that $\MU \to \ko \otimes \MU$ is $2$-quasi-lci.
\end{proof}

\begin{proposition} \label{prop: relation to motivic filtration}
The filtrations from Definition~\ref{def: motivic filtrations on
THH TN and TP} for $F \in \{\THH, \TC^{-},
\TP\}$ agree with the motivic filtrations considered in
\cite{HRW}*{Definition~4.2.1}.  Moreover, there
exist maps
\[
\can \,,\, \varphi \: \fil_{\mot}^{\star} \TC^{-}(\ko)
	\longto \fil_{\mot}^{\star} \TP(\ko)
\]
of $\bE_\infty$ algebras in~$\Sp_2^{\fil}$, which converge to the
canonical map and Frobenius map from~\cite{NS18}.
\end{proposition}

\begin{proof}
By Proposition~\ref{prop: evenly free map} we can apply
\cite{HRW}*{Corollary~2.2.17} to $A  = \THH(\ko)$ and $B
= \THH(\ku/\MU)$.  In this case $B^{\otimes_A q+1} = C^q(B/A)$ is even by Proposition~\ref{prop:
equiv of gr modules}, so the totalization of the double-speed Whitehead
filtration agrees with the even filtration.

The map of connective $\bE_\infty$-rings $\bS \to \ko$ is chromatically
$2$-quasi-lci in the sense of~\cite{HRW}*{Definition~4.1.11}, because
by Proposition~\ref{prop:koMUpolynomial} we know that both $\MU$
and $\ko \otimes \MU$ are even, and the map $\MU \to \ko \otimes \MU$
is $2$-quasi-lci.  Hence the even filtrations also define the motivic
filtrations, as per~\cite{HRW}*{Definition~4.2.1}, and this confirms the
first statement. 

The second statement then follows directly from
\cite{HRW}*{Theorem~1.3.6}.
\end{proof}
 
\begin{definition} \label{def:syntomic}
Motivated by Proposition~\ref{prop: relation to motivic filtration},
we define $\bE_\infty$ algebras
\begin{align*}
\fil_{\mot}^{\star} \TC(\ko) &:= \eq (\can \,,\, \varphi \:
    \fil_{\mot}^{\star} \TC^{-}(\ko) \longto \fil_{\mot}^{\star} \TP(\ko)) \\
\gr_{\mot}^* \TC(\ko) &:= \gr^*
	\left( \fil_{\mot}^{\star} \TC(\ko) \right)
\end{align*}
in $\Sp_2^{\fil}$ and $\Sp_2^{\gr}$, respectively.  By taking the
(homotopy) equalizer of $\can$ and~$\varphi$ we retain the multiplicative
structure, but additively $\fil_{\mot}^{\star} \TC(\ko)$ is also the
(homotopy) fiber of $\can - \varphi$.  In light of~\cite{HRW}*{\S5}
and~\cite{BMS19} we refer to $\gr_{\mot}^w \TC(\ko)$ and $\pi_*
\gr_{\mot}^w \TC(\ko)$ as the weight~$w$ \emph{syntomic cohomology}
spectrum and syntomic cohomology groups of~$\ko$, refer to the spectral
sequence
\[
E_2^{n,2w-n} = \pi_n\gr_{\mot}^w \TC(\ko) \Longrightarrow \pi_n\TC(\ko)
\]
as the motivic spectral sequence, and follow the same grading conventions
as in Definition~\ref{def:gradings}. This spectral sequence converges, as we now
review. 
\end{definition}

\begin{lemma}\label{convergence}
For $F \in \{\THH, \TC^{-}, \TP, \TC, F(S^{2\ell+1}_+,\THH)^{\bT}, \THH^{hC_2},\THH^{tC_2} \}$, the filtration $\fil_{\mot}^{s} F(\ko)$ is (derived) complete and exhaustive. In other words, 
\[ \lim \fil_{\mot}^{\star} F(\ko)=0\qquad  \text{ and } \qquad \colim \fil_{\mot}^{\star} F(\ko)=F(\ko) \,. \]
\end{lemma}

\begin{proof}
Since limits commute with limits, $\lim  \fil_{\mot}^{\star} F(\ko) =0$ for each 
\[F \in \{\THH, \TC^{-}, \TP, \TC, F(S^{2\ell+1}_+,\THH)^{\bT},\THH^{hC_2},\THH^{tC_2}\}\,.\]

It follows from~\cite{HRW}*{Theorems~1.3.5 and 1.3.6, Corollary~2.2.17}, Proposition~\ref{prop: evenly free map} and Proposition~\ref{prop: relation to
motivic filtration} that $\colim \fil_{\mot}^{\star} F(\ko)  =F(\ko)$
for $F \in \{\THH, \TC^{-}, \TP,\TC\}$. It therefore suffices to show that $\colim \fil_{\mot}^{\star} F(\ko)  =F(\ko)$
for each $F\in \{F(S^{2\ell+1}_+,\THH)^{\bT},\THH^{hC_2},\THH^{tC_2}\}$. Here we adapt the proof of \cite{HRW}*{Proposition 2.3.5}.
Let 
\[ 
Y^{\bullet}:=C^{\bullet}(\THH(\ku/\MU)/\THH(\ko)) 
\]
and note that the canonical map
\begin{align*} 
\colim_w  \fil_{\mot}^{w} G(\THH(\ko)) &= \colim_w  \Tot  \tau_{\ge 2w} G(Y^{\bullet}) \overset{\simeq}{\longrightarrow } \Tot G(Y^{\bullet})
\end{align*}
 is an equivalence for each $G\in \{F(S^{2\ell+1}_+,-)^{\bT}, (-)^{hC_2},(-)^{tC_2}\}$ since $\Tot$ preserves coconnectivity. 
Note also that the canonical maps 
\[ F(S^{2\ell+1}_+,\THH(\ko))^{\bT} \overset{\simeq}{\longrightarrow} \Tot (F(S^{2\ell+1}_+,Y^{\bullet})^{\bT}) \]
and $ \THH(\ko)^{hC_2} \overset{\simeq}{\longrightarrow} \Tot ((Y^{\bullet})^{hC_2})$
are equivalences since limits commute with limits. This handles the cases $F=F(S^{2\ell+1}_+,\THH)^{\bT}$ and $F= \THH^{hC_2}$. 

To see that the canonical map 
$\THH(\ko)^{tC_2} \overset{\simeq}{\longrightarrow} \Tot ((Y^{\bullet})^{tC_2})$
is an equivalence, it now suffices to show that the canonical map 
\[ \THH(\ko)_{hC_2} \overset{\simeq}{\longrightarrow} \Tot ((Y^{\bullet})_{hC_2}) \]
is an equivalence. This will follow from the fact that the fiber $I$ of the map 
\[ A:=\THH(\ko)\to \THH(\ku/\MU)\] 
is $1$-connective and that $(-)_{hC_2}$ preserves connectivity. In particular, by~\cite{MNN17}*{Proposition~2.14}, the fiber of the map 
\[ \THH(\ko)_{hC_2}  \longrightarrow \Tot_n ((Y^{\bullet})_{hC_2}) \]
can be identified with $(I^{\otimes_{A} (n+1)})_{hC_2}$, which is $(n+1)$-connective, implying the claim after passing to the limit as $n$ grows. Here we use the fact that $\Tot_n$ is a finite limit and therefore commutes with $(-)_{hC_2}$. 
\end{proof}

We now introduce the relevant type~$2$ finite coefficient spectra.
Let $V(0)= C2$ denote the cofiber of the map $2\: \bS\to \bS$ and let
$C\eta$ denote the cofiber of the Hopf map $\eta \: \Sigma \bS\to \bS$.
By \cite{DM81}*{Proposition~2.1}, there are precisely four equivalence
classes of finite spectra $X = (V(0) \otimes C\eta)/v_1$ with the property
that $H^*(X) \cong \cA(1)$ as an $\cA(1)$-module, each characterized by
the $Sq^4$-action in its mod~$2$ cohomology.

\begin{notation} \label{not: A(1) def}
Following~\cite{BEM17}*{\S1}, we write $A(1)[ij]$ with $i,j \in
\{0,1\}$ for the spectrum $(V(0) \otimes C\eta)/v_1$ where $Sq^4$
on the cohomology generator in degree~$0$ is $i$~times the generator
in degree~$4$, and $Sq^4$ on the cohomology generator in degree~$2$
is $j$~times the generator in degree~$6$.  We write $A(1)$ in place
of $A(1)[ij]$ when making statements that hold for any choice of $i,j
\in \{0,1\}$.  The identity map $A(1) \to A(1)$ has additive order~$4$,
see~Lemma~\ref{lem:d3v24zero}, and \cite{BEM17}*{Main Theorem} proves
that $A(1)$ admits a $v_2^{32}$ self-map.  Hence the $A(1)$-homotopy
$A(1)_* Y = \pi_*(A(1) \otimes Y)$ of any spectrum~$Y$ is naturally
a $\bZ/4[v_2^{32}]$-module.
\end{notation}

In view of Proposition~\ref{prop: equiv of gr modules} and the
fact that $A(1) \otimes \ku$ is even, we know that $A(1) \otimes
C^{\bullet}(\THH(\ku/\MU) / \THH(\ko))$ is even.  This motivates the
following definition.

\begin{definition} \label{def:filmotTHHkoHF2}
Let
\[
\fil_{\mot}^{\star} (A(1) \otimes \THH(\ko))
	:= \Tot \bigl( \tau_{\ge 2\star} \bigl(
        A(1) \otimes  C^{\bullet}(\THH(\ku/\MU) / \THH(\ko))
        \bigr) \bigr) \,,
\]
and write $\gr_{\mot}^* (A(1) \otimes \THH(\ko)) := \gr^* \bigl(
\fil_{\mot}^{\star} (A(1) \otimes \THH(\ko)) \bigr)$
for the associated graded spectrum.
\end{definition}

\begin{remark} \label{rem: Einfty structure}
The spectrum $A(1) \otimes \ko \simeq H\bF_2$ admits a unique $\bE_\infty$
$\ko$-algebra structure.  Hence, for each $\bE_\infty$ $\ko$-algebra~$R$
the identification
\[
A(1) \otimes R = A(1) \otimes \ko \otimes_{\ko} R
	\simeq H\bF_2 \otimes_{\ko} R
\]
lets us regard $A(1) \otimes R$ as an $\bE_\infty$ $H\bF_2$-algebra.
For example, we will identify $\THH(\ko, H\bF_2) := H\bF_2 \otimes_{\ko}
\THH(\ko)$ with $A(1) \otimes \THH(\ko)$, and may therefore consider
the latter as an $\bE_\infty$ $H\bF_2$-algebra.
\end{remark}

\begin{lemma} \label{lem: THH computation}
Let $A := A(1) \otimes \THH(\ko)$ and $B := A(1)
\otimes \THH(\ku/\MU)$.  Then there is an isomorphism
\[
 A_* \cong \Lambda(\lambda'_1, \lambda_2) \otimes \bF_2[\mu]
\]
of graded $\bF_2$-algebras, where $|\lambda'_1| = 5$, $|\lambda_2|
= 7$ and $|\mu| = 8$.  There is also an isomorphism
\[
B_* \cong \Lambda(\hat\xi_1^2) \otimes \bF_2[\mu] \otimes P
\]
of $A_*$-algebras, where $|\hat\xi_1^2| = 2$, $|\mu| = 8$ and
$P$ is a polynomial algebra with generators in even degrees.  The
$A_*$-algebra structure is determined by the map $A_* \to B_*$ sending
$\lambda'_1$ and $\lambda_2$ to zero and mapping $\mu$ to~$\mu$.
\end{lemma}

\begin{proof}
The complexification map $c \: \ko \to \ku$ and the unique $\bE_\infty$-ring map $\ku \to H\bF_2$ induce monomorphisms
\[
H_*(\ko) = \bF_2[\xi_1^4, \bar\xi_2^2, \bar\xi_k \mid k\ge 3]
\longto H_*(\ku) = \bF_2[\xi_1^2, \bar\xi_2^2, \bar\xi_k \mid k\ge 3]
	\longto \cA^\vee
\]
of $\cA^\vee$-comodule algebras.
By Milnor's construction of the $\xi_i$, the composite map $\MU \to \ku
\to H\bF_2$ induces the homomorphism $H_*(\MU) = \bF_2[b_i \mid i\ge1]
\to H_*(H\bF_2) = \cA^\vee$ given by $b_{2^i-1} \mapsto \xi_i^2$ and $b_j
\mapsto 0$ for $j \ne 2^i-1$.
Hence these formulas also hold in $H_*(\ku)$.  Let $\bar b_i = \chi b_i$
denote the conjugate classes in $H_*(\MU)$, so that $\bar b_{2^i-1}
\mapsto \bar\xi_i^2$ and $\bar b_j \mapsto 0$ for $j \ne 2^i-1$.
Standard Hochschild homology computations (cf.~\cite{MS93}*{Proposition~2})
yield
\begin{align*}
\HH_* (H_* (\ko)) &= H_* (\ko) \otimes \Lambda(\sigma\xi_1^4,
	\sigma\bar\xi_2^2, \sigma\bar\xi_k \mid k\ge 3 ) \\
\HH_* (H_* (\ku)) &= H_* (\ku) \otimes \Lambda(\sigma\xi_1^2,
	\sigma\bar\xi_2^2, \sigma\bar\xi_k \mid k\ge 3 ) \,.
\end{align*}
The usual argument for hidden extensions in the B{\"o}kstedt spectral
sequence (see e.g.~\cite{AR05}*{Theorem~6.2}) implies that
\[
(\sigma\bar\xi_3)^{2^{k-3}} = \sigma\bar\xi_k \,,
\]
and produces identifications
\begin{align*}
H_*(\THH(\ko)) & \cong H_* (\ko) \otimes \Lambda(\sigma\xi_1^4,
	\sigma\bar\xi_2^2) \otimes \bF_2[\sigma\bar\xi_3] \\
H_*(\THH(\ku)) & \cong H_* (\ku) \otimes \Lambda(\sigma\xi_1^2,
	\sigma\bar\xi_2^2) \otimes \bF_2[\sigma\bar\xi_3] \,.
\end{align*}
As noted above, $H_*(A(1) \otimes \ko) \cong \cA^\vee$.
By Remark~\ref{rem: Einfty structure} and the evident collapsing
K{\"u}nneth spectral sequence we have
\[
H_* (A(1) \otimes \ku)
	\cong H_* (H\bF_2 \otimes_{\ko} \ku)
	\cong \cA^\vee \otimes_{H_*(\ko)} H_*(\ku)
	\cong \cA^\vee \otimes \Lambda(\hat\xi_1^2) \,,
\]
where $\hat\xi_1^2 := 1 \otimes \xi_1^2 + \xi_1^2 \otimes 1$ denotes
the $\cA^\vee$-comodule primitive class, so that $A(1)_*(\ku) = \pi_*
(A(1) \otimes \ku) = \Lambda(\hat\xi_1^2)$.
We conclude that
\begin{align*}
A_* = A(1)_* \THH(\ko)
	&= \Lambda(\lambda'_1, \lambda_2) \otimes \bF_2[\mu] \\
A(1)_* \THH(\ku)
	&= \Lambda(\hat\xi_1^2, \lambda_1, \lambda_2) \otimes \bF_2[\mu] \,,
\end{align*}
with $\lambda'_1$, $\lambda_2$ and $\mu$ in $A(1)_* \THH(\ko)$ having
Hurewicz images $\sigma\xi_1^4$, $\sigma\bar\xi_2^2 + \xi_1^2 \cdot
\sigma\xi_1^4$ and $\sigma\bar\xi_3 + \xi_1 \cdot \sigma\bar\xi_2^2 +
\xi_2 \cdot \sigma\xi_1^4$, respectively, while $\lambda_1$, $\lambda_2$
and $\mu$ in $A(1)_* \THH(\ku)$ have Hurewicz images $\sigma\xi_1^2$,
$\sigma\bar\xi_2^2$ and $\sigma\bar\xi_3 + \xi_1 \cdot \sigma\bar\xi_2^2$,
cf.~\cite{AR05}*{Proposition~8.7}.  Note that $\sigma\xi_1^4 =
\sigma \xi_1^2 \cdot \xi_1^2 + \xi_1^2 \cdot \sigma\xi_1^2 = 0$ in the
$\ku$-case.  Hence the induced homomorphism $A(1)_* \THH(\ko) \to A(1)_*
\THH(\ku)$ is given by $\lambda'_1 \mapsto 0$, $\lambda_2 \mapsto
\lambda_2$ and ~$\mu \mapsto \mu$.

Next, we compute $A(1)_* \THH(\ku/\MU)$, using the fact that
\begin{equation} \label{eq:fact}
A(1) \otimes \THH(\ku/\MU)
	\simeq (A(1) \otimes \THH(\ku)) \otimes_{\THH(\MU)} \MU \,.
\end{equation}
We know that
\[
\pi_* \THH(\MU) \cong \MU_* \otimes \Lambda(\sigma\bar b_i \mid i\ge 1)
\]
by \cite{MS93}*{Remark~4.3}, cf.~\cite{Rog20}*{Proposition~4.5}.
We expand
\[
\THH(\MU) \longto \THH(\ku)\longto A(1) \otimes \THH(\ku)
\]
as the composite
\begin{align*}
\MU \otimes_{\MU \otimes \MU} \MU
	& \longto \ku \otimes_{\ku \otimes \ku}\ku \\
	& \longto A(1) \otimes \ku \otimes_{A(1) \otimes \ku \otimes \ku}
	A(1) \otimes \ku \,,
\end{align*}
where
\begin{align*}
\MU_*[\bar b_i \mid i\ge1] =
\pi_*(\MU \otimes \MU) & \longto \pi_*(\ku \otimes \ku) \\
	& \longto \pi_*(A(1) \otimes \ku \otimes \ku)
	= \Lambda(\hat\xi_1^2) \otimes H_*(\ku)
\end{align*}
takes $\bar b_{2^i-1}$ to $\bar\xi_i^2$ for $i\ge1$ and $\bar b_j$
to~$0$ for $j \ne 2^i-1$.  Hence $\pi_* \THH(\MU) \to A(1)_* \THH(\ku)$ is
given by $\sigma\bar b_1 \mapsto \sigma\xi_1^2 = \lambda_1$, $\sigma\bar
b_3 \mapsto \sigma\bar\xi_2^2 = \lambda_2$, and $\sigma\bar {b}_i \mapsto
0$ for $i \notin \{1,3\}$.  (This uses that $\sigma\bar b_{2^i-1}
\mapsto \sigma\bar\xi_i^2 = \sigma\bar\xi_i \cdot \bar\xi_i + \bar\xi_i
\cdot \sigma\bar\xi_i = 0$ for $i\ge3$, while $\xi_1$ and $\bar\xi_2$
do not exist in $H_*(\ku)$.)  The $\pi_* \THH(\MU)$-algebra structure on
$\MU_*$ is given by mapping $x_i$ to $x_i$ and mapping $\sigma\bar b_i$
trivially for all $i\ge 1$.  The K{\"u}nneth spectral sequence associated
to~\eqref{eq:fact} therefore has $E^2$-term
\[
	\Lambda(\hat\xi_1^2) \otimes \bF_2[\mu]
	\otimes \Gamma(\sigma^2 \bar b_{2^i-1} \mid i\ge3)
	\otimes \Gamma(\sigma^2 \bar b_j \mid j \ne 2^i-1) \,.
\]
This spectral sequence is concentrated in even total degrees and therefore
collapses at the $E^2$-term.
We resolve the hidden multiplicative extensions using
Steinberger's computation~\cite{BMMS86}*{III.2} of the Dyer--Lashof
operations on $H_*(H\bF_2) = \cA^\vee$ and Kochman's computation
\cite{Koc73}*{Theorem~6} of the Dyer--Lashof operations on $H_*(\BU)
\cong H_*(\MU)$, as in the proof of \cite{HW22}*{Lemma~2.4.1}.  (Note
that Steinberger's result is used for the $\sigma^2 \bar b_{2^i-1}$,
while Kochman's theorem is used for the remaining $\sigma^2\bar b_j$.)
This produces the identification
\[
B_* = A(1)_* \THH(\ku/\MU)
	\cong \Lambda(\hat\xi_1^2) \otimes \bF_2[\mu] \otimes P \,,
\]
where
\[
P := \bF_2[w_i \mid i\ge0]
	\otimes \bF_2[y_{j,i} \mid \text{$j\ge2$ even, $i\ge0$}]
\]
is a polynomial algebra with algebra generators in even degrees.
Here $w_i$ is any choice of lift of $\gamma_{2^i}(\sigma^2 \bar b_7)$
and $y_{j,i}$ is any choice of lift of $\gamma_{2^i}(\sigma^2 \bar b_j)$.
We conclude that $B_*$ is an even $\bE_\infty$-ring, and the $A_*$-algebra
structure is determined by $\lambda'_1$ and $\lambda_2$ mapping
trivially for degree reasons and $\mu$ mapping to~$\mu$ by the first
half of this proof.
\end{proof}

\begin{corollary} \label{cor: evenly free map}
The map $A(1) \otimes \THH(\ko) \to A(1) \otimes \THH(\ku/\MU)$, induced
by the complexification map $c \: \ko \to \ku$ and the unit map $\bS
\to \MU$, is evenly free.
\end{corollary}

\begin{proof}
This can be proven directly using Lemma~\ref{lem: THH computation}, but
instead we simply point out that it follows from Proposition~\ref{prop:
evenly free map} and Remark~\ref{rem: Einfty structure} by base change along
$\ko \to H\bF_2$, using the following three pushout squares of
$\bE_\infty$-rings.
\[
\xymatrix{
\ko \ar[r] \ar[d] & \THH(\ko) \ar[r] \ar[d] & \THH(\ku/\MU) \ar[d] \\
H\bF_2 \ar[r] & A(1) \otimes \THH(\ko) \ar[r] \ar[d]
	& A(1) \otimes \THH(\ku/\MU) \ar[d] \\
& C \ar[r] & \THH(\ku/\MU) \otimes_{\THH(\ko)} C
}
\]
\end{proof}

\begin{convention}
To be consistent with our implicit $2$-completion, we write
$\fil_{\ev}^{\star}$ for the functor denoted $\fil_{\ev,2}^{\star}$
in \cite{HRW}*{Construction~2.1.11}.
\end{convention}

\begin{remark} %% \label{rem: our motivic with coeff is even}
By \cite{HRW}*{Corollary~2.2.17}, Remark~\ref{rem: Einfty structure}
and Corollary~\ref{cor: evenly free map}, we can identify
\[
\fil_{\mot}^{\star} (A(1) \otimes \THH(\ko))
	\simeq \fil_{\ev}^{\star} (A(1) \otimes \THH(\ko)) \,,
\]
in the sense of Definition~\ref{def:filmotTHHkoHF2} and \cite{HRW}*{Construction~2.1.11}.
\end{remark}

\begin{theorem} \label{thm: A(1) coeff}
There is an isomorphism
\[
\pi_* \gr_{\mot}^* (A(1) \otimes \THH(\ko))
	\cong \Lambda(\lambda'_1, \lambda_2) \otimes \bF_2[\mu]
\]
of bigraded $\bF_2$-algebras, with $\|\lambda'_1\| = (5,1)$,
$\|\lambda_2\| = (7,1)$ and $\|\mu\| = (8,0)$.
\end{theorem}

\begin{proof}
We closely follow~\cite{HW22} and~\cite{HRW}.  Starting with the
proof of \cite{HW22}*{Proposition~6.1.6}, let $A := A(1) \otimes
\THH(\ko)$ and $B := A(1) \otimes \THH(\ku/\MU)$, so that $A_* \cong
\Lambda(\lambda'_1, \lambda_2) \otimes \bF_2[\mu]$ and $B_* \cong
\Lambda(\hat\xi_1^2) \otimes \bF_2[\mu] \otimes P$ by Lemma~\ref{lem:
THH computation}.  The descent spectral sequence associated to the
cosimplicial Amitsur resolution $C^{\bullet}(B/A) = B^{\otimes_A
\bullet+1}$ for $A \to B$ has $E_1$-term
\[
E^q_1(B/A) = \pi_*(B^{\otimes_A q+1})
\]
for $q\ge0$, and converges to~$A_*$.  Since $B_*$ is concentrated in even
stems, Corollary~\ref{cor: evenly free map} implies that $\Sigma :=
\pi_*(B \otimes_A B)$ is even and free over~$B_*$, so that $(B_*, \Sigma)$
is a flat Hopf algebroid.
Let $C^*_\Sigma(B_*, B_*)$ denote the associated cobar complex.
It follows by induction on~$q$ that the natural homomorphism
\begin{multline*}
C^q_{\Sigma}(B_*, B_*) = \Sigma \otimes_{B_*} \dots \otimes_{B_*} \Sigma
	\overset{\cong}\longto
\pi_*((B \otimes_A B) \otimes_B \dots \otimes_B (B \otimes_A B)) \\ 
	\cong \pi_*(B \otimes_A \dots \otimes_A B) = E^q_1(B/A)
\end{multline*}
is an isomorphism for each $q\ge0$, since the relevant K{\"u}nneth
spectral sequences collapse.
Passing to cohomology, we obtain an isomorphism
\[
\Ext_{\Sigma}^*(B_*, B_*) \cong E_2^*(B/A) \,,
\]
identifying the descent spectral sequence $E_2$-term with the Hopf
algebroid cohomology of $(B_*, \Sigma)$.  We claim that in each stem
this $E_2$-term has the same finite order as~$A_*$, so that the descent
spectral sequence for $A \to B$ must collapse at $E_2 = E_\infty$.

By convergence, the descent $E_2$-term is an upper bound for~$A_*$.
To show that the bound is exact, we consider the multiplicative Whitehead
filtrations $\tau_{\ge\star} A$ and $\tau_{\ge\star} B$ of $A$ and~$B$,
respectively.
For each $q\ge0$ we equip $B^{\otimes_A q+1}$ with the relative
convolution filtration
\[
\fil^{\star} B^{\otimes_A q+1}
	= (\tau_{\ge\star} B)^{\otimes_{(\tau_{\ge\star} A)} q+1} \,,
\]
having associated graded $\bE_\infty$-ring
\[
\gr^* B^{\otimes_A q+1}
	= (H\pi_* B)^{\otimes_{(H\pi_* A)} q+1} \,.
\]
Here $H\pi_* A$ and $H\pi_* B$ are to be interpreted as the graded
$\bE_\infty$-rings $\gr^*(\tau_{\ge\star} A)$ and $\gr^*(\tau_{\ge\star}
B)$, respectively.  We proved in Lemma~\ref{lem: THH computation} that
$A_* \to B_*$ is given by $\lambda'_1 \mapsto 0$, $\lambda_2 \mapsto 0$
and $\mu \mapsto \mu$, so that
\[
\olSigma := \pi_* (H\pi_* B \otimes_{H\pi_* A} H\pi_* B) \cong
	\Gamma(\sigma\lambda'_1, \sigma\lambda_2)
	\otimes \Lambda(\hat\xi_1^2) \otimes P
	\otimes \Lambda(\hat\xi_1^2) \otimes P
	\otimes \bF_2[\mu]
\]
is even and free over~$B_*$.  Hence $(B_*, \olSigma)$ is a flat Hopf
algebroid, and as above we have compatible isomorphisms
\[
C^q_{\olSigma}(B_*, B_*)
	= \olSigma \otimes_{B_*} \dots \otimes_{B_*} \olSigma
\overset{\cong}\longto (H\pi_* B)^{\otimes_{(H\pi_* A)} q+1}
\]
for all $q\ge0$.  Since these bigraded groups are concentrated in
even stems, and differentials reduce the stem by one, the convolution
filtration spectral sequence
\[
\pi_* ((H\pi_* B)^{\otimes_{(H\pi_* A)} q+1})
	\Longrightarrow \pi_*(B^{\otimes_A q+1})
\]
collapses at this term.  This proves that $\pi_*(B^{\otimes_A q+1}) =
E^q_1(B/A)$ has a descending filtration with associated graded given
by $C^q_{\olSigma}(B_*, B_*)$.  These filtrations are compatible for
varying $q\ge0$, so the descent $E_1$-term is a filtered differential
graded algebra with associated graded $E_1 = C^*_{\olSigma}(B_*, B_*)$.
Passing to cohomology, we obtain the May--Ravenel spectral sequence
\[
E_2 = \Ext_{\olSigma}(B_*, B_*) \Longrightarrow \Ext_\Sigma(B_*, B_*)
\]
converging to the descent $E_2$-term, cf.~\cite{Rav86}*{Theorem~A1.3.9}.

We now view the Hopf algebroid $(B_*, \olSigma)$ as the tensor product
of the three Hopf algebroids
\[
(\bF_2, \Gamma(\sigma\lambda'_1, \sigma\lambda_2))
	\quad,\quad
(\Lambda(\hat\xi_1^2) \otimes P,
	\Lambda(\hat\xi_1^2) \otimes P \otimes \Lambda(\hat\xi_1^2) \otimes P)
	\quad\text{and}\quad
(\bF_2[\mu], \bF_2[\mu]) \,.
\]
These have cohomology algebras $\Lambda(\lambda'_1, \lambda_2)$, $\bF_2$
and $\bF_2[\mu]$, respectively, with $\lambda'_1, \lambda_2 \in \Ext^1$
and $\mu \in \Ext^0$.  This confirms that the May--Ravenel $E_2$-term
\[
\Ext_{\olSigma}(B_*, B_*)
	\cong \Lambda(\lambda'_1, \lambda_2) \otimes \bF_2[\mu]
\]
has the same finite order in each stem as~$A_*$, which implies that
the May--Ravenel spectral sequence and the descent spectral sequence
both collapse at their $E_2$-terms.  Moreover, there is no room for
hidden multiplicative extensions, since $\lambda'_1$ and $\lambda_2$
both square to zero in~$A_*$.

We have now established that the descent spectral sequence
\[
E_1^q(B/A) = \pi_*(B^{\otimes_A q+1})
	\Longrightarrow A_*
\]
is concentrated in even internal degrees $n+q = 2w$, having $E_2$-term
\[
E_2(B/A) \cong \Ext_{\Sigma}(B_*, B_*) \cong
\Lambda(\lambda'_1, \lambda_2) \otimes \bF_2[\mu]
\]
with $(n,q)$-bidegrees $\|\lambda'_1\| = (5,1)$, $\|\lambda_2\| = (7,1)$
and $\|\mu\| = (8,0)$.  Following \cite{HRW}*{Example~4.2.3} we apply
\cite{HRW}*{Corollary 2.2.17(1)} to the evenly free map $A \to B$,
to see that
\[
\fil_{\ev}^{\star} A \overset{\simeq}\longto
	\Tot \bigl( \fil_{\ev}^{\star} B^{\otimes_A \bullet+1} \bigr)
	= \Tot \bigl( \tau_{\ge 2\star}(B^{\otimes_A \bullet+1}) \bigr)
\]
is an equivalence.  For each integer weight~$w$ there is a spectral
sequence converging to 
\[\pi_* \Tot \bigl( \tau_{\ge 2w}(B^{\otimes_A
\bullet+1}) \bigr) \,,\]
with $E_1$-term given by the part of the descent
spectral sequence $E_1(B/A)$ that is located in internal degrees $n+q
\ge 2w$.  The $d_1$-differential preserves this part, so the $E_2$-term
for weight~$w$ is given by the part of $E_2(B/A)$ in the same range of
internal degrees.  By naturality, this spectral sequence must collapse
at the $E_2$-term, since the full descent spectral sequence does so.
It follows that
\[
\pi_* \fil_{\ev}^w A \longto A_*
\]
maps the source isomorphically to the subgroup of classes in internal
degree $\ge 2w$, and $\pi_*\gr_{\ev}^w A$ is isomorphic to the summand
in~$A_*$ consisting of classes in internal degree $= 2w$.  Hence
\[
\pi_*\gr_{\ev}^* A \cong \Lambda(\lambda'_1, \lambda_2) \otimes \bF_2[\mu]
\]
as bigraded algebras, with $(n, 2w-n)$-bidegrees $\|\lambda'_1\| =
(5,1)$, $\|\lambda_2\| = (7,1)$ and $\|\mu\| = (8,0)$.
\end{proof}

By~\cite{HRW}*{Corollary~2.2.21(1)}, there are equivalences 
\[\fil_{\ev}^{\star}
\bS \simeq \Tot \bigl( \tau_{\ge 2\star} C^\bullet(\MU/\bS) \bigr)
\text{ and }
\gr_{\ev}^*\bS \simeq \Tot \Sigma^{2*} H\pi_{2*} C^{\bullet}(\MU/\bS) \,.
\]
The even filtration is extended
to module spectra in~\cite{HRW}*{Appendix~A}.  As such the even filtration is lax symmetric monoidal~\cite{HRW}*{Remark A.1.6} and the associated graded functor is symmetric monoidal, so $\gr_{\ev}^*
\bS$ is an $\bE_\infty$ algebra in graded $2$-complete spectra, and
$\gr_{\ev/\bS}^*$ is a lax symmetric monoidal functor from $2$-complete
spectra to $\gr_{\ev}^* \bS$-modules. 
Note that when $R$ is an $\bE_{\infty}$-ring and $\MU_*R$ is concentrated in even degrees, 
then $\fil_{\ev}^{\star}R$ agrees with  $\fil_{\ev/\bS}^{\star}R$ by~\cite{HRW}*{Corollary~1.1.6 and Remark~1.1.7}. 

\begin{convention} \label{conv: modules over gr even}
We will simply write $\otimes$ for $\otimes_{\gr_{\ev}^* \bS}$ when it
is clear from the context that we are in the category of modules over
$\gr_{\ev}^* \bS$.
\end{convention}

\begin{construction} \label{con:A1V1Ceta}
For finite spectra~$V$ with $\MU_*(V)$ concentrated in even
degrees, we shall write $\olV$ for the $\gr_{\ev}^* \bS$-module
$\gr^{*}_{\ev/\bS}V \simeq \Tot \Sigma^{2*} H\pi_{2*} (C^\bullet(\MU/\bS) \otimes V)$.
In particular, this defines the modules $\olV(0)$, $\olC\eta$ and
$\olA(1)$.
 
By~\cite{GIKR22}, we can identify $\gr_{\ev}^*\bS$ with
$C\tau$ in the $\bC$-motivic homotopy category $\SH(\bC)$
(cf.~\cite{HRW}*{Recollection~6.0.1}).  Consequently, the $\gr_{\ev}^*
\bS$-module~$\olV$ corresponds to the even $\MU_*\MU$-comodule $\MU_*(V)$
under the equivalence of \cite{GWX21}*{Theorem~1.13(2)}.  Moreover,
by \cite{GWX21}*{Remark~4.15}, this equivalence is symmetric monoidal.

Hence, for each $k\ge0$ there is an essentially unique $\bE_\infty$
$\gr_{\ev}^* \bS$-algebra $\olV(k)$ that corresponds to the commutative
$\MU_*\MU$-comodule algebra $\MU_*/(2, v_1, \dots, v_k)$.  For $k=0$
this gives the previously defined $\gr_{\ev}^* \bS$-module $\olV(0)$
an $\bE_\infty$ algebra structure, even though $V(0) = C2$ is not a
ring spectrum, and for $k\ge1$ it defines the $\bE_\infty$ algebras
$\olV(k)$, in spite of $V(k)$ not existing as a spectrum.  Furthermore,
the isomorphism
\begin{equation} \label{eq:MUA1}
	\MU_*(A(1)) \cong \MU_*/(2,v_1) \otimes_{\MU_*} \MU_*(C\eta)
\end{equation}
of $\MU_*\MU$-comodules exhibits $\MU_*(A(1))$ as an
$\MU_*/(2,v_1)$-module in that category.  It follows that we have an
equivalence $\olA(1) \simeq \olV(1) \otimes \olC\eta$, exhibiting
$\olA(1)$ as a $\olV(1)$-module in the category of $\gr_{\ev}^*
\bS$-modules.  
Here we use the fact
that $\eta \in \pi_1 \bS$ is detected in Novikov filtration~$1$, so
the cofiber of $\eta \: \Sigma^{1,1} \gr_{\ev}^* \bS \to \gr_{\ev}^* \bS$
is equivalent to $\olC\eta$ as defined above.
Hence we have a cofiber sequence
\begin{equation} \label{eq:V1-A1-cofibseq}
\Sigma^{1,1} \olV(1) \overset{\eta}\longto \olV(1)
	\overset{i}\longto \olA(1)
	\overset{j}\longto \Sigma^{2,0} \olV(1)
\end{equation}
of $\olV(1)$-modules, mapping to the cofiber sequence
\begin{equation} \label{eq:V2-V2Ceta-cofibseq}
\Sigma^{1,1} \olV(2) \overset{\eta}\longto \olV(2)
        \overset{i}\longto \olV(2) \otimes \olC\eta
        \overset{j}\longto \Sigma^{2,0} \olV(2)
\end{equation}
of $\olV(2)$-modules.

When $\olV$ and $M^*$ are $\gr_{\ev}^* \bS$-modules, we shall write
\[
\olV_* M^* := \pi_* (\olV \otimes M^*)
\]
for the homotopy groups of the graded spectrum $\olV \otimes M^*$, keeping
in mind that this is a bigraded abelian group.  Note that we are applying
Convention~\ref{conv: modules over gr even} throughout this construction.
\end{construction}

\begin{lemma} \label{lem: not a ring}
The $\MU_*\MU$-comodules $\MU_*(C\eta)$ and $\MU_*(A(1))$ do not admit the structure of
$\MU_*\MU$-comodule algebras. 
\end{lemma}

\begin{proof}
In view of~\eqref{eq:MUA1}, it suffices to prove this for~$\MU_*(A(1))$.
Writing $\MU_*(C\eta) = \MU_*\{1, b_1\}$, so that $\MU_*(A(1)) \cong
\MU_*/(2,v_1) \{1, b_1\}$, the coaction
\[
\nu \: \MU_*(A(1)) \longto \MU_*\MU \otimes_{\MU_*} \MU_*(A(1))
\] 
satisfies $\nu(b_1) = b_1 \otimes 1 + 1 \otimes b_1$ and $\nu(x_2) = b_2 \otimes
1 + 1 \otimes x_2$, see e.g.~\cite{Rog20}*{(2.1)}. If $\MU_*(A(1))$
were an $\MU_*\MU$-comodule algebra, we would have $\nu(b_1^2) = (b_1
\otimes 1 + 1 \otimes b_1)^2 = b_1^2 \otimes 1 +
1 \otimes b_1^2$.  This contradicts each of the possibilities $b_1^2 = 0$
or $b_1^2 = x_2$, and there are no other alternatives. 
\end{proof}

\begin{lemma} \label{lem: identification}
There is an equivalence
\[
\olA(1) \otimes \gr_{\mot}^* \THH(\ko)
	\simeq \gr_{\mot}^* (A(1) \otimes \THH(\ko))
\]
of $\gr_{\ev}^* \bS$-modules. 
\end{lemma}
 
\begin{proof}
Let $v_1 \: \Sigma^2 V(0) \otimes C\eta \to V(0) \otimes C\eta$ be one
of the (eight) $v_1$-maps with cofiber one of the four spectra~$A(1)$.
Multiplication by~$2$ acts injectively on $\MU_*(\bS)$ and on $\MU_*(\ku)
\cong \ku_* [b_k \mid k\ge1]$, so the natural map $\olV(0) \otimes
\gr_{\ev}^* \ku \to \gr_{\ev/\bS}^*( V(0) \otimes \ku)$ is an equivalence by
\cite{GIKR22}*{Proposition~3.18}.  Since $\MU_*(C\eta)$ is a free $\MU_*$-module
on two generators, the maps $\olV(0) \otimes \olC\eta \to \gr_{\ev/\bS}^*(V(0)
\otimes C\eta)$ and $\olC\eta \otimes \gr_{\ev/\bS}^*(V(0) \otimes \ku)
\to \gr_{\ev/\bS}^*(V(0) \otimes C\eta \otimes \ku)$ are equivalences by
\cite{Pst23}*{Lemma~4.24}.  Thus the lower horizontal arrow in the following
commutative square is also an equivalence.
\[
\xymatrix{
\olV(0) \otimes \olC\eta \otimes \gr_{\ev}^* \ku
	\ar[r]^-{\simeq} \ar[d]^-{\simeq}
    & \olC\eta \otimes \gr_{\ev/\bS}^* (V(0) \otimes \ku)
	\ar[d]^-{\simeq} \\
\gr_{\ev/\bS}^*(V(0) \otimes C\eta) \otimes \gr_{\ev}^* \ku
	\ar[r]^-{\simeq}
    & \gr_{\ev/\bS}^* (V(0) \otimes C\eta \otimes \ku)
}
\]
Since $v_1$ acts injectively on $\MU_*(V(0) \otimes C\eta)$ and
on $\MU_*(V(0) \otimes C\eta \otimes \ku)$ it follows, by \cite{GIKR22}*{Proposition~3.18} again,  that the columns in the following diagram are
cofiber sequences.

\[
\begin{tikzcd}
\gr_{\ev/\bS}^*(\Sigma^2 V(0) \otimes C\eta) \otimes \gr_{\ev}^* \ku \arrow{r}{\simeq}  \arrow{d}{v_1}  &  \gr_{\ev/\bS}^*(\Sigma^2 V(0) \otimes C\eta \otimes \ku) \arrow{d}{v_1} \\ 
 \gr_{\ev/\bS}^*(V(0) \otimes C\eta) \otimes \gr_{\ev}^* \ku  \arrow{r}{\simeq}  \arrow{d}  & \gr_{\ev/\bS}^*(V(0) \otimes C\eta \otimes \ku) \arrow{d}  \\ 
\olA(1) \otimes \gr_{\ev}^* \ku \arrow{r}{\simeq}   & \gr_{\ev}^*(A(1) \otimes \ku)
\end{tikzcd}
\]
%\[
%\begin{tikzcd}[column sep=small]
%\gr_{\ev/\bS}^*(\Sigma^2 V(0) \otimes C\eta) \otimes \gr_{\ev}^* \ku
%	\arrow{r}{v_1}  \arrow{d}{\simeq}
%	& \gr_{\ev/\bS}^*(V(0) \otimes C\eta) \otimes \gr_{\ev}^* \ku
%\arrow{r} \arrow{d}{\simeq}
%	& \olA(1) \otimes \gr_{\ev}^* \ku \arrow{d}{\simeq} \\
%\gr_{\ev/\bS}^*(\Sigma^2 V(0) \otimes C\eta \otimes \ku)
%	\arrow{r}{v_1}
%& \gr_{\ev/\bS}^*(V(0) \otimes C\eta \otimes \ku) \ar[r]
%	& \gr_{\ev}^*(A(1) \otimes \ku)
%\end{tikzcd}
%\]
Hence the lower horizontal arrow is also an equivalence.

Using Proposition~\ref{prop: equiv of gr modules},
it follows that there are equivalences
\[
\olA(1) \otimes \gr_{\ev}^* C^q(\THH(\ku/\MU)/\THH(\ko)) 
	\overset{\simeq}\longto
\gr_{\ev}^* \bigl( A(1) \otimes C^q(\THH(\ku/\MU)/\THH(\ko)) \bigr)
\]
for all $q\ge0$, compatible with the cosimplicial structure maps.
Passing to totalizations, and using that $\olA(1)$ is
a finite $\gr_{\ev}^* \bS$-module, we obtain
\begin{multline*}
\olA(1) \otimes \Tot \bigl(
	\gr_{\ev}^* C^\bullet(\THH(\ku/\MU)/\THH(\ko)) \bigr) \\
\simeq \Tot \bigl(
	\olA(1) \otimes \gr_{\ev}^* C^\bullet(\THH(\ku/\MU)/\THH(\ko)) \bigr) \\
\simeq \Tot \gr_{\ev}^* \bigl(
	A(1) \otimes C^\bullet(\THH(\ku/\MU)/\THH(\ko)) \bigr) \,.
\end{multline*}
In view of Definitions~\ref{def: motivic filtrations on THH TN and TP}
and~\ref{def:filmotTHHkoHF2}, this establishes the asserted equivalence.
\end{proof}

\begin{remark} \label{rem: A(1) mult}
A consequence of Lemma~\ref{lem: not a ring} is that $\olC\eta$
and $\olA(1)$ are not $\bE_1$ algebras in $\gr_{\ev}^* \bS$-modules.
However, by Lemma~\ref{lem: identification}, there is an identification
of $\gr_{\ev}^*\bS$-modules
\[
\olA(1) \otimes \gr_{\mot}^* \THH(\ko)
	\simeq \gr_{\mot}^* (A(1) \otimes \THH(\ko)) \,,
\]
where the right-hand side is an $\bE_\infty$ $\gr_{\ev}^*\bS$-algebra.
We therefore use this to equip the left-hand side with an $\bE_\infty$
$\gr_{\ev}^*\bS$-algebra structure.  Note that the left-hand side also
has a canonical action of the circle~$\bT$, but this $\bT$-action is
not an action through $\bE_\infty$-ring maps, because the right-hand side
is not equipped with a compatible $\bT$-action.  See Remark~\ref{rem:
lack of Leibniz rule} for an algebraic incarnation of this.
\end{remark}

\begin{lemma} \label{lem:etaBockspseq}
Let $M^*$ be a $\gr_{\ev}^* \ko$-module.  Then $\olC\eta \otimes M^*
\simeq \gr_{\ev}^* \ku \otimes_{\gr_{\ev}^* \ko} M^*$ and there is a
natural, trigraded, $\eta$-Bockstein spectral sequence
\[
E_1 = (\olC\eta_* M^*) \, [\eta]
	\Longrightarrow \pi_*(M^*) \,.
\]
If $M^*$ is uniformly bounded below, then this spectral sequence is
conditionally convergent.  If $M^*$ is a $\gr_{\ev}^* \ko$-algebra,
then this is an algebra spectral sequence.
\end{lemma}

\begin{proof}
The Wood cofiber sequence~\eqref{eq:Wood} induces a cofiber sequence
\[
\Sigma^{1,1} \gr_{\ev}^* \ko
	\overset{\eta}\longto \gr_{\ev}^* \ko
	\overset{c}\longto \gr_{\ev}^* \ku
	\overset{R}\longto \Sigma^{2,0} \gr_{\ev}^* \ko
\]
in the category of $\gr_{\ev}^* \ko$-modules, where $c$
is a map of $\bE_\infty$ algebras.  This follows as in
\cite{GIKR22}*{Proposition~3.18}, since $\MU_*(\ko)$ is concentrated in
even degrees by Proposition~\ref{prop:koMUpolynomial}.  Hence $\olC\eta
\otimes M^*$ and $\gr_{\ev}^* \ku \otimes_{\gr_{\ev}^* \ko} M^*$ are
both the cofiber of $\eta \: \Sigma^{1,1} M^* \to M^*$.

The Bousfield--Kan homotopy (= descent) spectral sequence for the
cosimplicial $\bE_\infty$ $\gr_{\ev}^* \ko$-algebra
\[
C^\bullet( \gr_{\ev}^* \ku / \gr_{\ev}^* \ko )
\]
is well-known to be multiplicative, and converges conditionally (and
strongly) to $\pi_* \gr_{\ev}^* \ko$.  The normalized $\Tot$-filtration
is the same as the $\eta$-adic tower
\[
\dots \overset{\eta}\longto \Sigma^{2,2} \gr_{\ev}^* \ko
	 \overset{\eta}\longto \Sigma^{1,1} \gr_{\ev}^* \ko
	 \overset{\eta}\longto \gr_{\ev}^* \ko \,,
\]
so we can equally well call this the $\eta$-Bockstein spectral sequence.
In particular, its $E_1$-term is
\[
E_1^q = \Sigma^{q,q} \pi_* \gr_{\ev}^* \ku
	\cong (\pi_* \gr_{\ev}^* \ku) \{\eta^q\}
\]
for each $q\ge0$, and is concentrated in even internal
degrees (= integer weights).

Tensoring over~$\gr_{\ev}^* \ko$ with~$M^*$, the Bousfield--Kan spectral
sequence for the cosimplicial $\gr_{\ev}^* \ko$-module
\[
C^\bullet( \gr_{\ev}^* \ku / \gr_{\ev}^* \ko )
	\otimes_{\gr_{\ev}^* \ko} M^*
\]
has abutment~$\pi_* M^*$, and is multiplicative if~$M^*$
is a $\gr_{\ev}^* \ko$-algebra.  The normalized $\Tot$-filtration
is the same as the $\eta$-adic tower
\[
\dots \overset{\eta}\longto \Sigma^{2,2} M^*
	 \overset{\eta}\longto \Sigma^{1,1} M^*
	 \overset{\eta}\longto M^* \,,
\]
and the $E_1$-term is
\[
E_1^* \cong \pi_* (\olC\eta \otimes M^*) \, [\eta] \,.
\]
If $M^*$ is uniformly bounded below, then its $\eta$-adic tower
has trivial (homotopy) limit, which ensures conditional convergence.
\end{proof}

\begin{example}
The $\eta$-Bockstein spectral sequence
\[
E_1 = (\pi_* \gr_{\ev}^* \ku) \, [\eta] \Longrightarrow \pi_* \gr_{\ev}^* \ko
\]
has $E_1 = \bZ[\eta, u]$, $d_1(u) = 2 \eta$ and $E_2 = E_\infty =
\bZ[\eta, u^2]/(2\eta) \cong \pi_* \gr_{\ev}^* \ko$.  The motivic
(= Novikov) spectral sequence
\[
E_2 = \pi_* \gr_{\ev}^* \ko 
	\Longrightarrow \pi_* \ko
\]
has $E_2 = \bZ[\eta, u^2]/(2\eta)$ with $\|\eta\| = (1,1)$, $\|u^2\|
= (4,0)$, $d_3(u^2) = \eta^3$ and $E_4 = E_\infty = (\bZ\{1, 2u^2\}
\oplus \bZ/2\{\eta, \eta^2\}) \otimes \bZ[u^4]$.  Here $A \in \pi_4(\ko)$
and~$B \in \pi_8(\ko)$ are detected by $2 u^2$ and~$u^4$, respectively.
The bigraded homotopy rings of the $\bE_\infty$ $\gr_{\ev}^* \ko$-algebras
\[
\xymatrix{
\gr_{\ev}^* \ko \ar[r]^-{i} \ar[d]^-{i_0}
	& \olC\eta \otimes \gr_{\ev}^* \ko \simeq
	\gr_{\ev}^* \ku \ar[d]^-{i_0} \\
\olV(0) \otimes \gr_{\ev}^* \ko \ar[r]^-{i} \ar[d]^-{i_1}
	& \olV(0) \otimes \olC\eta \otimes \gr_{\ev}^* \ko \simeq
	\olV(0) \otimes \gr_{\ev}^* \ku \ar[d]^-{i_1} \\
\olV(1) \otimes \gr_{\ev}^* \ko \ar[r]^-{i}
	& \olA(1) \otimes \gr_{\ev}^* \ko \simeq
	\olV(1) \otimes \gr_{\ev}^* \ku
}
\]
are thus
\[
\xymatrix@C+6pc{
\bZ[\eta, u^2]/(2\eta) \ar[r] \ar[d]
	& \bZ[u] \ar[d] \\
\bZ/2[\eta, u] \ar[r] \ar[d]
	& \bZ/2[u] \ar[d] \\
\bZ/2[\eta] \ar[r]
	& \bZ/2 \rlap{\,.}
}
\]
\end{example}

\begin{definition} \label{def:epsilon2}
Let $v_2 \: \Sigma^{6,0} \olV(1) \to \olV(1)$ be the
$\gr_{\ev}^* \bS$-module map corresponding to the
$\MU_*\MU$-comodule homomorphism
$v_2 \: \Sigma^6 \MU_*/(2,v_1) \to \MU_*/(2,v_1)$, so that
there is a cofiber sequence
\[
\Sigma^{6,0} \olV(1) \overset{v_2}\longto \olV(1)
	\overset{i_2}\longto \olV(2)
	\overset{j_2}\longto \Sigma^{7,-1} \olV(1)
\]
of $\gr_{\ev}^* \bS$-modules.
The induced map
\[
v_2 \: \Sigma^{6,0} \olV(1) \otimes \gr_{\ev}^* \ko
	\longto \olV(1) \otimes \gr_{\ev}^* \ko
\]
of $\gr_{\ev}^* \ko$-modules is null-homotopic, and there is a unique
class
\[
\varepsilon_2 \in \olV(2)_* \gr_{\ev}^* \ko
\]
in bidegree~$\|\varepsilon_2\| = (7,-1)$ with $j_2(\varepsilon_2) =
\Sigma^{7,-1} 1$.  We have $\varepsilon_2^2 = 0$, since the group in
bidegree~$(14,-2)$ is trivial.  Then $\olV(2)_* \gr_{\ev}^* \ko \cong
\Lambda(\varepsilon_2) \otimes \olV(1)_* \gr_{\ev}^* \ko$, and in general
we have a natural algebra isomorphism
\[
\Lambda(\varepsilon_2) \otimes \olV(1)_* M^* \cong \olV(2)_* M^*
\]
for any $\gr_{\ev}^* \ko$-algebra~$M^*$.
\end{definition}

\begin{corollary} \label{cor: A(1) coeff}
There are preferred isomorphisms of bigraded $\bF_2$-algebras
\[
\olA(1)_* \gr_{\mot}^* \THH(\ko)
	\cong \Lambda(\lambda'_1, \lambda_2) \otimes \bF_2[\mu]
\]
and
\[
(\olV(2) \otimes \olC\eta )_* \gr_{\mot}^* \THH(\ko)
	\cong \Lambda(\varepsilon_2, \lambda'_1, \lambda_2)
		\otimes \bF_2[\mu] \,.
\]
\end{corollary}

\begin{proof}
The first isomorphism is a direct consequence of Theorem~\ref{thm: A(1)
coeff}, Lemma~\ref{lem: identification} and Remark~\ref{rem: A(1) mult}.
The second isomorphism arises as in Definition~\ref{def:epsilon2}
with $M^* = \olC\eta \otimes \gr_{\mot}^* \THH(\ko)$.
\end{proof}

In the following proposition, we evaluate the spectral sequence from
Lemma~\ref{lem:etaBockspseq} in the case when $M^* = \olV(1) \otimes
\gr_{\mot}^*\THH(\ko)$ and $\olC\eta \otimes M^* \cong \olA(1) \otimes
\gr_{\mot}^* \THH(\ko)$.  See Figure~\ref{fig:etaBockTHHko} for an
illustration.  The authors thank an anonymous referee for pointing out
how to remove an indeterminacy in the following statement, which we had
previously only resolved later in the paper, in a more indirect manner.

\begin{proposition} \label{prop: eta BSS computatation}
The $\eta$-Bockstein spectral sequence
\[
E_1 = \olA(1)_* \gr_{\mot}^* \THH(\ko) \, [\eta]
	\Longrightarrow \olV(1)_* \gr_{\mot}^* \THH(\ko)
\]
has differentials
\begin{align*}
	d_1 (\lambda_2) &= \eta \lambda'_1 \\
	d_3 (\lambda'_1\lambda_2) &= \eta^3 \mu
\end{align*}
and no further differentials besides those generated by the Leibniz rule.
Moreover, there is no room for $\eta$-extensions.
Consequently, we can identify
\[
\olV(1)_* \gr_{\mot}^* \THH(\ko) \cong
 \frac{
\bF_2[\eta, \lambda'_1, \mu] }{ (\eta\lambda'_1,
	(\lambda'_1)^2 = \eta^2 \mu) }
\]
as a bigraded $\bF_2$-algebra, where $\|\eta\| = (1,1)$, $\|\mu\| =
(8,0)$ and $\|\lambda'_1\| = (5,1)$.
\end{proposition}

\begin{proof}
We deduce these differentials using a small part of the known (implicitly
$2$-complete) computation of $\pi_*\THH(\ko)$ from~\cite{AHL10}*{\S7}.
The unit $\ko \to\THH(\ko)$ and augmentation $\epsilon \:\THH(\ko)
\to \ko$ exhibit~$\ko$ as a retract of $\THH(\ko)$ in the category of
$\bE_\infty$-rings.  We write $\THH(\ko)/\ko$ for the complementary
summand in $\ko$-modules.  In degrees $* < 12$ we have 
\[H_*(\ko)
\{\sigma\bar\xi_1^4, \sigma\bar\xi_2^2, \sigma\bar\xi_3\} \cong
H_*(\THH(\ko)/\ko)\,,\] 
so there is an $11$-connected map $\Sigma^5 \ksp
\simeq \ko \otimes (S^5 \cup_\eta e^7 \cup_2 e^8) \to\THH(\ko)/\ko$.
By \cite{AHL10}*{Corollary~7.3, Figure 5}, the $\eta^2$-multiple in $\pi_6
\ksp \cong \pi_{11} \Sigma^5 \ksp$ maps to zero in $\THH$, so
\[\pi_*(\THH(\ko)/\ko) \cong (\bZ, 0, 0, 0, \bZ, \bZ/2, 0)\] 
for $5 \le * \le 11$.

We consider the $\eta$-Bockstein spectral sequence
\[
E_1 = \Lambda(\lambda'_1, \lambda_2) \otimes \bF_2[\eta, \mu]
	\Longrightarrow \olV(1)_* \gr_{\mot}^* \THH(\ko)
\]
with $\|\lambda'_1\| = (5,1)$, $\|\lambda_2\| = (7,1)$,
$\|\mu\| = (8,0)$ and $\|\eta\| = (1,1)$,
the $v_1$-Bockstein spectral sequence
\[
E_1 = \olV(1)_* \gr_{\mot}^* \THH(\ko) \, [v_1]
	\to \olV(0)_*\gr_{\mot}^* \THH(\ko)
\]
with $\|v_1\| = (2,0)$,
the $v_0$-Bockstein spectral sequence
\[
E_1 = \olV(0)_* \gr_{\mot}^* \THH(\ko) \, [v_0]
	\Longrightarrow \pi_*\gr_{\mot}^* \THH(\ko)
\]
with $\|v_0\| = (0,0)$, and the motivic spectral sequence
\[
E_2 = \pi_*\gr_{\mot}^* \THH(\ko)
	\Longrightarrow \pi_*\THH(\ko) \,.
\]
In each case the spectral sequence for~$\ko$ splits off as a direct
summand.  Taking this into account, there is no possible source or
target for a differential affecting $\lambda'_1$ in any of these
spectral sequences.  Hence $\lambda'_1$ survives in bidegree $(5,1)$
to detect the generator of $\pi_5(\THH(\ko)/\ko) \cong \bZ$.
Since $\pi_6(\THH(\ko)/\ko) = 0$, it follows that $\eta \lambda'_1$
in bidegree $(6,2)$ is an infinite cycle that detects zero, i.e., a
boundary in one of these spectral sequences.  Since $\eta \lambda'_1$
is not a $v_1$- or $v_0$-multiple, it cannot be a $v_1$-Bockstein or
$v_0$-Bockstein boundary.  Since the motivic $E_2$-term is readily seen
to be zero in bidegree $(7,0)$, it can also not be a motivic boundary.
Hence $d_1(\lambda_2) = \eta \lambda'_1$ in the $\eta$-Bockstein spectral
sequence is the only remaining possibility. 

By construction, this $d_1$-differential can
be rewritten as $i(j(\lambda_2)) = \Sigma^{2,0} \lambda'_1$, with~$i$
and~$j$ induced by the maps in~\eqref{eq:V1-A1-cofibseq}.  Here
\[
i \: \olV(1)_* \gr_{\mot}^* \THH(\ko)
	\longto \olA(1)_* \gr_{\mot}^* \THH(\ko)
\]
must be injective (in fact bijective) in bidegree~$(5,1)$, since there
are no nonzero $\eta$-multiples there, so $\Sigma^{-2,0} j(\lambda_2)$
in its source is the unique lift over~$i$ of~$\lambda'_1$, which we also
denote by~$\lambda'_1$.  We shall return to this insight at the end of the
proof.

There is no room for other $\eta$-Bockstein $d_1$-differentials, so
the next differential to be determined is $d_3(\lambda'_1 \lambda_2)
\in \bF_2\{\eta^3 \mu\}$.  On one hand, if $d_3(\lambda'_1 \lambda_2)
= \eta^3 \mu$ then the $\eta$-Bockstein $E_\infty$-term (modulo the
summand for~$\ko$) will be
\[
\bF_2\{\lambda'_1, \mu, \eta \mu, \eta^2 \mu\}
\]
in stems $\le 12$.  On the other hand, if $d_3(\lambda'_1 \lambda_2) =
0$ then it will be
\[
\bF_2\{\lambda'_1, \mu, \eta \mu, \eta^2 \mu, \eta^3 \mu\}
\]
in stems $\le 11$, with the $12$-stem concentrated in motivic filtrations
$\ge 2$.  In either case this determines $\olV(1)_* \gr_{\mot}^* \THH(\ko)$
in these stems.

The first nonzero $v_1$-Bockstein differential is $d_1(\mu) =
v_1 \lambda'_1$.  If it were not there, then $v_1 \lambda'_1$ would
survive to $\olV(0)_*\gr_{\mot}^* \THH(\ko)$ and $\pi_*\gr_{\mot}^*
\THH(\ko)$ to detect a nonzero class in the trivial group $\pi_7(\THH(\ko)/\ko)$, which
is impossible.  There is no room for other $v_1$-Bockstein differentials
affecting stems $\le 11$, so if $d_3(\lambda'_1 \lambda_2) = \eta^3 \mu$
then the $v_1$-Bockstein $E_\infty$-term (modulo the summand for~$\ko$)
will be
\[
\bF_2\{\lambda'_1, \eta \mu, \eta^2 \mu, v_1 \eta \mu\}
\]
in stems $\le 11$, while if $d_3(\lambda'_1 \lambda_2) = 0$ then it
will be
\[
\bF_2\{\lambda'_1, \eta \mu, \eta^2 \mu, \eta^3 \mu, v_1 \eta \mu\}
\]
in these stems.  In either case the $12$-stem is concentrated
in motivic filtrations $\ge 2$, and these expressions determine
$\olV(0)_*\gr_{\mot}^* \THH(\ko)$ in this range of stems.
 
In the $v_0$-Bockstein spectral sequence, there is no room for
differentials on ($\lambda'_1$ and) $\eta \mu$.  Multiplying by $\eta^2$,
it follows that $\eta^3 \mu$ is an infinite cycle (but possibly zero).
Since it is not a $v_0$-multiple, it cannot be a $v_0$-Bockstein boundary,
and since it is in motivic filtration~$3$, and the motivic $E_2$-term is
now known to be zero in bidegrees $(12,0)$ and $(12,1)$, it cannot be a
motivic $d_r$-boundary for $r\ge2$.  Hence if $d_3(\lambda'_1 \lambda_2)$
were zero, then $\eta^3 \mu$ would survive to $\olV(0)_*\gr_{\mot}^*
\THH(\ko)$ and $\pi_*\gr_{\mot}^* \THH(\ko)$ to detect a nonzero class
in $\pi_{11}(\THH(\ko)/\ko) = 0$, which is impossible.

This contradiction shows that $d_3(\lambda'_1 \lambda_2) = \eta^3
\mu$, as claimed.  This leaves the $\eta$-Bockstein $E_4$-term
\[ 
\frac{ \Lambda(\lambda'_1) \otimes \bF_2[\eta, \mu] }
	{ (\eta \lambda'_1, \eta^3 \mu) } \,.
\]  
There is no room for further differentials, so this is also
the $E_\infty$-term.  The only multiplicative ambiguity in the abutment $\olV(1)_* \gr_{\mot}^*
\THH(\ko)$ is whether $(\lambda'_1)^2$ is~$0$
or $\eta^2 \mu$.  To resolve this indeterminacy, observe that
\[
j \: \olA(1)_* \gr_{\mot}^* \THH(\ko) \longto
	\Sigma^{2,0} \olV(1)_* \gr_{\mot}^* \THH(\ko)
\]
is a map of $\olV(1)_* \gr_{\mot}^* \THH(\ko)$-modules, so
$j(\lambda_2) = \Sigma^{2,0} \lambda'_1$ implies that
\[
j(\lambda'_1 \lambda_2) = \Sigma^{2,0} (\lambda'_1)^2 \,.
\]
If $(\lambda'_1)^2$ were~$0$, then $\lambda'_1 \lambda_2$ would lift 
over~$i$ to a class in $\olV(1)_* \gr_{\mot}^* \THH(\ko)$, but this cannot
occur because it would contradict the nonzero differential $d_3(\lambda'_1
\lambda_2) = \eta^3 \mu$.  Hence $(\lambda'_1)^2 = \eta^2 \mu$.
\end{proof}

\begin{figure}
\centering
\resizebox{\textwidth}{!}{ \input{etaBockTHHko.inp} }
\caption{
 $\eta$-Bockstein $E_{1}\Longrightarrow \olV(1)_{*}\gr_{\mot}^{*}\THH(\ko)$
	\label{fig:etaBockTHHko}
}
\end{figure}

\begin{corollary} \label{cor: V2 homotopy}
We can identify
\[
\olV(2)_* \gr_{\mot}^* \THH(\ko) \cong
	\Lambda(\varepsilon_2) \otimes
 \frac{
\bF_2[\eta, \lambda'_1, \mu] }{ (\eta\lambda'_1,
	(\lambda'_1)^2 = \eta^2 \mu ) }
\]
as a bigraded $\bF_2$-algebra. 
\end{corollary}

\begin{proof}
Here $\varepsilon_2$ is chosen as in Definition~\ref{def:epsilon2}
with $M^* = \gr_{\mot}^* \THH(\ko)$.
\end{proof}
 
\section{Filtering the motivic associated graded}\label{sec:filtering}
In subsequent sections, we will use various filtrations on the associated graded of the motivic filtration on topological negative cyclic homology, topological periodic cyclic homology, and the $C_2$-homotopy fixed points and $C_2$-Tate construction on topological Hochschild homology. Here we discuss two filtrations that one might consider and show in what sense they agree. 
  
We first discuss a collection of increasing filtrations $\fil_{\star}^{\GM}(-)$. Here we use (subscripted) homological indexing. 
Recall that the circle group~$\bT$ acts freely on $S^1 \subset S^3 \subset \dots \subset S^\infty = E\bT$ and $C_2$ acts freely on $S^0 \subset S^1 \subset \dots \subset S^{\infty}=EC_2$. 
\begin{definition}\label{GM}
Define a filtered $\bT$-spectrum $\widetilde{E\bT}_{\star}$ by $\widetilde{E\bT}_{2k}= \widetilde{E\bT}_{2k+1}= S^{k\bC}$, and a filtered $C_2$-spectrum $(\widetilde{EC_2})_{\star}$ by  $(\widetilde{EC_2})_k=S^{k\bR}$, for all integers $k$. Following Greenlees–May~\cite{GM95}, for each spectrum $X$ with $\bT$-action we define a filtered spectrum 
\[ \GM_{\star}^{\bT}(X):= [ \widetilde{E\bT}_{\star}\otimes F(E\bT_+,X) ]^{\bT}\]
and for each spectrum $X$ with $C_2$-action we define a filtered spectrum 
\[ \GM_{\star}^{C_2}(X):=  [(\widetilde{EC_2})_{\star}\otimes F(EC_{2 +},X)]^{C_2} \,.\]
In filtration $\star = 0$ these specialize to the homotopy fixed point
spectra $ \GM_{0}^{\bT}(X)= X^{h\bT}$ and $\GM_{0}^{C_2}(X)= X^{hC_2}$. 
More generally, we have 
\[ 
\GM_{-2k}^{\bT}(X)=\GM_{-2k+1}^{\bT}(X)\simeq F(E\bT/S^{2k-1},X)^{\bT}
\]
and $\GM_{-k}^{C_2}(X)\simeq F(EC_2/S^{k-1},X)^{C_2}$ for all $k\ge 0$, cf.~\cite{Gre87}*{p.~437} or~\cite{HR24}*{Lemma~6.32}. 

For brevity, let  
\[ Y^{\bullet}=C^{\bullet}(\THH(\ku/\MU)/\THH(\ko)) \,.\] 
This is naturally an $\bE_\infty$ algebra in $C^\bullet(\MU/\bS)$-modules with $\bT$-action.
Define a filtered graded spectrum 
\[
\fil_{\star}^{\GM} \gr_{\mot}^{\ast}\TP(\ko) :=   \Tot \Sigma^{2\ast}H\pi_{2\ast } \GM_{\star}^{\bT}(Y^{\bullet})   \,,
\]
and define $\fil_{\star}^{\GM} \gr_{\mot}^{\ast}\TC^{-}(\ko)$ and $\fil_{\star}^{\GM} \gr_{\mot}^{\ast}F(S^{2\ell+1}_+,\THH(\ko))^{\bT}$ for $\ell\ge 0$ by truncating this filtration to live in the ranges $\star \le 0$ and $-2\ell \le \star \le 0$,
respectively. Similarly, define a filtered graded spectrum 
\[
\fil_{\star}^{\GM} \gr_{\mot}^{\ast}\THH(\ko)^{tC_2} := \Tot  \Sigma^{2\ast}H\pi_{2\ast} \GM_{\star}^{C_2}(Y^{\bullet})  \,,
\]
and define $\fil_{\star}^{\GM} \gr_{\mot}^{\ast}\THH(\ko)^{hC_2}$ by truncating this filtration to the range $\star \le 0$.

We know from
\cite{BBLNR14}*{Proposition~3.8} (cf.~\cite{NS18}*{Lemma~II.4.2})
that there is a natural equivalence $G \: (Y^q)^{t\bT} \simeq
((Y^q)^{tC_2})^{h\bT}$ for each~$q\ge0$. Since each $Y^q$ is even and $2$-torsion free by Proposition~\ref{prop: equiv of gr modules}, we know that each
$(Y^q)^{tC_2}$ is even, and we obtain a second filtration of $\gr_{\mot}^*\TP(\ko) \simeq \gr_{\mot}^*(\THH(\ko)^{tC_2})^{h\bT}$ by truncating the filtration
$\Tot \Sigma^{2\ast} H\pi_{2\ast} \GM_{\star}^{\bT}((Y^{\bullet})^{tC_2})$
to the range $\star\le 0$. We write $\fil_{\star}^{\GM} \gr_{\mot}^{\ast}(\THH(\ko)^{tC_2})^{h\bT}$ for this filtration. Letting the weight~$w$ vary, these are modules over $\gr_{\ev}^*
\bS \simeq \Tot \Sigma^{2*} H\pi_{2*} C^\bullet(\MU/\bS)$, regarded as a filtered graded spectrum with $\gr_{\ev}^*\bS$ in filtrations $\star \ge 0$ and $0$ in filtrations $\star <0$, where the structure maps are identities for $\star \ge 0$ and the zero map for $\star <0$.

We refer to these as Greenlees--May filtrations. Each one
is complete and exhaustive, hence induces a conditionally
convergent spectral sequence.  
We write
\[  \gr_{s}^{\GM} \gr_{\mot}^wF(\ko)= \fil_{s}^{\GM} \gr_{\mot}^wF(\ko)/ \fil_{s-1}^{\GM} \gr_{\mot}^wF(\ko)  \]
for each $F\in \{\TC^{-},\TP, F(S^{2\ell+1}_+,\THH)^{\bT},\THH^{hC_2},\THH^{tC_2},(\THH^{tC_2})^{h\bT}\}$. 
Letting~$s$ and~$w$ vary, these are modules over $\gr_{\ev}^* \bS$, with action mapping $\gr_{s}^{\GM} \gr_{\mot}^wF(\ko)$ to $\gr_{s}^{\GM} \gr_{\mot}^{w+*}F(\ko)$.
\end{definition}

Alternatively, we can take the approach of~\cite{HRW}*{Construction 2.1.14} and define the following filtrations of the associated graded of the motivic filtration.  Here we return to our standard conventions, and use (superscripted) cohomological indexing. 
\begin{definition}\label{filplus}
Let  $Y^{\bullet}=C^{\bullet}(\THH(\ku/\MU)/\THH(\ko))$. 
We define filtered graded objects 
\begin{align*} 
\fil_+^{\star}\gr_{\mot}^*\TC^{-}(\ko) & := \Tot \Sigma^{2*}H\pi_{2*} \left ( \tau_{\ge \star} Y^{\bullet} \right )^{h\bT} \,,  \\
\fil_+^{\star}\gr_{\mot}^*\TP(\ko)& := \Tot  \Sigma^{2*}H\pi_{2*} \left ( \tau_{\ge \star} Y^{\bullet} \right )^{t\bT}  \,,\\
\fil_+^{\star}\gr_{\mot}^*F(S^{2\ell+1}_+,\THH(\ko))^{\bT}& := \Tot  \Sigma^{2*}H\pi_{2*}F(S^{2\ell+1}_+, \tau_{\ge \star} Y^{\bullet} )^{\bT}   \,,\\
\fil_+^{\star}\gr_{\mot}^*\THH(\ko)^{hC_2} & := \Tot  \Sigma^{2*}H\pi_{2*} \left ( \tau_{\ge \star} Y^{\bullet} \right )^{hC_2}   \,,\\
\fil_+^{\star}\gr_{\mot}^*\THH(\ko)^{tC_2} & := \Tot  \Sigma^{2*}H\pi_{2*} \left ( \tau_{\ge \star} Y^{\bullet} \right )^{tC_2}     \text{ and }\\
\fil_+^{\star}\gr_{\mot}^*(\THH(\ko)^{tC_2})^{h\bT} & := \Tot \Sigma^{2*}H\pi_{2*} \left ( \tau_{\ge \star} ((Y^{\bullet})^{tC_2}) \right )^{h\bT} \,,  
\end{align*}
where $\ell\ge 0$. Each one is complete and exhaustive, by proofs similar
to that of Lemma~\ref{convergence}, hence induces a conditionally
convergent spectral sequence.
We refer to these as Whitehead filtrations. 
They come equipped with an action of the $\bE_\infty$ algebra in filtered graded spectra 
\[\Tot  \Sigma^{2*}H\pi_{2*} \left ( \tau_{\ge \star} C^{\bullet}(\MU/\bS ) \right ) \simeq
\begin{cases}
\gr_\ev^* \bS & \text{for $\star \le 2*$,} \\
0 & \text{for $\star > 2*$.}
\end{cases}
\]

We write 
\[ \gr_+^s\gr_{\mot}^wF(\ko)= \fil_+^s\gr_{\mot}^wF(\ko)/\fil_+^{s+1}\gr_{\mot}^wF(\ko) \]
for $F\in \{\TC^{-},\TP, F(S^{2\ell+1}_+,\THH)^{\bT}, \THH^{hC_2}, \THH^{tC_2},(\THH^{tC_2})^{h\bT}\}$. 
Letting~$s$ and~$w$ vary, these are modules over $\gr_{\ev}^* \bS$, with action mapping $\gr_+^s \gr_{\mot}^w F(\ko)$ to $\gr_+^{s+2*} \gr_{\mot}^{w+*} F(\ko)$.
\end{definition}

The filtrations from Definition~\ref{filplus} are clearly multiplicative, whereas our identification of the spectral sequence starting pages is more clear for the filtrations of Definition~\ref{GM}. Luckily, we can prove that they agree, up to an indexing shift that depends upon the weight
grading. 

\begin{proposition}\label{agreement1}
There is an equivalence of filtered $\gr_{\ev}^*\bS$-modules
\[  \fil_{-2n}^{\GM}\gr_{\mot}^{\ast}F(\ko)  \longrightarrow \fil_+^{2n+2\ast}\gr_{\mot}^{\ast}F(\ko)   \]
for $F\in \{\TC^{-},\TP,F(S^{2\ell+1}_+,\THH)^{\bT},(\THH^{tC_2})^{h\bT}\}$. 
\end{proposition}

\begin{proof}
Again, let $Y^{\bullet}=C^{\bullet}(\THH(\ku/\MU)/\THH(\ko))$. We first discuss the case $F=\TP$. 
We have
$$
\gr_{\mot}^w \THH(\ko) = \Tot \Sigma^{2w}H\pi_{2w}Y^\bullet
$$
and 
$$
\gr_{\mot}^w \TP(\ko) = \Tot \Sigma^{2w}H\pi_{2w} ((Y^\bullet)^{t\bT}) \,.
$$
The latter is filtered by 
$$
\fil_{m}^{\GM} \gr_{\mot}^w \TP(\ko)
	= \Tot \Sigma^{2w}H\pi_{2w}\GM_{m}^{\bT} (Y^\bullet)
$$
for $m\in \bZ$, with associated graded
\begin{align*}
\gr_{-2n}^{\GM} \gr_{\mot}^w \TP(\ko)
	&= \Tot \Sigma^{2w}H\pi_{2w} (\Omega^{2n} Y^\bullet) \\
	&\cong (\Tot \Sigma^{2w+2n}H\pi_{2w+2n} Y^\bullet) \, \{t^n\}
\end{align*}
in even gradings, where $t^n$ reduces stem by~$2n$.  In odd gradings, $\gr_{-2n+1}^{\GM} \gr_{\mot}^w \TP(\ko)= 0$.

The Whitehead filtration 
$$
\fil^{2w+2n}_+ \gr_{\mot}^w \TP(\ko)
	= \Tot \Sigma^{2w}H\pi_{2w} ((\tau_{\ge 2w+2n} Y^\bullet)^{t\bT}) 
$$
has associated graded 
\begin{align*}
\gr^{2w+2n}_+ \gr_{\mot}^w \TP(\ko)
	&= \Tot \Sigma^{2w}H\pi_{2w} ((\Sigma^{2w+2n}H\pi_{2w+2n} Y^\bullet)^{t\bT}) \\
	&\cong (\Tot \Sigma^{2w+2n}H\pi_{2w+2n} Y^\bullet) \, \{t^n\} \,.
\end{align*}
Again, $\gr^{2w+2n-1}_+ \gr_{\mot}^w \TP(\ko)= 0$. These constructions are lax symmetric monoidal,
by~\cite{NS18}*{Corollary~I.4.3}. 

The composite map 
\[
\Sigma^{2w}H\pi_{2w}([\widetilde{E\bT}_{-2n} \otimes F(E\bT_+, Y^\bullet)]^{\bT})
	\longto \Sigma^{2w}H\pi_{2w} ((Y^\bullet)^{t\bT}) 
	\longto \Sigma^{2w}H\pi_{2w} ((\tau_{< 2w+2n} Y^\bullet)^{t\bT})
\]
factors through 
$$
\Sigma^{2w}H\pi_{2w}( [ \widetilde{E\bT}_{-2n} \otimes F(E\bT_+, \tau_{< 2w+2n} Y^\bullet)]^{\bT}) = 0 \,,
$$
hence the first map in the composite lifts to induce a map of filtrations
$$
\fil_{-2n}^{\GM} \gr_{\mot}^w \TP(\ko)
	\longto \fil^{2w+2n}_+ \gr_{\mot}^w \TP(\ko) \,,
$$   
both with limit~$0$ and colimit
$\gr_{\mot}^w \TP(\ko)$. From the definition of the filtered $\gr_{\ev}^*\mathbb{S}$-module structure on either side it is clear that this is a map of filtered $\gr_{\ev}^*\mathbb{S}$-modules. The associated graded spectra are equivalent as $\gr_{\ev}^*\mathbb{S}$-modules,
hence so are the filtrations, up to the reindexing from $\fil_{-2n}^{\GM}$
to $\fil^{2w+2n}_+$. Hence the filtrations agree, up to reindexing. 
By restricting to filtrations
$m\le 0$, or $-2\ell \le m\le 0$, the same proof applies when replacing $\TP$ with $\TC^{-}$ or $F(S^{2\ell+1}_+,\THH)^{\bT}$. 
Since $Y^q$ is even and $2$-torsion free, we know $(Y^{q})^{tC_2}$ is even for each $q\ge 0$, and again the same proof applies to deduce the result for $F=(\THH^{tC_2})^{h\bT}$.  
\end{proof} 
 
\begin{corollary}\label{mult-TP}
Let $\olV$ be a $\gr_{\ev}^*\bS$-module. There are trigraded spectral sequences
\begin{align*} 
\olV_* \gr_{\mot}^* \THH(\ko) \, [t]
	& \Longrightarrow \olV_* \gr_{\mot}^* \TC^{-}(\ko) \,, \\ 
\olV_* \gr_{\mot}^* \THH(\ko) \, [t^{\pm 1}] & \Longrightarrow \olV_*  \gr_{\mot}^* \TP(\ko)  \,, \\ 
\olV_* \gr_{\mot}^* \THH(\ko) \, [t]/(t^{\ell+1}) 
	& \Longrightarrow \olV_* \gr_{\mot}^* F(S^{2\ell+1}_+,\THH(\ko))^{\bT}  \text{ and }\\
\olV_* \gr_{\mot}^* \THH(\ko)^{tC_2} \, [t] & \Longrightarrow \olV_* \gr_{\mot}^* (\THH(\ko)^{tC_2})^{h\bT} \,.
\end{align*}
The first is the (algebraic) homotopy fixed point spectral sequence, the second is the (algebraic) Tate spectral sequence, the third is the approximate fixed point spectral sequence, and the last is the $\mu$-inverted homotopy fixed point spectral sequence. 

When $\olV$ is a $\gr_{\ev}^*\bS$-algebra, these are algebra spectral sequences. When $\olV$ is a module over a given
$\gr_\ev^* \bS$-algebra, they are module spectral sequences over the
corresponding algebra spectral sequences. When $\olV$ is a finite $\gr_{\ev}^* \bS$-module, they
are conditionally convergent.
\end{corollary}

\begin{proof}
We apply $\olV\otimes$ to the filtered $\gr_{\ev}^*\bS$-modules $\fil_{\star}^{\GM}\gr_{\mot}^{\ast}F(\ko)$ for 
\[ F\in \{\TC^{-}, \TP, F(S^{2\ell+1}_+,\THH)^{\bT}, (\THH^{tC_2})^{h\bT}\}\,.\]
In the case of $F=\TP$, we observed in the proof of Proposition~\ref{agreement1} that there is an equivalence
$\gr^{\GM}_{-2n} \gr_{\mot}^* \TP(\ko) \simeq \gr_{\mot}^{*+n} \THH(\ko)
\{t^n\}$, so that we have an equivalence 
\[\olV \otimes \gr^{\GM}_{-2n} \gr_{\mot}^* \TP(\ko) \simeq
\olV \otimes \gr_{\mot}^{*+n} \THH(\ko) \{t^n\}\] 
for each integer~$n$. 
The cases $F = \TC^{-}$ and $F = F(S^{2\ell+1}_+, \THH)^{\bT}$ are obtained
by restricting the filtrations.  
The case of $F = (\THH^{tC_2})^{h\bT}$
is similar to that of~$\TC^{-}$, since each $(Y^q)^{tC_2}$ is even by Proposition~\ref{prop: equiv of gr modules}. The fact that these are algebra spectral sequences when $\olV$ is a $\gr_{\ev}^*\bS$-algebra follows by Proposition~\ref{agreement1} and the fact that the filtrations in Definition~\ref{filplus} are multiplicative. 
\end{proof}

\begin{proposition}\label{C2Tate}
Suppose $\olV$ is a $\gr_{\ev}^*\bS$-module such that $2$ acts trivially on $\olV$. Then we can identify the $C_2$-Tate spectral sequence associated to the filtered $\gr_{\ev}^*\bS$-module
\[ \olV\otimes \fil_+^{\star}\gr_{\mot}^*\THH(\ko)^{tC_2} \]
with the spectral sequence associated to the filtered $\gr_{\ev}^*\bS$-module 
\[ \olV\otimes \fil_{\star}^{\GM}\gr_{\mot}^*\THH(\ko)^{tC_2} \,.\]
The spectral sequence is of the form 
\begin{align*}
\pi_* \gr_{\mot}^* \THH(\ko) \, [t^{\pm 1}]\otimes \Lambda(u_1) 
	& \Longrightarrow \pi_* \gr_{\mot}^* \THH(\ko)^{tC_2} \,, 
\end{align*}
where $t$ in Tate filtration~$-2$ has bidegree $\|t\| = (-2, 0)$, and $u_1$ in Tate filtration~$-1$ has bidegree $\|u_1\| = (-1, -1)$. 

Under the same hypotheses, we can similarly identify the $C_2$-homotopy fixed point spectral sequence associated to the filtered $\gr_{\ev}^*\bS$-module
\[ \olV\otimes \fil_+^{\star}\gr_{\mot}^*\THH(\ko)^{hC_2} \]
with the spectral sequence associated to the filtered $\gr_{\ev}^*\bS$-module 
\[ \olV\otimes \fil_{\star}^{\GM}\gr_{\mot}^*\THH(\ko)^{hC_2} \,.\]
In this case, the spectral sequence is of the form 
\begin{align*}
\pi_* \gr_{\mot}^* \THH(\ko) \, [t]\otimes \Lambda(u_1) 
	& \Longrightarrow \pi_* \gr_{\mot}^* \THH(\ko)^{hC_2} \,, 
\end{align*}
where $t$
in homotopy fixed point filtration~$-2$ has bidegree $\|t\| = (-2,
0)$, and $u_1$ in homotopy fixed point filtration~$-1$ has bidegree
$\|u_1\| = (-1, -1)$. 

In each case, the spectral sequence is an algebra spectral sequence when $\olV$ is a $\gr_{\ev}^*\bS$-algebra, and is conditionally convergent
when $\olV$ is a finite $\gr_{\ev}^* \bS$-module. 
\end{proposition}

\begin{proof}
Let $Y^{\bullet}=C^{\bullet}(\THH(\ku/\MU)/\THH(\ko))$. 
We write $\tau_{[a,b]}=\cof (\tau_{\ge b+1} \to \tau_{\ge a})$ for integers $a<b$ and similarly we write 
$\GM^{C_2}_{[m,m+1]}(X)=\cof(\GM_{m-1}^{C_2}(X)\to \GM_{m+1}^{C_2}(X))$.
When $m$ is even,
there are equivalences
$$
\tau_{[2w,2w+1]}(\GM_{[m,m+1]}^{C_2} (Y^q))
	\simeq \tau_{[2w,2w+1]} ((\Sigma^{2w-m}H\pi_{2w-m} Y^q)^{tC_2}) 
$$
since $Y^q$ is even for each $q$. These are compatible with the cosimplicial structure maps, and totalize to an equivalence
$$
\fil^\GM_{[m,m+1]} \gr_{\mot}^w \THH(\ko)^{tC_2}
	\simeq \fil_+^{2w-m} \gr_{\mot}^w \THH(\ko)^{tC_2} \,.
$$
Here we write $\fil^\GM_{[m,m+1]} \gr_{\mot}^w \THH(\ko)^{tC_2}$ for the cofiber of 
$$
\fil^\GM_{m-1} \gr_{\mot}^w \THH(\ko)^{tC_2}
	\longto \fil^\GM_{m+1} \gr_{\mot}^w \THH(\ko)^{tC_2} \,.
$$

The free $C_2$-equivariant $(m+1)$-cell in $\widetilde{EC}_2
= S^\infty$ is attached to the free $C_2$-equivariant $m$-cell
by a pinch map $S^m \to S^m/S^{m-1} \cong C_{2+} \wedge S^m$.
Hence there are cofiber sequences
\[
\Sigma^{2w} H\pi_{2w}(\Sigma^m Y^q)
	\longto \tau_{[2w,2w+1]} \GM_{[m,m+1]}(Y^q) 
	\longto \Sigma^{2w+1} H\pi_{2w+1}(\Sigma^{m+1} Y^q)
	\overset{N}\longto \Sigma(-) \,,
\]
with $N = (+1)_*+(-1)_*$ the norm (or trace) map given by the
sum of the actions by the elements in $C_2$,
having totalization a cofiber sequence
\[
\gr^\GM_m \gr_{\mot}^w \THH(\ko)^{tC_2}
	\longto \fil^\GM_{[m,m+1]} \gr_{\mot}^w \THH(\ko)^{tC_2} 
	\longto \gr^\GM_{m+1} \gr_{\mot}^w \THH(\ko)^{tC_2}
	\overset{N}\longto \Sigma(-) \,.
\]
Since the $C_2$-action on each~$Y^q$ is the restriction of a $\bT$-action,
the map~$N$ is homotopic to multiplication by~$2$.

Using the earlier equivalence, and letting $w$ vary, we get a cofiber
sequence
\begin{multline*}
\gr^\GM_m \gr_{\mot}^{\ast} \THH(\ko)^{tC_2}
	\longto \fil_+^{2\ast-m} \gr_{\mot}^{\ast}\THH(\ko)^{tC_2} \\
	\longto \gr^\GM_{m+1} \gr_{\mot}^{\ast}\THH(\ko)^{tC_2}
	\overset{N}\longto \Sigma(-)
\end{multline*}
of $\gr_{\ev}^*\bS$-modules, where the $\gr_{\ev}^*\bS$-actions are as in Definitions~\ref{GM} and~\ref{filplus}.  For each $\gr_{\ev}^*\bS$-module~$\olV$,
we get another cofiber sequence
\begin{multline*}
\olV \otimes \gr^\GM_m \gr_{\mot}^* \THH(\ko)^{tC_2}
        \longto \olV \otimes  \fil_+^{2\ast-m} \gr_{\mot}^* \THH(\ko)^{tC_2} \\
        \longto \olV \otimes \gr^\GM_{m+1} \gr_{\mot}^* \THH(\ko)^{tC_2}
        \overset{\olV \otimes N}\longto \dots \,,
\end{multline*}
where we again consider the $\gr_{\ev}^*\bS$-actions as in Definitions~\ref{GM} and~\ref{filplus}. If multiplication by~$2$ acts trivially on $\olV$, e.g.~if this is a
$\olV(0)$-module, then the connecting map $\olV \otimes  N$ is null-homotopic, and the cofiber sequence splits.
Then the spectral sequence associated to the
filtration $\olV \otimes \fil_+^{\star} \gr_{\mot}^*
\THH(\ko)^{tC_2}$ of $\olV \otimes  \gr_{\mot}^* \THH(\ko)^{tC_2}$
has starting term the sum over $n \in \bZ$ of
\begin{multline*}
\olV_* \gr_+^{2*+2n} \gr_{\mot}^* \THH(\ko)^{tC_2} \\
\cong \olV_* \gr^\GM_{-2n} \gr_{\mot}^* \THH(\ko)^{tC_2}
\oplus \olV_* \gr^\GM_{-2n+1} \gr_{\mot}^* \THH(\ko)^{tC_2} \\
\cong
\olV_* \gr_{\mot}^* \THH(\ko)^{tC_2} \{t^n, u_1 t^{n-1}\} \,,
\end{multline*}
with $t^n$ and $u_1 t^{n-1}$ in (stem, motivic filtration)
bigrading equal to $(-2n, 0)$ and $(-2n+1, -1)$, respectively. Truncating at $\star \le 0$ produces the result for $(-)^{hC_2}$ in place of $(-)^{tC_2}$. 

Suppose $\olV$ is a $\gr_{\ev}^*\bS$-algebra. To confirm the formula
\[  
    \olV_*\gr^{\GM}_*\gr_{\mot}^*\THH(\ko)^{tC_2}
    \cong
    \olV_* \gr_{\mot}^* \THH(\ko)\, [t^{\pm1}] \otimes \Lambda(u_1)
\]
for the $E^2$-page of the $C_2$-Tate spectral sequence as a trigraded $\mathbb{F}_2$-algebra, it
suffices to check that multiplication by $u_1t^{-1}$ takes the copy of 
$\olV_* \gr_{\mot}^* \THH(\ko)$
in Tate filtration $0$ to the copy in Tate filtration $1$, since we know that $t$, in the case of the $\bT$-Tate spectral sequence, acts invertibly on this spectral sequence. 
However, this is already built into the notation since, in the $n=0$ case, the two module generators there are $1 = t^0$ and $u_1 t^{-1}$. Note that $u_1^2$ lands in a bidegree that is trivial, so the relation $u_1^2=0$ follows. In the case of the $C_2$-homotopy fixed point spectral sequence, the algebra structure on the $E^2$-page can be determined from the fact that it is a subalgebra of the $E^2$-page of the $C_2$-Tate spectral sequence. 
\end{proof}

\section{Detection} \label{sec:detection}

The classical mod~$2$ Adams spectral sequence
\[
{}^{\Ad} E_2(X) = \Ext_{\cA^\vee}(\bF_2, H_*(X))
	\Longrightarrow \pi_*(X^\wedge_2)
\]
is strongly convergent for bounded below spectra~$X$ with $H_*(X)$ of
finite type.  Its $E_2$-term can be calculated as the cohomology of the
normalized cobar complex
\[
0 \longto H_*(X) \overset{d_1^0} \longto
	\bar\cA^\vee \otimes H_*(X) \overset{d_1^1} \longto
	\bar\cA^\vee \otimes \bar\cA^\vee \otimes H_*(X) \longto \dots \,.
\]
Here $\bar\cA^\vee = \cok(\bF_2 \to \cA^\vee)$, and we will use
the notation $[a]m = a \otimes m \in \bar\cA^\vee \otimes H_*(X)$.
Recall that $d_1^0$ is given by the normalized $\cA^\vee$-coaction
on $H_*(X)$, while $d_1^1$ also involves the coproduct $\psi \: \cA^\vee
\to \cA^\vee \otimes \cA^\vee$.

When $X = A(1)[ij]$ as in Notation~\ref{not: A(1) def}, the Adams
$E_2$-terms
\[
{}^{\Ad} E_2(A(1)[ij]) = \Ext_{\cA}(H^*(A(1)[ij]), \bF_2)
	\Longrightarrow \pi_* A(1)[ij]
\]
are readily calculated in a finite range using Bruner's {\tt ext}
software~\cites{Bru93, BR}.  The results in stems $* \le 28$ are
shown in Figure~\ref{fig:Ext-A-A1abcd}, with the usual (stem, Adams
filtration) bigrading.  Lines of bidegree $(0,1)$, $(1,1)$ and~$(3,1)$
(dashed) indicate multiplications by $h_0$, $h_1$ and~$h_2$, respectively.

In each case, the three $1$-cochains
\begin{equation} \label{eq:cocycles}
[\xi_1^4]1
	\quad,\quad
[\xi_2^2]1 + [\xi_1^4]\xi_1^2
	\quad\text{and}\quad
[\xi_3]1 + [\xi_2^2]\xi_1 + [\xi_1^4]\xi_2
\end{equation}
in $\bar\cA^\vee \otimes \cA(1)^\vee$ are cocycles, but not coboundaries,
hence they represent nonzero classes in ${}^{\Ad} E_2(A(1)[ij])$ in bidegrees
$(3, 1)$, $(5, 1)$ and~$(6, 1)$, respectively.  For sparsity reasons,
these classes all survive to ${}^{\Ad} E_\infty(A(1)[ij])$, and detect nonzero
homotopy classes in stems $3$, $5$ and~$6$, denoted
\[
\nu  \quad,\quad  w  \quad\text{and}\quad  v_2
\]
in $\pi_* A(1)[ij]$, for each $i,j \in \{0,1\}$.  Observe that $v_2$
is only defined modulo $\nu^2$, and $2 v_2 = 0$ if $i=0$ while $2 v_2 =
\nu^2$ if $i=1$.

\begin{figure}
\centering
\resizebox{\textwidth}{!}{ \input{Ext-A-A1a.inp} }
\vskip 5mm
\resizebox{\textwidth}{!}{ \input{Ext-A-A1b.inp} }
\vskip 5mm
\resizebox{\textwidth}{!}{ \input{Ext-A-A1c.inp} }
\vskip 5mm
\resizebox{\textwidth}{!}{ \input{Ext-A-A1d.inp} }
\caption{Adams $E_2$-terms for $A(1)[00]$, $A(1)[10]$, $A(1)[01]$ and
	$A(1)[11]$ (from top to bottom) \label{fig:Ext-A-A1abcd}}
\end{figure}

\begin{lemma} \label{lem: Adams for A(1)}
In the Adams spectral sequences for the $A(1)[ij]$ the differentials
originating in stems $* \le 24$ are all zero.  The class $\nu \bar\kappa
\in \pi_{23}(\bS)$ maps to zero in $\pi_{23} A(1)[ij]$.
\end{lemma}

\begin{proof}
This mostly follows from sparsity and the module structure over the
Adams spectral sequence for~$\bS$, using that $d_2(h_4) = h_0 h_3^2$
maps to zero under $\bS \to A(1)$.  Only the Adams $d_2$-differential from
bidegree $(t-s,s) = (19,2)$ requires special attention, but the Novikov
$E_2$-term (see Figure~\ref{fig:Ext-BPA1}) shows that $\pi_{19} A(1)$
has order $2^2 = 4$, so there is no room for such an Adams differential.

In each case, the map ${}^{\Ad} E_2(\bS) \to {}^{\Ad} E_2(A(1))$ of Adams
$E_2$-terms takes the bidegree $(23,5)$ class $h_2 g$ detecting $\nu
\bar\kappa$ to zero, as can be checked with {\tt ext}, and the target
has no classes in stem~$23$ and Adams filtration $\ge 6$.  It follows
that $\nu \bar\kappa \mapsto 0$.
\end{proof}

\begin{figure}
\centering
\resizebox{\textwidth}{!}{ \input{Einfty-v2Bock.inp} }
\caption{$v_2$-Bockstein $E_\infty \Longrightarrow
	\Ext_{\BP_*\BP}(\BP_*, \BP_*/I_2)$ \label{fig:Ext-BP/I2}}
\end{figure}

To calculate the Novikov $E_2$-term
\[
{}^{\Nov} E_2(A(1)) = \Ext_{\MU_*\MU}(\MU_*, \MU_* A(1))
	\cong \Ext_{\BP_*\BP}(\BP_*, \BP_* A(1))
\]
for these spectra, we can note that $\BP_* A(1) = \BP_*/I_2 \{1, t_1\}$
and use the long exact sequence obtained by applying $\Ext_{\BP_*\BP}(\BP_*,
-)$ to the $\BP_*\BP$-comodule extension
\[
0 \longto \BP_*/I_2 \longto \BP_* A(1) \longto \Sigma^2 \BP_*/I_2 \to 0
\]
classified by
\[
h_{10} = [t_1] \in \Ext_{\BP_*\BP}(\BP_*, \BP_*)
\]
in (stem, Novikov filtration) bidegree~$(1,1)$.  The groups
\[
\Ext_{\BP_*\BP}(\BP_*, \BP_*/I_2)
\]
are calculated in a range as in \cite{Rav86}*{\S4.4, p.~162}, starting
with the isomorphism
\[
\Ext_{\cA}(\bF_2, \bF_2) \cong \Ext_{\BP_*\BP}(\BP_*, \BP_*/I_\infty)
\]
that doubles internal degrees, followed by the $v_n$-Bockstein spectral
sequences
\[
E_1 = \Ext_{\BP_*\BP}(\BP_*, \BP_*/I_{n+1}) \, [v_n]
	\Longrightarrow \Ext_{\BP_*\BP}(\BP_*, \BP_*/I_n)
\]
for descending $n\ge2$.  The $v_2$-Bockstein spectral
sequence $E_\infty$-term for $\BP_*/I_2$ in stems $* \le
26$ is shown in Figure~\ref{fig:Ext-BP/I2}, corresponding to
\cite{Rav86}*{Fig. 4.4.23(c)}.  Lines of bidegree $(1,1)$, $(3,1)$ and
$(7,1)$ (dashed) indicate multiplications by $h_{10} = [t_1]$, $h_{11} =
[t_1^2]$ and $h_{12} = [t_1^4]$, respectively.  (Some) hidden extensions
are shown as dotted lines.  We emphasize the relation
\begin{equation} \label{eq:v2h113=v22h103}
v_2 h_{11}^3 = v_2^2 h_{10}^3 \,,
\end{equation}
which follows from $v_2 h_{12} = v_2^2 h_{10}$, which in turn follows
from the formula $\eta_R(v_3) \equiv v_3 + v_2 t_1^4 + v_2^2 t_1 \mod I_2$
for the right unit in~$\BP_*\BP$, see~\cite{Rav86}*{4.3.1}.

Alternatively, one can start with the internal-degree-doubling
isomorphism
\[
\Ext_{\cA}(H^*(C2), \bF_2)
	\cong \Ext_{\BP_*\BP}(\BP_*, \BP_*/I_\infty\{1, t_1\})
\]
and calculate the $v_n$-Bockstein spectral sequences
\[
E_1 = \Ext_{\BP_*\BP}(\BP_*, \BP_*/I_{n+1}\{1, t_1\}) \, [v_n]
	\Longrightarrow \Ext_{\BP_*\BP}(\BP_*, \BP_*/I_n\{1, t_1\})
\]
for descending $n\ge2$.  The Adams $E_2$-term for~$C2$ in stems $* \le 16$
is shown in Figure~\ref{fig:Ext-A-M1}, and the resulting $v_2$-Bockstein
$E_\infty$-term for $\BP_*/I_2\{1, t_1\} = \BP_* A(1)$ in stems $* \le 26$
is shown in Figure~\ref{fig:Ext-BPA1}.  Again, (some) hidden extensions
are shown as dotted lines.

\begin{figure}
\centering
\resizebox{.7\textwidth}{!}{ \input{Ext-A-M1.inp} }
\caption{Adams $E_2$-term for~$C2$ \label{fig:Ext-A-M1}}
\end{figure}

\begin{figure}
\centering
\resizebox{\textwidth}{!}{ \input{Einfty-v2BockCeta.inp} }
\caption{$v_2$-Bockstein $E_\infty \Longrightarrow
	\Ext_{\BP_*\BP}(\BP_*, \BP_* A(1))$ \label{fig:Ext-BPA1}}
\end{figure}

\begin{lemma} \label{lem: Nov A1}
In the Novikov spectral sequences for the $A(1)[ij]$ the
nonzero differentials originating in stems $* \le 22$ are
\[
d_3(v_2^2) = h_{11}^2 w
	\qquad\text{and}\qquad
d_3(v_2^3) = v_2 h_{11}^2 w \,.
\]
In the cases $A(1)[10]$ and $A(1)[11]$ there is a nonzero
$d_3$ from bidegree $(t-s,s) = (23,1)$.

In every case $d_3(v_2^4) = 0$ and
$d_5(v_2^4) \ne 0$.
\end{lemma}

\begin{proof}
This follows by comparison of the order in each stem of the Adams
$E_\infty$-term, which equals that of the abutment $\pi_* A(1)[ij]$,
with the order in each stem of the Novikov $E_2$-term.  In particular,
$\pi_{12} A(1) = \bZ/2$ implies that $v_2^2$ must support a nonzero
differential.  Similarly, the group $\pi_{18} A(1)$ has order~$2^2$,
so $v_2^3$ must support a nonzero differential.  The groups $\pi_{22}
A(1)[ij]$ have order $2^3 = 8$ for $i = 0$ and order $2^2 = 4$
for $i = 1$, while the groups $\pi_{23} A(1)[ij]$ have order~$2^4$
for $i = 0$ and~$2^3$ for $i = 1$.  To account for this, the Novikov
differential~$d_3$ from bidegree $(t-s,s) = (23,1)$ to $(22,4)$ must
be nonzero when $i = 1$.  Moreover, there must be a rank~$1$ Novikov
differential from the $24$-stem to the $23$-stem.  By $h_{11}$-linearity,
it cannot originate in bidegree $(24,2)$, hence it is either a~$d_3$
or a~$d_5$ starting on $v_2^4$. 

Inspection of the Novikov $E_2$-term for~$\bS$ in \cite{Rav86}*{Figure
4.4.45} shows that $\nu \bar\kappa \in \pi_{23}(\bS)$ is detected by a
generator~$x$ of the $\bZ/8$ in (stem, Novikov filtration) bidegree $(23,
5)$ of ${}^{\Nov} E_2(\bS)$.  The unit map $\bS \to A(1)$ takes this
generator~$x$ to the generator~$y$ of the $\bZ/2$ in the same bidegree
of ${}^{\Nov} E_2(A(1))$, see Figure~\ref{fig:Ext-BPA1}.  Since $\nu
\bar\kappa$ maps to zero in $\pi_{23} A(1)$ (by Lemma~\ref{lem: Adams
for A(1)}) it follows that this nonzero class~$y$ is a boundary, and so
$d_5(v_2^4) = y \ne 0$ is the only possibility.  In particular, we must
have $d_3(v_2^4) = 0$.
\end{proof}

There is a cofiber sequence
\begin{equation} \label{eq:approx}
\Sigma^{-2}\THH(\ko)
	\overset{i}\longto F(S^3_+, \THH(\ko))^{\bT}
	\overset{p}\longto\THH(\ko)
	\overset{\sigma}\longto \Sigma^{-1}\THH(\ko) \,,
\end{equation}
where $\sigma$ is induced by the $\bT$-action on $\THH(\ko)$, and
a commutative diagram
\[
\xymatrix{
\bS \ar[r] \ar[d] & \TC^{-}(\ko) \ar[r]^-{q} \ar[dr]
	& F(S^3_+, \THH(\ko))^{\bT} \ar[d]^p \\
\ko \ar[rr] &&\THH(\ko) \,.
}
\]

By truncating the ordinary homotopy fixed point spectral sequence
\[
E^2 = A(1)_*\THH(\ko) \, [t] \Longrightarrow A(1)_* \TC^{-}(\ko) \,,
\]
where $t$ is in stem~$-2$, we obtain a two-column approximate fixed point spectral sequence
\begin{equation} \label{eq:approxspeqseq}
E^2 = A(1)_*\THH(\ko) \{1, t\}
        \Longrightarrow A(1)_* F(S^3_+, \THH(\ko))^{\bT} \,,
\end{equation}
which is really just the long exact sequence in $A(1)$-homotopy
associated to the cofiber sequence~\eqref{eq:approx}.
We have the following analogue of \cite{AR02}*{Proposition~4.8}.

\begin{proposition} \label{prop: detection}
The unit images in $A(1)_* \TC^{-}(\ko)$ and $A(1)_* F(S^3_+,
\THH(\ko))^{\bT}$ of the classes $\nu$, $w$ and $v_2 \in \pi_* A(1)$ are
detected by
\[
t \lambda'_1
	\quad,\quad
t \lambda_2
	\quad\text{and}\quad
t \mu \,,
\]
respectively, in the homotopy fixed point and approximate fixed point
spectral sequences.
\end{proposition}

\begin{proof}
By naturality, it suffices to prove this in the approximate fixed point
case.  The unit map takes the infinite cycles in~\eqref{eq:cocycles},
detecting $\nu$, $w$ and~$v_2$ in $\pi_* A(1)$, to the $1$-cocycles
\begin{equation} \label{eq:approxcocycles}
\begin{aligned}
	& [\xi_1^4](1 \otimes 1) \\
	& [\xi_2^2](1 \otimes 1) + [\xi_1^4](\xi_1^2 \otimes 1) \\
	& [\xi_3](1 \otimes 1) + [\xi_2^2](\xi_1 \otimes 1)
		+ [\xi_1^4](\xi_2 \otimes 1)
\end{aligned}
\end{equation}
in
$\bar\cA^\vee \otimes \cA(1)^\vee \otimes H_*(F(S^3_+, \THH(\ko))^{\bT})$.
We claim that these are not in the image of the coboundary $d_1^0$
from the $0$-cochains
\[
\cA(1)^\vee \otimes H_*(F(S^3_+, \THH(\ko))^{\bT}) \,,
\]
hence represent nonzero classes in ${}^{\Ad} E_\infty(A(1) \otimes
F(S^3_+, \THH(\ko))^{\bT})$, detecting the (nonzero) images of
$\nu$, $w$ and~$v_2$ in $A(1)_* F(S^3_+, \THH(\ko))^{\bT}$.

Recall that $H_*(\ko) = \bF_2[\bar\xi_1^4, \bar\xi_2^2, \bar\xi_3,
\dots]$, $H_*\THH(\ko) = H_*(\ko) \otimes \Lambda(\sigma\bar\xi_1^4,
\sigma\bar\xi_2^2) \otimes \bF_2[\sigma\bar\xi_3]$ and $\cA(1)^\vee =
\bF_2[\xi_1, \xi_2]/(\xi_1^4, \xi_2^2)$.
In the long exact sequence associated to~\eqref{eq:approx}, the map
$\sigma$ has kernel $\bF_2\{1, \sigma\bar\xi_1^4, \sigma\bar\xi_2^2\}$
in degrees $\le 7$, and the image of~$i$ consists of $t$-multiples.
In the extension
\[
0 \longto \cA(1)^\vee \otimes \im(i)
	\longto \cA(1)^\vee \otimes H_* F(S^3_+, \THH(\ko))^{\bT}
	\longto \cA(1)^\vee \otimes \ker(\sigma) \to 0
\]
the coboundaries on classes in $\cA(1)^\vee \otimes \im(i)$ will lie
in $\bar\cA^\vee \otimes \cA(1)^\vee \otimes \im(i)$, hence do not
contribute any terms of the form $[a](m \otimes 1)$.  The classes $1$,
$\sigma\bar\xi_1^4$ and $\sigma\bar\xi_2^2$ are $\cA^\vee$-comodule
primitive in $\ker(\sigma)$, hence lift to classes in $H_* F(S^3_+,
\THH(\ko))^{\bT}$ that are $\cA^\vee$-comodule primitive modulo $\im(i)$,
so also the coboundaries on (the lifts of) $\cA(1)^\vee \otimes \bF_2\{1,
\sigma\bar\xi_1^4, \sigma\bar\xi_2^2\}$ do not contain any terms of the
form in~\eqref{eq:approxcocycles}.  This proves our claim.
 
It remains to be determined where in~\eqref{eq:approxspeqseq} the
(nonzero) unit images of $\nu$, $w$ and~$v_2$ are detected.  Recall that
$A(1)_*\THH(\ko) = \Lambda(\lambda'_1, \lambda_2) \otimes \bF_2[\mu]$
is equal to $\bF_2\{1, \lambda'_1, \lambda_2, \mu\}$ in stems $\le 8$.
The composite map
\[
A(1) \longto A(1) \otimes F(S^3_+, \THH(\ko))^{\bT}
	\overset{p}\longto A(1) \otimes \THH(\ko)
\]
factors through $A(1) \otimes \ko \simeq \bF_2$,
so the images of $\nu$, $w$ and~$v_2$ in $A(1)_*\THH(\ko)$ are all zero.
(This was obvious for $\nu$ and~$v_2$.)  Hence the nonzero images of
$\nu$, $w$ and~$v_2$ must all be detected by $t$-multiples in the
approximate fixed point spectral sequence, and for degree reasons the
only possible detecting classes are $t \lambda'_1$, $t \lambda_2$ and
$t \mu$, respectively.
\end{proof}
 
\begin{remark} 
In light of~\cite{HW22}*{Lemma~A.4.1}, we may view Proposition~\ref{prop: detection} as confirming that there are unique lifts of $\nu$, $w$ and $v_2$ so that $\lambda'_{1}=\sigma^{2}\nu$, $\lambda_{2}=\sigma^{2}w$ and $\mu=\sigma^{2}v_{2}$, where 
\[ \sigma^{2}: A(1)_{*}\Sigma^2 \fib(\bS\to \ko)\longrightarrow A(1)_{*}\THH(\ko)  \] 
is the reduced suspension map of \cite{HW22}*{Example A.2.4}. Here uniqueness follows because $A(1)_{*}\ko=\bF_2$. 
\end{remark}

Since $\MU_* A(1)$ is even, the motivic spectral sequence
\[
E_2 = \pi_* \olA(1) \Longrightarrow \pi_* A(1)
\]
can be identified with the Novikov spectral sequence
\[
{}^{\Nov} E_2(A(1)) = \Ext_{\BP_*\BP}(\BP_*, \BP_* A(1))
	\Longrightarrow \pi_* A(1) \,,
\]
as in \cite{HRW}*{Corollary~A.2.6}.  The spectral sequence
must collapse in stems $\le 10$, for sparsity reasons, so the three
classes denoted $h_{11}$, $w$ and~$v_2$ in Figure~\ref{fig:Ext-BPA1}
must detect $\nu$, $w$ and~$v_2$ in $\pi_* A(1)$, respectively.  We can
also identify $\pi_* \olV(1)$ with $\Ext_{\BP_*\BP}(\BP_*, \BP_*/I_2)$,
with classes $h_{11}$ and~$v_2$ as shown in Figure~\ref{fig:Ext-BP/I2},
but in this case there is no Novikov spectral sequence to $\pi_* V(1)$,
since the spectrum~$V(1)$ does not exist.

\begin{corollary} \label{cor:detection}
The classes $h_{11}$, $w$ and~$v_2$ in $\pi_* \olA(1)$ map by the
unit to classes in
\[
\olA(1)_* \gr_{\mot}^* \TC^{-}(\ko)
\]
detected by $t \lambda'_1$, $t \lambda_2$ and $t\mu$, respectively.
Likewise, the images of $h_{11}$ and~$w$ in $\pi_*(\olV(2) \otimes
\olC\eta)$ are detected by $t \lambda'_1$ and $t \lambda_2$ in $(\olV(2)
\otimes \olC\eta)_* \gr_{\mot}^* \TC^{-}(\ko)$.

The classes $h_{11}$ and~$v_2$ in $\pi_* \olV(1)$ map by the unit to
classes in
\[
\olV(1)_* \gr_{\mot}^* \TC^{-}(\ko)
\]
detected by $t \lambda'_1$ and $t\mu$, respectively.  Likewise, the image
of $h_{11}$ in $\pi_* \olV(2)$ in $\olV(2)_* \gr_{\mot}^* \TC^{-}(\ko)$
is detected by $t \lambda'_1$.

In each of these cases, for $V \in \{V(1), A(1), V(2), V(2) \otimes
C\eta\}$, a unit image detected by a class in $\olV_* \gr_{\mot}^*
\TC^{-}(\ko)$ is also detected in $\olV_* \gr_{\mot}^* F(S^3_+,
\THH(\ko))^{\bT}$, by the class with the same name.
\end{corollary} 

\begin{proof}
Naturality with respect to
\[
q \: \TC^{-}(\ko)=F(E\bT_{+},\THH(\ko))^{\bT}
	\longto  F(S^{3}_{+},\THH(\ko))^{\bT}
\]
shows that it suffices to prove these assertions in the case of the
approximate homotopy fixed points.  Using Corollary~\ref{cor: A(1) coeff}, the approximate fixed point spectral sequence 
\begin{align*}
E^2 =\Lambda(\lambda'_1, \lambda_2) \otimes \bF_2[\mu] \, \{1, t\}  \implies  \olA(1)_*\gr_{\mot}^*  F(S^3_+,\THH(\ko))^{\bT}
\end{align*}
and the motivic spectral sequence
\begin{align*}
E_2 = \olA(1)_*\gr_{\mot}^*  F(S^3_+,\THH(\ko))^{\bT}\Longrightarrow A(1)_* F(S^3_+,\THH(\ko))^{\bT}
\end{align*}
are concentrated in motivic filtrations $0 \le * \le 2$ and integer weights,
hence the motivic spectral sequence collapses, and the result follows from Proposition~\ref{prop:
detection}.
The claim with coefficients in $\olV(2) \otimes \olC\eta$
follows by passing to cofibers for multiplication by~$v_2$.

The $E^2$-term of the approximate fixed point spectral sequence 
\begin{align*}
E^2=(\olV(1)_* \gr_{\mot}^* \THH(\ko)) \, \{1, t\} \Longrightarrow \olV(1)_*\gr_{\mot}^* F(S^3_+,\THH(\ko))^{\bT} \,, 
\end{align*}
which computes the $E_{2}$-term of the motivic spectral sequence
\begin{align*}
E_2=\olV(1)_*\gr_{\mot}^* F(S^3_+,\THH(\ko))^{\bT}\Longrightarrow V(1)_*  F(S^3_+,\THH(\ko))^{\bT} \,,
\end{align*}
was determined in
Proposition~\ref{prop: eta BSS computatation}.  Since $h_{11}$ and~$v_2$
in $\pi_* \olV(1)$ map to $h_{11}$ and~$v_2$ in $\pi_* \olA(1)$, the
claim with $\olV(1)$-coefficients follows from the commuting square
\[
\xymatrix{
\pi_*\olV(1) \ar[r]^-{i} \ar[d]
	& \pi_*\olA(1) \ar[d] \\
\olV(1)_* \gr_{\mot}^* F(S^3_+,\THH(\ko))  \ar[r]^-{i \otimes 1}
	& \olA(1)_*\gr_{\mot}^* F(S^3_+,\THH(\ko))    
}
\]
and the injectivity of the bottom map in the relevant bidegrees, which can be observed from the map of approximate fixed point spectral sequences. 
Again, the claim with $\olV(2)$-coefficients follows by passing to
cofibers for multiplication by~$v_2$.
\end{proof}

In Section~\ref{sec:syntomic} we shall calculate the syntomic cohomology
$\olA(1)_* \gr_{\mot}^* \TC(\ko)$, and in Section~\ref{sec:tc} we shall
use the associated motivic spectral sequence to calculate $A(1)_*
\TC(\ko)$.  The following lemma and corollary will be used to show
that the $d_3$-differentials in this spectral sequence propagate in a
$v_2^4$-periodic pattern.  Recall from~\cite{Lan73}*{Proposition~2.11}
that
\[
\bF_2[v_2] = \Ext_{\BP_*\BP}^0(\BP_*, \BP_*/(2,v_1))
	\subset \BP_*/(2,v_1) \,.
\]

\begin{lemma} \label{lem:d3v24zero}
Let $\End A(1) = F(A(1), A(1))$ be the endomorphism $\bS$-algebra of
any one of the spectra~$A(1)[ij]$.  The induced $\BP_*\BP$-comodule
$\BP_*$-algebra structure on $\BP_* \End A(1)$ descends to a
$\BP_*\BP$-comodule $\BP_*/(2,v_1)$-algebra structure.  Hence the
$E_2$-term of the Novikov spectral sequence
\[
{}^{\Nov} E_2(\End A(1)) = \Ext_{\BP_*\BP}(\BP_*, \BP_* \End A(1))
	\Longrightarrow \pi_* \End A(1)
\]
is an algebra over $\Ext_{\BP_*\BP}(\BP_*, \BP_*/(2,v_1)) \supset
\bF_2[v_2]$.  In particular,
\[
d_2(v_2^2) = 0
\qquad\text{and}\qquad
d_3(v_2^4) = 0
\]
in this Novikov spectral sequence.  Moreover, $4 \cdot \id = 0$ in $\pi_0
\End A(1)$.
\end{lemma}

\begin{proof}
To see that the $\BP_*$-algebra structure in $\BP_*\BP$-comodules
descends to a $\BP_*/(2,v_1)$-algebra structure in that category, it
suffices to check that the homomorphism 
\[\BP_* \to \BP_* \End A(1)\,,\]
induced by the $\bS$-algebra unit map, sends $2$ and~$v_1$ to zero.
This can be verified using the map of Adams spectral sequences
\[
{}^{\Ad} E_2(\BP) = \Ext_{\cA}(H^* \BP, \bF_2) 
	\longto \Ext_{\cA}(H^*(\BP \otimes \End A(1)), \bF_2)
	= {}^{\Ad} E_2(\BP \otimes \End A(1)) \,.
\]
Standard minimal resolution calculations, which can be obtained from
Bruner's {\tt ext} program, see~\cites{Bru93, BR}, show that multiplication
by $h_0$ and $\< h_0, h_1, -\>$ both act trivially from bidegree~$(0,0)$
on the right-hand side, and that there are no classes in stems~$0$ or~$2$
that have Adams filtration $\ge 2$.  The last fact also implies that $4
\cdot \id = 0$ in $\pi_0 \End A(1)$, and a closer inspection shows that
$2 \cdot \id \ne 0$ in each case.

The functor $\Ext_{\BP_*\BP}(\BP_*, -)$ is lax symmetric monoidal, so
it follows that ${}^{\Nov} E_2(\End A(1))$ is an $\Ext_{\BP_*\BP}(\BP_*,
\BP_*/(2,v_1))$-algebra.  In particular, it is an $\bF_2[v_2]$-algebra,
with $v_2$ acting centrally and $1+1=0$.  Hence we have equalities $d_2(v_2^2) = d_2(v_2)
\cdot v_2 + v_2 \cdot d_2(v_2) = 0$ and $d_3(v_2^4) = d_3(v_2^2) \cdot
v_2^2 + v_2^2 \cdot d_3(v_2^2) = 0$.
\end{proof}

The tautological left action of~$\End A(1)$ on $A(1)$ induces a
left action of the Novikov filtration of~$\End A(1)$ on the Novikov
filtration of~$A(1)$.  The latter is equivalent to the even filtration
$\fil_{\ev}^{\star} A(1)$, since $\MU_* A(1)$ is concentrated in
even degrees.  Hence $\fil_{\Nov}^{\star} \End A(1)$ also acts on the
convolution product filtration
\[
\fil_{\ev}^{\star} A(1) \otimes_{\fil_{\ev}^{\star} \bS}
	\fil_{\mot}^{\star} \TC(\ko) \,,
\]
and induces a left action of the Novikov spectral sequence for $\End A(1)$
on the motivic spectral sequence converging to~$A(1)_* \TC(\ko)$.
 
\begin{corollary} \label{cor:d3isv24periodic}
The differentials in the motivic spectral sequence
\[
E_2 = \olA(1)_* \gr_{\mot}^* \TC(\ko)
        \Longrightarrow A(1)_* \TC(\ko)
\]
satisfy $d_3(v_2^4 \cdot y) = v_2^4 \cdot d_3(y)$ for all $y \in E_3$.
\end{corollary}

\begin{proof}
This follows from the Leibniz rule
$d_3(v_2^4 \cdot y) = d_3(v_2^4) \cdot y + v_2^4 \cdot d_3(y)$
for the pairing of spectral sequences
\[
{}^{\Nov} E_2(\End A(1)) \otimes \olA(1)_* \gr_{\mot}^* \TC(\ko)
	\longto \olA(1)_* \gr_{\mot}^* \TC(\ko) \,,
\]
since $d_3(v_2^4) = 0$ by Lemma~\ref{lem:d3v24zero}.
\end{proof}

\section{Prismatic cohomology} \label{sec:prismatic}

Recall that by Corollary~\ref{mult-TP} there is a $\olV(2)$-homotopy $\bT$-Tate spectral sequence
\begin{equation} \label{eq: v2 Tate}
\begin{aligned}
\hat E^2(\bT) &= \olV(2)_* \gr_{\mot}^* \THH(\ko) \, [t^{\pm1}] \\
	&\Longrightarrow \olV(2)_* \gr_{\mot}^* \TP(\ko)
\end{aligned}
\end{equation}
and a $\olV(2) \otimes \olC\eta$-homotopy $\bT$-Tate spectral sequence
\begin{equation} \label{eq: v2 Ceta Tate}
\begin{aligned}
\hat E^2(\bT)
	&= (\olV(2) \otimes \olC\eta)_* \gr_{\mot}^* \THH(\ko) \, [t^{\pm1}] \\
	&\Longrightarrow (\olV(2) \otimes \olC\eta)_* \gr_{\mot}^* \TP(\ko)  \,. 
\end{aligned}
\end{equation}

They can be reindexed as cohomologically graded periodic $t$-Bockstein
spectral sequences, in which case $\hat E^{2r}(\bT)$ and~$d^{2r}$
correspond to $E_r$ and~$d_r$.  However, we shall need to make a
comparison with similar $C_2$-Tate spectral sequences, for which our
indexing is convenient. 
\begin{theorem}[Prismatic cohomology modulo $(2, v_1, v_2)$ of~$\ko$]
	\label{thm: prismatic}
The $\olV(2)$-homotopy $\bT$-Tate spectral sequence~\eqref{eq: v2 Tate}
is an algebra spectral sequence with $E^2$-term
\[
\hat E^2(\bT) = \Lambda(\varepsilon_2) \otimes
	\frac{ \bF_2[\eta, \lambda'_1, \mu] }
	{ (\eta \lambda'_1, (\lambda'_1)^2 =\eta^2 \mu) }
	\otimes \bF_2[t^{\pm1}]
\]
and differentials
\begin{alignat*}{2}
d^2(t^{-1}) &= \eta
	&\qquad\qquad
d^2(\varepsilon_2) &= t \mu \\
d^6(t^{-2}) &= t \lambda'_1
	&\qquad\qquad
d^6(t^{-1} \lambda'_1) &= t^2 (\lambda'_1)^2 = t^2 \eta^2 \mu \,,
\end{alignat*}
leading to
\[
\hat E^\infty(\bT) = \bF_2\{1, t^2 \lambda'_1, (t^2 \lambda'_1)^2,
	\lambda'_1\} \otimes \bF_2[t^{\pm4}] \,.
\]
Hence there is a preferred isomorphism
\[
\olV(2)_* \gr_{\mot}^* \TP(\ko)
	\cong \bF_2\{1, \eta, \eta^2, \lambda'_1\}
		\otimes \bF_2[t^{\pm4}] \,,
\]
where $1$, $\eta$, $\eta^2$, $\lambda'_1$ and~$t^{\pm4}$ in
bidegrees $(0,0)$, $(1,1)$, $(2,2)$, $(5,1)$ and~$(\mp8,0)$ are detected
by $1$, $t^2 \lambda'_1$, $(t^2 \lambda'_1)^2 = t^4 \eta^2 \mu$,
$\lambda'_1$ and~$t^{\pm4}$, respectively.
\end{theorem}

\begin{theorem}[Prismatic cohomology modulo $(2, \eta, v_1, v_2)$ of~$\ko$]
	\label{thm: prismatic mod eta}
The $\olV(2) \otimes \olC\eta$-homotopy $\bT$-Tate spectral
sequence~\eqref{eq: v2 Ceta Tate} is a module spectral sequence
over~\eqref{eq: v2 Tate}, with $E^2$-term
\[
\hat E^2(\bT) = \Lambda(\varepsilon_2) \otimes
	\Lambda(\lambda'_1)\{1, \lambda_2\} \otimes \bF_2[\mu]
	\otimes \bF_2[t^{\pm1}]
\]
and differentials
\begin{alignat*}{2}
d^2(\varepsilon_2) &= t \mu
	&\qquad\qquad
& \\
d^6(t^{-1}) &= t^2 \lambda'_1
	&\qquad\qquad
d^6(t^{-2}) &= t \lambda'_1 \\
d^6(t^{-1} \lambda_2) &= t^2 \lambda'_1 \lambda_2
	&\qquad\qquad
d^6(t^{-2} \lambda_2) &= t \lambda'_1 \lambda_2 \\
d^8(t^{-3}) &= t \lambda_2
	&\qquad\qquad
d^8(t^{-1} \lambda'_1) &= t^3 \lambda'_1 \lambda_2
\end{alignat*}
leading to
\[
\hat E^\infty(\bT) = \Lambda(\lambda'_1)\{1, \lambda_2\}
	\otimes \bF_2[t^{\pm4}] \,.
\]

Hence there is a preferred isomorphism
\[
(\olV(2) \otimes \olC\eta)_* \gr_{\mot}^* \TP(\ko)
	\cong \Lambda(\lambda'_1) \{1, \lambda_2\} \otimes \bF_2[t^{\pm4}] \,,
\]
where $1$, $\lambda'_1$, $\lambda_2$, $\lambda'_1 \lambda_2$
and~$t^{\pm4}$ in bidegrees $(0,0)$, $(5,1)$, $(7,1)$, $(12,2)$
and~$(\mp8,0)$ are detected by the classes in the $E^\infty$-term with the
same names.
\end{theorem}

The proofs of these theorems will occupy the remainder of this section.
By Corollary~\ref{cor: V2 homotopy}, we can identify the $E^2$-term
in~\eqref{eq: v2 Tate} as stated in Theorem~\ref{thm: prismatic}. 
%\[
%\hat E^2(\bT) = \Lambda(\varepsilon_2) \otimes \frac{
%\bF_2[\eta, \lambda'_1, \mu] }{ (\eta \lambda'_1,
%	(\lambda'_1)^2 = \eta^2 \mu) }
%	\otimes \bF_2[t^{\pm1}] \,.
%\]

\begin{proposition} \label{prop: d2 differentials}
The spectral sequence~\eqref{eq: v2 Tate} is multiplicative and has
differentials
\begin{alignat*}{3}
d^2(t^{-1}) &= \eta
	&\qquad\qquad
d^2(\eta) &= 0
	&\qquad\qquad
d^2(\lambda'_1) &= 0  \\
d^2(\varepsilon_2) &= t\mu
	&\qquad\qquad
\text{and}&
	&\qquad\qquad
d^2(\mu) &= t \eta \mu \,.
\end{alignat*}
Consequently, we can identify
\[
\hat E^4(\bT) = \bF_2\{1, t \lambda'_1, \lambda'_1, \eta^2 \mu\}
\otimes \bF_2[t^{\pm2}]
\]
with $\eta^2 \mu = \eta^3 \varepsilon_2$.
\end{proposition}

\begin{proof}
The first claim follows from Corollary~\ref{mult-TP}. Using the $\bT$-equivariant attaching maps
of the standard $\bT$-CW complex structure on $S^{\infty} = E\bT$,
we compute differentials
\[
d^2(t^{-1}) = \eta
	\qquad\text{and}\qquad
d^2(\eta) = 0
\]
as in \cite{Hes96}*{Lemma~1.4.2}.

We know that $d^2(t\lambda'_1) = 0$, because $t\lambda'_1$ detects~$\nu
\in \{h_{11}\}$ by Corollary~\ref{cor:detection}.  Consequently,
$d^2(\lambda'_1) = 0$ by the Leibniz rule and the fact that $\eta \cdot
\lambda'_1 = 0$.

To show that $d_2(\varepsilon_2) = t \mu$ we apply the graded analogue
of \cite{BR22}*{Proposition~2.3} (a variant of~\cite{BG95}*{Lemma~2.2})
to the smash product of cofiber sequences
\[
\xymatrix@C-1.1pc{
\Sigma^{4,-1} \olV(1) \otimes \THH(\ko) \ar[r] \ar[d]^-{v_2}
	& \Sigma^{6,0} \olV(1) \otimes F(S^3_+, \THH(\ko))^{\bT}
		\ar[r]^-{p} \ar[d]^-{v_2}
	& \Sigma^{6,0} \olV(1) \otimes \THH(\ko) \ar[d]^-{v_2} \\
\Sigma^{-2,-1} \olV(1) \otimes \THH(\ko) \ar[r] \ar[d]^-{i_2}
	& \olV(1) \otimes F(S^3_+, \THH(\ko))^{\bT} \ar[r]^-{p} \ar[d]^-{i_2}
	& \olV(1) \otimes \THH(\ko) \ar[d]^-{i_2} \\
\Sigma^{-2,-1} \olV(2) \otimes \THH(\ko) \ar[r]
	& \olV(2) \otimes F(S^3_+, \THH(\ko))^{\bT} \ar[r]^-{p}
	& \olV(2) \otimes \THH(\ko) \,.
}
\]
Starting in the upper right-hand corner, the unit $\Sigma^{6,0} 1$ lifts
over the connecting map~$j_2$ to $\varepsilon_2$ in the lower right-hand
corner and maps under the connecting map $\sigma$ to $d^2(\varepsilon_2)$
in the lower left-hand corner.  This is the same as the result of lifting
to $\Sigma^{6,0} 1$ at the top, mapping to $v_2$ in the center, lifting to
$t \mu$ at the left, and pushing to $t \mu$ in the lower left-hand corner.

In particular, $t \mu$ is a $d^2$-cycle, and the Leibniz rule implies
that $d^2(\mu) = t \eta \mu$.  This leads to the $E^4$-term shown in
Figure~\ref{Figure:prismatic}.
\end{proof}

\begin{figure}
\centering
\begin{tikzpicture}[radius=.07,scale=0.7]
\foreach \n in {-8, -6, ..., 8} \node [below] at (\n,-.5) {$\n$};
\foreach \s in {0,5,10} \node [left] at (-8.5,\s) {$\s$};
\clip (-8.5,-1) rectangle (8.5,10.5);
\draw [thin,color=lightgray] (-8,0) grid (8,10);
\draw [fill] (-8,0) circle; \node [below] at (-8,0) {$t^4$};
\draw [] (-4,0) circle; \node [below] at (-4,0) {$t^2$};
\draw [fill] (0,0) circle; \node [below] at (0,0) {$1$};
\draw [] (4,0) circle; \node [below] at (4,0) {$t^{-2}$};
\draw [fill] (8,0) circle; \node [below] at (8,0) {$t^{-4}$};
\draw [fill] (-8,5) circle;
\draw [] (-6,5) circle;
\draw [fill] (-4,5) circle;
\draw [] (-2,5) circle; \node [below] at (-2-.3,5) {$t \lambda'_1$};
\draw [fill] (0,5) circle; \node [below] at (0,5) {$\lambda'_1$};
\draw [] (2,5) circle; \node [below] at (2+.2,5) {$t^{-1} \lambda'_1$};
\draw [fill] (4,5) circle;
\draw [] (6,5) circle;
\draw [fill] (8,5) circle;
\draw [fill] (-8,10) circle;
\draw [] (-4,10) circle;
\draw [fill] (0,10) circle;
	\node [below] at (0,10) {$\eta^2\mu = (\lambda'_1)^2$};
\draw [] (4,10) circle;
\draw [fill] (8,10) circle;
\draw [->] (-4-.1,0+.1) -- (-10+.1,5-.1);
\draw [->] (4-.1,0+.1) -- (-2+.1,5-.1);
\draw [->] (12-.1,0+.1) -- (6+.1,5-.1);
\draw [->] (-6-.1,5+.1) -- (-12+.1,10-.1);
\draw [->] (2-.1,5+.1) -- (-4+.1,10-.1);
\draw [->] (10-.1,5+.1) -- (4+.1,10-.1);
\end{tikzpicture}
\caption{$\bT$-Tate $\hat E^4(\bT) \Longrightarrow
	\olV(2)_* \gr_{\mot}^* \TP(\ko)$ \label{Figure:prismatic}}
\end{figure}

\begin{remark}
In Figure~\ref{Figure:prismatic} (resp.~Figure~\ref{Figure:prismatic2}), the horizontal axis is the $\bT$-Tate filtration and the vertical axis is the stem of $\olV(2)_{*}\gr_{\mot}^{*}\THH(\ko)$ (resp.~$(\olV(2)\otimes \olC\eta)_{*}\gr_{\mot}^{*}\THH(\ko)$). The spectral sequence converges to the sum of the two gradings and the differentials follow the convention of the homological Serre spectral sequence; in other words, the $d^{r}$-differential goes from the coordinate $(a,b)$ to  $(a-r,b+r-1)$. The $d^{r}$-differentials also increase motivic filtration by one. Similar conventions apply for
the $\bT$-homotopy fixed point spectral sequences
in Figures~\ref{fig:TC-v2} and~\ref{Figure:T-homotopy-fixed-points-SS}.
\end{remark}

\begin{proposition} \label{prop: t4k perm cycles}
The classes $t^{\pm4}$ and $\lambda'_1$ are permanent cycles in the
spectral sequence~\eqref{eq: v2 Tate}.  Moreover, there are differentials
\[
d^6(t^{-2}) = t\lambda'_1
	\qquad\text{and}\qquad
d^6(t^{-1}\lambda'_1) = t^2 \eta^2 \mu \,.
\]
The spectral sequence~\eqref{eq: v2 Tate}
collapses at $\hat E^8(\bT) = \hat E^{\infty}(\bT)$.
\end{proposition}

\begin{proof}
We know that $t\lambda'_1$ is an infinite cycle in~\eqref{eq: v2 Tate},
because by Corollary~\ref{cor:detection}
it detects~$\nu \in \{h_{11}\}$ in $\olV(2)_* \gr_{\mot}^*
\TC^{-}(\ko)$.  By Corollary~\ref{mult-TP} and Proposition~\ref{C2Tate} there is a commutative
square of spectral sequences converging to
\[
\xymatrix{
\olV(2)_* \gr_{\mot}^* \TC^{-}(\ko) \ar[r]^-{F^h} \ar[d]_-{\can}
	& \olV(2)_* \gr_{\mot}^* \THH(\ko)^{hC_2} \ar[d]^-{\can} \\
\olV(2)_* \gr_{\mot}^* \TP(\ko) \ar[r]^-{F^t}
	& \olV(2)_* \gr_{\mot}^* \THH(\ko)^{tC_2} \,,
}
\]
with $E^2$-terms
\[
\xymatrix{
\olV(2)_* \gr_{\mot}^* \THH(\ko) \otimes \bF_2[t] \ar[r] \ar[d]
	& \olV(2)_* \gr_{\mot}^* \THH(\ko)
		\otimes \Lambda(u_1) \otimes \bF_2[t] \ar[d] \\
\olV(2)_* \gr_{\mot}^* \THH(\ko) \otimes \bF_2[t^{\pm1}] \ar[r]
	& \olV(2)_* \gr_{\mot}^* \THH(\ko)
	\otimes \Lambda(u_1) \otimes \bF_2[t^{\pm1}] \,.
}
\]
Recall that in the two right-hand cases, $u_1$ has (stem, motivic filtration)
bidegree $\|u_1 \| = (-1,-1)$.

We know that $\nu \in \{h_{11}\}$ maps to zero in $\olV(2)_* \gr_{\mot}^*
\THH(\ko)$, because the target is zero in the relevant bidegree by
Corollary~\ref{cor: V2 homotopy}.  A chase in the diagram
\[
\xymatrix@C+2pc{
& \bS \ar[d] \\
\TF(\ko) \ar[r]^-{F} \ar[d]_-{\Gamma}
	& \THH(\ko)^{C_2} \ar[r]^-{R} \ar[d]_-{\Gamma_1}
	& \THH(\ko) \ar[d]^{ \hat\Gamma_1 = \varphi_2} \\
\TC^{-}(\ko) \ar[r]^-{F^h} \ar[d]_-{q}
	& \THH(\ko)^{hC_2} \ar[r]^-{R^h = \can} \ar[d]_-{F^h}
	& \THH(\ko)^{tC_2} \\
F(S^3_+, \THH(\ko))^{\bT} \ar[r]^-{p} & \THH(\ko) \,,
}
\]
similar to the proof of \cite{AR02}*{Theorem~5.5}, shows that $t
\lambda'_1$ must be boundary in the $\olV(2)$-homotopy $C_2$-Tate
spectral sequence.  There is no earlier $C_2$-Tate $d^r$-differential
hitting $t \lambda'_1$, for $2 \le r \le 5$, since $\eta$ is an infinite
cycle.  Consequently,
\[
d^6(t^{-2}) = t \lambda'_1
\]
in the $\olV(2)$-homotopy $C_2$-Tate spectral sequence, and
therefore also in the $\olV(2)$-homotopy $\bT$-Tate spectral sequence,
cf.~Figure~\ref{Figure:prismatic}.  Here we use the fact that the map from the $\olV(2)$-homotopy $\bT$-Tate spectral sequence to the $\olV(2)$-homotopy $C_2$-Tate spectral sequence is injective in the relevant bidegrees. (In fact, as we will observe in Proposition~\ref{prop: C2 Tate computation}, the map of spectral sequences is injective at every page because the class $u_1$ is a permanent cycle in the target spectral sequence.) 
 
To complete the proof, we use the
Leibniz rule to deduce that $d^6(t^{-1} \lambda'_1) = d^6(t \lambda'_1
\cdot t^{-2}) = (t \lambda'_1)^2 = t^2 \eta^2 \mu$, that
$d^6(t^{-4}) = t\lambda'_1 \cdot t^{-2} + t^{-2} \cdot t\lambda'_1 = 0$,
and that $d^6(t^4) = 0$.  All later differentials are zero, because the
target groups are trivial.
\end{proof}

The following corollary immediately follows from Proposition~\ref{prop: t4k perm cycles}. 
\begin{corollary}\label{cor: low degrees TC- mod p v1 v2}
There are isomorphisms
$$
\olV(2)_n \gr_{\mot}^* \TP(\ko) \cong \begin{cases}
\bF_2\{1\} & \text{for $n=0$,} \\
\bF_2\{t^2 \lambda'_1\} & \text{for $n=1$,} \\
\bF_2\{t^4 \eta^2 \mu\} & \text{for $n=2$,} \\
0 & \text{for $n=3,4,6,7$,} \\
\bF_2\{\lambda'_1\} & \text{for $n=5$.}
\end{cases}
$$
These repeat $8$-periodically, via multiplication by~$t^{\pm4}$.
\qed
\end{corollary}

We now move towards computing the spectral sequence~\eqref{eq: v2
Ceta Tate}.  By Corollary~\ref{cor: A(1) coeff}, its $E^2$-term has the
form asserted in Theorem~\ref{thm: prismatic mod eta}.

\begin{remark} \label{rem: lack of Leibniz rule}
In general the differentials in~\eqref{eq: v2 Ceta Tate} do not satisfy
the Leibniz rule with respect to the (non-equivariant) algebra structure
from Corollary~\ref{cor: A(1) coeff}. 
This is commensurable with Remark~\ref{rem: A(1) mult}.
\end{remark}

\begin{proposition} \label{prop: some differentials}
The spectral sequence~\eqref{eq: v2 Ceta Tate} is a module over the
spectral sequence~\eqref{eq: v2 Tate}.  There are differentials
\[
d^2(\varepsilon_2) = t \mu
	\qquad,\qquad
d^6(t^{-2}) = t \lambda'_1
	\qquad\text{and}\qquad
d^8(t^{-3}) = t \lambda_2
\]
in~\eqref{eq: v2 Ceta Tate}, and multiplication by $t^{\pm4}$ and
$\lambda'_1$ commutes with all differentials in this spectral sequence.
\end{proposition}

\begin{proof}
The unit map $\olV(2) \to \olV(2) \otimes \olC\eta$ is a map of
$\olV(2)$-modules, so~\eqref{eq: v2 Ceta Tate} is a module spectral
sequence over~\eqref{eq: v2 Tate}, and the map from~\eqref{eq: v2 Tate}
to~\eqref{eq: v2 Ceta Tate} respects this module structure.  This implies
that multiplication by the infinite cycles $t^{\pm4}$ and $\lambda'_1$
will commute with each differential in~\eqref{eq: v2 Ceta Tate}.
The module structure also implies that $d^2(\varepsilon_2 \cdot 1) =
t \mu \cdot 1 - \varepsilon_2 \cdot d^2(1)$, so that $d^2(\varepsilon_2)
= t \mu$ in~\eqref{eq: v2 Ceta Tate}.  It follows that
\[
\hat E^4(\bT) = \Lambda(\lambda'_1, \lambda_2) \otimes \bF_2[t^{\pm1}] \,.
\]
See Figure~\ref{Figure:prismatic2}.

We showed in Corollary~\ref{cor:detection} that the images in
$(\olV(2) \otimes \olC\eta)_* \gr_{\mot}^* \TC^{-}(\ko)$ of $\nu \in
\{h_{11}\}$ and~$w$ are detected by $t\lambda'_1$ and~$t\lambda_2$,
respectively, so the same holds in $(\olV(2) \otimes \olC\eta)_*
\gr_{\mot}^* \THH(\ko)^{hC_2}$.  We also know that the images of $\nu$
and~$w$ are trivial in $(\olV(2) \otimes \olC\eta)_* \gr_{\mot}^*
\THH(\ko)$, so this means that classes detected by $t\lambda'_1$ and
$t\lambda_2$ map trivially to $(\olV(2) \otimes \olC\eta)_* \gr_{\mot}^*
\THH(\ko)^{tC_2}$.  It follows that $t\lambda'_1$ and $t\lambda_2$
must be hit by differentials in the $\olV(2) \otimes \olC\eta$-homotopy
$C_2$-Tate spectral sequence.  By examination of bidegrees, the only
possibility is that $d^6(t^{-2}) = t \lambda'_1$ and $d^8(t^{-3}) =
t \lambda_2$.  Since the map of spectral sequences converging to
\[
F^t \: (\olV(2) \otimes \olC\eta)_* \gr_{\mot}^* \TP(\ko)
\longto
(\olV(2) \otimes \olC\eta)_* \gr_{\mot}^* \THH(\ko)^{tC_2}
\]
is injective in the relevant bidegrees, we also have the stated
differentials in the spectral sequence~\eqref{eq: v2 Ceta Tate}.
\end{proof}

\begin{figure}
\centering
\begin{tikzpicture}[radius=.07,scale=0.7]
\foreach \n in {-8, -6, ..., 8} \node [below] at (\n,-.5) {$\n$};
\foreach \s in {0,5,7,12} \node [left] at (-8.5,\s) {$\s$};
\clip (-8.5,-1) rectangle (8.5,12.5);
\draw [thin,color=lightgray] (-8,0) grid (8,12);
\draw [fill] (-8,0) circle; \node [below] at (-8,0) {$t^4$};
\draw [] (-6,0) circle; \node [below] at (-6,0) {$t^3$};
\draw [] (-4,0) circle; \node [below] at (-4,0) {$t^2$};
\draw [] (-2,0) circle; \node [below] at (-2,0) {$t$};
\draw [fill] (0,0) circle; \node [below] at (0,0) {$1$};
\draw [] (2,0) circle; \node [below] at (2,0) {$t^{-1}$};
\draw [] (4,0) circle; \node [below] at (4,0) {$t^{-2}$};
\draw [] (6,0) circle; \node [below] at (6,0) {$t^{-3}$};
\draw [fill] (8,0) circle; \node [below] at (8,0) {$t^{-4}$};
\draw [fill] (-8,5) circle;
\draw [] (-6,5) circle;
\draw [] (-4,5) circle;
\draw [] (-2,5) circle;
\draw [fill] (0,5) circle; \node [left] at (0,5) {$\lambda'_1$};
\draw [] (2,5) circle;
\draw [] (4,5) circle;
\draw [] (6,5) circle;
\draw [fill] (8,5) circle;
\draw [fill] (-8,7) circle;
\draw [] (-6,7) circle;
\draw [] (-4,7) circle;
\draw [] (-2,7) circle;
\draw [fill] (0,7) circle; \node [right] at (0,7) {$\lambda_2$};
\draw [] (2,7) circle;
\draw [] (4,7) circle;
\draw [] (6,7) circle;
\draw [fill] (8,7) circle;
\draw [fill] (-8,12) circle;
\draw [] (-6,12) circle;
\draw [] (-4,12) circle;
\draw [] (-2,12) circle;
\draw [fill] (0,12) circle; \node [below] at (0,12) {$\lambda'_1 \lambda_2$};
\draw [] (2,12) circle;
\draw [] (4,12) circle;
\draw [] (6,12) circle;
\draw [fill] (8,12) circle;
\draw [->] (-6-.1,0+.1) -- (-12+.1,5-.1);
\draw [->] (-4-.1,0+.1) -- (-10+.1,5-.1);
\draw [->] (2-.1,0+.1) -- (-4+.1,5-.1);
\draw [->] (4-.1,0+.1) -- (-2+.1,5-.1);
\draw [->] (10-.1,0+.1) -- (4+.1,5-.1);
\draw [->] (12-.1,0+.1) -- (6+.1,5-.1);
\draw [->] (-6-.1,7+.1) -- (-12+.1,12-.1);
\draw [->] (-4-.1,7+.1) -- (-10+.1,12-.1);
\draw [->] (2-.1,7+.1) -- (-4+.1,12-.1);
\draw [->] (4-.1,7+.1) -- (-2+.1,12-.1);
\draw [->] (10-.1,7+.1) -- (4+.1,12-.1);
\draw [->] (12-.1,7+.1) -- (6+.1,12-.1);
\draw [->] (-2-.1,0+.1) -- (-10+.1,7-.1);
\draw [->] (6-.1,0+.1) -- (-2+.1,7-.1);
\draw [->] (14-.1,0+.1) -- (6+.1,7-.1);
\draw [->] (-6-.1,5+.1) -- (-14+.1,12-.1);
\draw [->] (2-.1,5+.1) -- (-6+.1,12-.1);
\draw [->] (10-.1,5+.1) -- (2+.1,12-.1);
\end{tikzpicture}
\caption{$\bT$-Tate $\hat E^4(\bT) \Longrightarrow (\olV(2) \otimes
	\olC\eta)_* \gr_{\mot}^* \TP(\ko)$ \label{Figure:prismatic2}}
\end{figure}

\begin{proposition} \label{prop: several differentials}
There are differentials
\begin{alignat*}{2}
d^6(t^{-1}) &= t^2 \lambda'_1
	&\qquad\qquad
d^6(t^{-1} \lambda_2) &= t^2 \lambda'_1 \lambda_2 \\
d^6(t^{-2} \lambda_2) &= t \lambda'_1 \lambda_2
	&\qquad\qquad
d^6(\lambda_2) &= 0
\end{alignat*}
in the spectral sequence~\eqref{eq: v2 Ceta Tate}.
\end{proposition}

\begin{proof}
We must have $d^6(t^{-3}) = 0$ in the spectral sequence~\eqref{eq: v2 Ceta
Tate}, since $t^{-3}$ survives to its $E^8$-term by Proposition~\ref{prop:
some differentials}.  By Proposition~\ref{prop: t4k perm cycles} we have
$d^6(t^{-2}) = t \lambda'_1$ in the spectral sequence~\eqref{eq: v2 Tate}.
Using the module structure of~\eqref{eq: v2 Ceta Tate} over~\eqref{eq:
v2 Tate}, we deduce that $d^6(t^{-5}) = d^6(t^{-2} \cdot t^{-3}) =
t \lambda'_1 \cdot t^{-3} + t^{-2} \cdot 0 = t^{-2} \lambda'_1$ and
$d^6(t^{-1}) = t^4 \cdot d^6(t^{-5}) = t^2 \lambda'_1$.

Since $t \lambda_2$ is a $d^8$-boundary by Proposition~\ref{prop: some
differentials} it must be a $d^6$-cycle, which implies that $d^6(t^{-1}
\lambda_2) = d^6(t^{-2} \cdot t \lambda_2) = t \lambda'_1 \cdot t
\lambda_2 + t^{-2} \cdot 0 = t^2 \lambda'_1 \lambda_2$.

The fact that $t \lambda_2$ is a $d^8$-boundary also implies that
$\lambda'_1 \cdot t \lambda_2 = t \lambda'_1 \lambda_2$ must be a
$d^r$-boundary for some $r \le 8$.  Since $t^{-3} \lambda'_1 = t^{-4}
\cdot t \lambda'_1$ is a $d^6$-boundary, it cannot be the source of this
$d^r$-differential, so the only remaining possibility is that $d^6(t^{-2}
\lambda_2) = t \lambda'_1 \lambda_2$.

Using the module structure over the spectral sequence~\eqref{eq: v2 Tate},
we can also conclude that $d^6(\lambda_2)=d^6(t^2 \cdot t^{-2} \lambda_2)
= t^5 \lambda'_1 \cdot t^{-2} \lambda_2 + t^2 \cdot t \lambda'_1 \lambda_2
= 0$.
\end{proof}

\begin{corollary} \label{cor: low degrees TC- mod p v1 v2 eta}
There are isomorphisms
\[
(\olV(2) \otimes \olC\eta)_n \gr_{\mot}^* \TP(\ko) \cong
	\begin{cases}
	\bF_2\{ 1 \} & \text{for $n=0$,} \\
	0 & \text{for $n = 1,2,3$,} \\
	\bF_2\{ t^4 \lambda'_1 \lambda_2 \} & \text{for $n=4$,} \\
	\bF_2\{ \lambda'_1 \} & \text{for $n=5$,}
	\end{cases}
\]
and these repeat $8$-periodically, via multiplication by $t^{\pm4}$.
\end{corollary}

\begin{proof}
This follows directly from Proposition~\ref{prop: some differentials}
and Proposition~\ref{prop: several differentials}.
\end{proof}

\begin{proposition} \label{prop: differential on t2lambda1prime} \,
%We have the following results:
\begin{enumerate}
\item[(a)]
The unit images of $\eta$ and~$\eta^2$ are detected by $t^2 \lambda'_1$
and $t^4 \eta^2 \mu$, respectively, in the spectral sequence~\eqref{eq:
v2 Tate}.

\item[(b)]
There is a nonzero differential
\[
d^8(t^{-1} \lambda'_1) = t^3 \lambda'_1 \lambda_2
\]
in the $\olV(2) \otimes \olC\eta$-homotopy $\bT$-Tate spectral
sequence~\eqref{eq: v2 Ceta Tate}.
Hence
\[
(\olV(2) \otimes \olC\eta)_n \gr_{\mot}^* \TP(\ko) \cong
	\begin{cases}
	0 & \text{for $n=6$,} \\
	\bF_2\{\lambda_2\} & \text{for $n=7$,}
	\end{cases}
\]
repeating $8$-periodically via multiplication by $t^{\pm4}$.
\end{enumerate}
\end{proposition}

\begin{proof} 
(a) The $\olV(2)$-module cofiber sequence~\eqref{eq:V2-V2Ceta-cofibseq}
induces a long exact sequence
\begin{multline*}
\dots \longto
(\olV(2) \otimes \olC\eta)_{n+2} \gr_{\mot}^{*+1} \TP(\ko)
	\overset{j}\longto
\olV(2)_n \gr_{\mot}^* \TP(\ko) \\
	\overset{\eta}\longto
\olV(2)_{n+1} \gr_{\mot}^{*+1} \TP(\ko)
	\overset{i}\longto
(\olV(2) \otimes \olC\eta)_{n+1} \gr_{\mot}^{*+1} \TP(\ko)
\longto \dots \,.
\end{multline*}

By the cases $n \in \{1,2\}$ of Corollary~\ref{cor: low degrees TC-
mod p v1 v2 eta}, and exactness, it %also 
follows that $\eta$ and~$\eta^2$
generate $\olV(2)_n \gr_{\mot}^* \TP(\ko) \cong \bF_2$ for $n=1$ and~$2$,
hence are detected by the only classes in stems $1$ and~$2$, namely $t^2
\lambda'_1$ and $t^4 \eta^2 \mu$, in the $E^\infty$-term of~\eqref{eq:
v2 Tate}.

(b) By Corollary~\ref{cor: low degrees TC- mod p v1 v2} and exactness, it follows that $(\olV(2) \otimes \olC\eta)_n
\gr_{\mot}^* \TP(\ko)$ is~$0$ for $n=6$ and $\bF_2$ for $n=7$.
Hence $t^3 \lambda'_1 \lambda_2$ in stem~$6$ cannot survive to the
$E^\infty$-term of~\eqref{eq: v2 Ceta Tate}, and since $d^6(\lambda_2)
= 0$ by Proposition~\ref{prop: several differentials} the only
possible source of a differential killing it is $t^{-1} \lambda'_1$.
Hence $d^8(t^{-1} \lambda'_1) = t^3 \lambda'_1 \lambda_2$, and the lone
surviving class in stem~$7$ of the $E^\infty$-term of~\eqref{eq: v2
Ceta Tate} is $\lambda_2$.

In summary, the homomorphisms in the long exact sequence above are given,
at the level of detecting classes, by $\eta \: 1 \mapsto t^2 \lambda'_1$,
$\eta \: t^2 \lambda'_1 \mapsto t^4 \eta^2 \mu$, $i \: 1 \mapsto 1$,
$i \: \lambda'_1 \mapsto \lambda'_1$, $j \: t^4 \lambda'_1 \lambda_2
\mapsto t^4 \eta^2 \mu$ and $j \: \lambda_2 \mapsto \lambda'_1$.
\end{proof}

%\begin{corollary} \label{cor: mod 2 v1 v2 homotopy of THH of ko}
%We have a preferred isomorphism of bigraded $\bF_2$-algebras
%\[
%\olV(1)_* \gr_{\mot}^* \THH(\ko)
%	\cong \frac{ \bF_2[\eta, \lambda'_1, \mu] }
%	{ (\eta \lambda'_1, (\lambda'_1)^2 = \eta^2 \mu) } \,.
%\]
%\qed
%\end{corollary}

%We can now prove Theorems~\ref{thm: prismatic} and~\ref{thm: prismatic smod eta}.

\begin{proof}[Proof of Theorem~\ref{thm: prismatic}]
This theorem follows by combining the
differentials from Propositions~\ref{prop: d2 differentials} and~\ref{prop: t4k perm cycles} with the detection result from Proposition~\ref{prop: differential on t2lambda1prime}(a).
\end{proof}

\begin{proof}[Proof of Theorem~\ref{thm: prismatic mod eta}]
This theorem follows by combining the differentials from Propositions \ref{prop: some differentials}, \ref{prop: several differentials} and~\ref{prop: differential on t2lambda1prime}(b).
\end{proof}

\section{Syntomic cohomology} \label{sec:syntomic}

We shall now calculate the syntomic cohomology modulo~$(2, v_1)$
and~$(2, \eta, v_1)$ of~$\ko$ (cf.~Definition~\ref{def:syntomic}).
We first carry out these computations in $\olV(2)$- and $\olV(2) \otimes
\olC\eta$-homotopy, and then use $v_2$-Bockstein spectral sequences to
lift the results to $\olV(1)$- and $\olA(1)$-homotopy.

%By restricting the $\bT$-Tate spectral sequences \eqref{eq: v2 Tate}
%and~\eqref{eq: v2 Ceta Tate} to the second quadrant, 
By Corollary~\ref{mult-TP}, there is a $\olV(2)$-homotopy $\bT$-homotopy fixed point spectral sequence
\begin{equation} \label{eq:TC-v2}
\begin{aligned}
E^2(\bT) &= \olV(2)_* \gr_{\mot}^* \THH(\ko) \, [t] \\
    &= \Lambda(\varepsilon_2) \otimes \frac{ \bF_2[\eta, \lambda'_1, \mu] }
	{ (\eta \lambda'_1, (\lambda'_1)^2 =\eta^2 \mu) } \otimes \bF_2[t] \\
    &\Longrightarrow \olV(2)_* \gr_{\mot}^* \TC^{-}(\ko)
\end{aligned}
\end{equation}
and a $\olV(2) \otimes \olC\eta$-homotopy
$\bT$-homotopy fixed point spectral sequence
\begin{equation} \label{eq:TC-v2Ceta}
\begin{aligned}
E^2(\bT) &= (\olV(2) \otimes \olC\eta)_* \gr_{\mot}^* \THH(\ko) \, [t] \\
    &= \Lambda(\varepsilon_2) \otimes \Lambda(\lambda'_1)\{1, \lambda_2\}
	\otimes \bF_2[\mu] \otimes \bF_2[t] \\
    &\Longrightarrow (\olV(2) \otimes \olC\eta)_* \gr_{\mot}^* \TC^{-}(\ko)
\,.
\end{aligned}
\end{equation}
The former is an algebra spectral sequence, and the latter is
a module spectral sequence over it by Corollary~\ref{mult-TP}.  They can be reindexed as
cohomologically graded $t$-Bockstein spectral sequences, but the
current indexing is the one inherited from the homologically graded
$C_2$- and $\bT$-Tate spectral sequences.  See Figures~\ref{fig:TC-v2}
and~\ref{Figure:T-homotopy-fixed-points-SS}.

\begin{proposition} \label{prop: 5.1}
There is an isomorphism
\begin{align*}
\olV(2)_* \gr_{\mot}^* \TC^{-}(\ko)
&\cong \bF_2[t^4] \{1, t^2 \lambda'_1, \lambda'_1, (\lambda'_1)^2\} \\
&\qquad\oplus \bF_2\{t \lambda'_1, (t \lambda'_1)^2\} \\
&\qquad\oplus \bF_2[\eta]\{\eta, \eta^4 \varepsilon_2\} \\
&\qquad\oplus
    \bF_2[\bar\mu]\{\bar\mu, \eta\bar\mu, \eta^2\bar\mu, \lambda'_1\mu\}
\end{align*}
with $\bar\mu = \mu + \eta \varepsilon_2$, where $(\lambda'_1)^2 = \eta^2
\mu \ne \eta^2 \bar\mu$, $\eta \cdot \eta^2\bar\mu = \eta^4 \varepsilon_2$
and $\bar\mu^2 = \mu^2$.
\end{proposition}

\begin{proof}
The map of spectral sequences induced by $\can \: \TC^{-}(\ko)
\to \TP(\ko)$ is given at the $E^2$-terms by inverting $t$, so the
differentials in~\eqref{eq: v2 Tate} from Theorem~\ref{thm: prismatic}
lift to differentials
\begin{alignat*}{2}
d^2(t) &= t^2 \eta
	&\qquad\qquad
d^2(\varepsilon_2) &= t \mu \\
d^6(t^2) &= t^5 \lambda'_1
	&\qquad\qquad
d^6(t^3 \lambda'_1) &= t^6 (\lambda'_1)^2
\end{alignat*}
in~\eqref{eq:TC-v2}.  Moreover, $\eta$, $\lambda'_1$ and $t\mu$ are
infinite cycles.  Some bookkeeping shows that
\begin{align*}
E^4(\bT) &= \bF_2[t^2] \{1, t \lambda'_1, \lambda'_1, (\lambda'_1)^2 \} \\
    &\qquad\oplus \bF_2[\eta] \{\eta, \eta^4 \varepsilon_2\} \\
    &\qquad\oplus \bF_2[\bar\mu] \{ \bar\mu, \eta\bar\mu,
	\eta^2\bar\mu, \lambda'_1\mu\}
\end{align*}
with $\bar\mu = \mu + \eta\varepsilon_2$ and $\eta \cdot \eta^2\bar\mu =
\eta^4\varepsilon_2$, and
\begin{align*}
E^8(\bT) = E^\infty(\bT)
	&= \bF_2[t^4] \{1, t^2 \lambda'_1, \lambda'_1, (\lambda'_1)^2\} \\
&\qquad\oplus \bF_2 \{t \lambda'_1, (t \lambda'_1)^2\} \\
&\qquad\oplus \bF_2[\eta] \{\eta, \eta^4 \varepsilon_2\} \\
&\qquad\oplus \bF_2[\bar\mu] \{ \bar\mu, \eta\bar\mu, \eta^2\bar\mu,
	\lambda'_1\mu\}
	\,. 
\end{align*}
\end{proof}

\begin{figure}
\centering
\resizebox{!}{.93\textheight}{ \input{V2TCminusko.inp} }
\caption{$\bT$-homotopy fixed point spectral sequence converging to
	$\olV(2)_* \gr_{\mot}^* \TC^{-}(\ko)$ \label{fig:TC-v2}}
\end{figure}

\begin{proposition} \label{prop: TN}
There is an isomorphism
\begin{align*}
(\olV(2) \otimes \olC\eta)_* \gr_{\mot}^* \TC^{-}(\ko)
	&\cong \frac{ \bF_2[t^4, \mu] }{ (t^4 \mu) }
	\otimes \Lambda(\lambda'_1)\{1, \lambda_2\} \\
&\qquad\oplus \bF_2 \{t^2 \lambda'_1, t \lambda'_1,
	t \lambda_2, t^3 \lambda'_1 \lambda_2, t^2 \lambda'_1 \lambda_2,
	t \lambda'_1 \lambda_2 \} \,.
\end{align*}
\end{proposition}

\begin{figure}
\centering
\resizebox{!}{.93\textheight}{ \input{V2CetaTCminusko.inp} }
\caption{$\bT$-homotopy fixed point spectral sequence converging
	to $(\olV(2) \otimes \olC\eta)_* \gr_{\mot}^* \TC^{-}(\ko)$
	\label{Figure:T-homotopy-fixed-points-SS}}
\end{figure}

\begin{proof}
The differentials in~\eqref{eq: v2 Ceta Tate} from Theorem~\ref{thm:
prismatic mod eta} lift over the canonical map to differentials
$d^2(\varepsilon_2) = t \mu$ (repeating $t$-periodically) and
\begin{alignat*}{2}
d^6(t^2) &= t^5 \lambda'_1
	&\qquad\qquad
d^6(t^3) &= t^6 \lambda'_1 \\
d^6(t^2 \lambda_2) &= t^5 \lambda'_1 \lambda_2
	&\qquad\qquad
d^6(t^3 \lambda_2) &= t^6 \lambda'_1 \lambda_2 \\
d^8(t) &= t^5 \lambda_2
	&\qquad\qquad
d^8(t^3 \lambda'_1) &= t^7 \lambda'_1 \lambda_2
\end{alignat*}
(repeating $t^4$-periodically) in~\eqref{eq:TC-v2Ceta}.  It follows
that
\[
E^4(\bT) = \bF_2[t, \mu]/(t \mu)
	\otimes \Lambda(\lambda'_1)\{1, \lambda_2\} \,,
\]
and $E^{10}(\bT) = E^\infty(\bT)$ is equal to
\[
\frac{ \bF_2[t^4, \mu] }{ (t^4 \mu) }
	\otimes \Lambda(\lambda'_1)\{1, \lambda_2\}
\, \oplus \, \bF_2 \{t^2 \lambda'_1, t \lambda'_1,
	t \lambda_2, t^3 \lambda'_1 \lambda_2, t^2 \lambda'_1 \lambda_2,
	t \lambda'_1 \lambda_2 \}
	\,. \qedhere
\]
\end{proof}

As discussed in the proof of Proposition~\ref{prop: t4k perm cycles},
there is a $\olV(2)$-homotopy $C_2$-Tate spectral sequence
\begin{equation} \label{eq:5.3}
\begin{aligned}
\hat E^2(C_2) &= \olV(2)_* \gr_{\mot}^* \THH(\ko)
	\otimes \Lambda(u_1) \otimes \bF_2[t^{\pm1}] \\
&= \Lambda(\varepsilon_2) \otimes \frac{ \bF_2[\eta, \lambda'_1, \mu] }
	{ (\eta\lambda'_1, (\lambda'_1)^2 = \eta^2\mu) }
	\otimes \Lambda(u_1) \otimes \bF_2[t^{\pm1}] \\
&\Longrightarrow \olV(2)_* \gr_{\mot}^* \THH(\ko)^{tC_2} \,.
\end{aligned}
\end{equation}
Similarly, we have a $\olV(2) \otimes \olC\eta$-homotopy $C_2$-Tate
spectral sequence
\begin{equation} \label{eq:5.4}
\begin{aligned}
\hat E^2(C_2) &= (\olV(2) \otimes \olC\eta)_* \gr_{\mot}^* \THH(\ko)
	\otimes \Lambda(u_1) \otimes \bF_2[t^{\pm1}] \\
&= \Lambda(\varepsilon_2) \otimes \Lambda(\lambda'_1) \{1, \lambda_2\}
	\otimes \bF_2[\mu] \otimes \Lambda(u_1) \otimes \bF_2[t^{\pm1}] \\
&\Longrightarrow (\olV(2) \otimes \olC\eta)_*
	\gr_{\mot}^* \THH(\ko)^{tC_2} \,.
\end{aligned}
\end{equation}
There is a map~$F^t$ of algebra spectral sequences from~\eqref{eq:
v2 Tate} to~\eqref{eq:5.3}, and~\eqref{eq:5.4} is a module spectral
sequence over~\eqref{eq:5.3}.

\begin{proposition} \label{prop: C2 Tate computation}
There is an isomorphism
\[
\olV(2)_* \gr_{\mot}^* \THH(\ko)^{tC_2}
	\cong \bF_2\{1, \eta, \eta^2, \lambda'_1\}
	\otimes \Lambda(u_1) \otimes \bF_2[t^{\pm4}] \,,
\]
where $\eta$, $\eta^2$ and~$\lambda'_1$ are detected by $t^2 \lambda'_1$,
$(t^2 \lambda'_1)^2$ and~$\lambda'_1$, respectively.  Under this
correspondence, the cyclotomic structure map
\[
\varphi_2 \: \olV(2)_* \gr_{\mot}^* \THH(\ko)
	\longto \olV(2)_* \gr_{\mot}^* \THH(\ko)^{tC_2}
\]
is given by $\varepsilon_2 \mapsto u_1 t^{-4}$, $\eta \mapsto \eta$,
$\lambda'_1 \mapsto \lambda'_1$ and $\mu \mapsto t^{-4}$, hence can be
identified with the localization homomorphism
\[
\olV(2)_* \gr_{\mot}^* \THH(\ko)
	\longto \mu^{-1} \olV(2)_* \gr_{\mot}^* \THH(\ko)
\]
that inverts~$\mu$.
\end{proposition}

\begin{proof}
We shall use naturality with respect to the complexification map $c \:
\ko \to \ku$ to access the cyclotomic structure map~$\varphi_2$ for~$\ko$
and the differentials in~\eqref{eq:5.3}, so we first review the results
of Hahn--Wilson about the complex case.

By~\cite{HW22}*{Proposition~6.1.6}, we have $\olV(1)_* \gr_{\mot}^*
\THH(\ku) \cong \Lambda(\lambda_1, \lambda_2) \otimes \bF_2[\mu]$, with
$\|\lambda_1\| = (3,1)$, $\|\lambda_2\| = (7,1)$ and~$\|\mu\| = (8,0)$.
Hence $\olV(2)_* \gr_{\mot}^* \THH(\ku) \cong \Lambda(\varepsilon_2,
\lambda_1, \lambda_2) \otimes \bF_2[\mu]$, with $\|\varepsilon_2\| =
(7,-1)$.  The $\olV(1)$-homotopy $C_2$-Tate spectral sequence
\begin{align*}
\hat E^2(C_2, \ku) &= \olV(1)_* \gr_{\mot}^* \THH(\ku)
	\otimes \Lambda(u_1) \otimes \bF_2[t^{\pm1}] \\
	&\Longrightarrow \olV(1)_* \gr_{\mot}^* \THH(\ku)^{tC_2}
\end{align*}
is an algebra spectral sequence with differentials $d^4(t^{-1}) =
t \lambda_1$, $d^8(t^{-2}) = t^2 \lambda_2$ and $d^9(u_1 t^{-4}) =
t \mu$, leaving $\hat E^\infty(C_2, \ku) = \Lambda(\lambda_1, \lambda_2)
\otimes \bF_2[t^{\pm4}]$, and the cyclotomic structure map is given in
$\olV(1)$-homotopy by $\lambda_1 \mapsto \lambda_1$, $\lambda_2 \mapsto
\lambda_2$ and $\mu \mapsto t^{-4}$, hence is identified with the ring
homomorphism that inverts~$\mu$, as in~\cite{HRW}*{Theorem~6.1.4}.

It follows that the $\olV(2)$-homotopy $C_2$-Tate spectral sequence
\begin{align*}
\hat E^2(C_2, \ku) &= \olV(2)_* \gr_{\mot}^* \THH(\ku)
	\otimes \Lambda(u_1) \otimes \bF_2[t^{\pm1}] \\
	&\Longrightarrow \olV(2)_* \gr_{\mot}^* \THH(\ku)^{tC_2}
\end{align*}
has differentials $d^2(\varepsilon_2) = t \mu$, $d^4(t^{-1}) =
t \lambda_1$ and $d^8(t^{-2}) = t^2 \lambda_2$, leaving behind
the $E^\infty$-term $\hat E^\infty(C_2, \ku) = \Lambda(\lambda_1,
\lambda_2) \otimes \Lambda(u_1) \otimes \bF_2[t^{\pm4}]$.  By exactness
of localization, the cyclotomic structure map
\[
\varphi_2 \: \olV(2)_* \gr_{\mot}^* \THH(\ku)
	\longto \olV(2)_* \gr_{\mot}^* \THH(\ku)^{tC_2}
\]
must also agree with the ring homomorphism that inverts~$\mu$.
Hence it is given on $\lambda_1$, $\lambda_2$ and~$\mu$ as
in the $\olV(1)$-case, while $\varphi_2(\varepsilon_2)$ can
only be detected by $u_1 t^{-4}$.

Next we appeal to naturality.  We saw in the proof of Lemma~\ref{lem:
THH computation} that
\[
c \: \olA(1)_* \gr_{\mot}^* \THH(\ko)
	\longto \olA(1)_* \gr_{\mot}^* \THH(\ku)
\]
is given by $\lambda'_1 \mapsto 0$, $\lambda_2 \mapsto \lambda_2$ and $\mu
\mapsto \mu$.  It follows by naturality with respect to $i \: \olV(1)
\to \olA(1)$ that $c \: \olV(1)_* \gr_{\mot}^* \THH(\ko) \to \olV(1)_*
\gr_{\mot}^* \THH(\ku)$ is given by $\eta \mapsto 0$, $\lambda'_1 \mapsto
0$ and $\mu \mapsto \mu$, and similarly with $\olV(2)$-coefficients,
where also $\varepsilon_2 \mapsto \varepsilon_2$.  
Chasing $\varepsilon_2$ and~$\mu$ around the commutative diagram
\[
\xymatrix{
\olV(2)_* \gr_{\mot}^* \THH(\ko) \ar[r]^-{\varphi_2} \ar[d]_-{c}
	& \olV(2)_* \gr_{\mot}^* \THH(\ko)^{tC_2} \ar[d]^-{c} \\
\olV(2)_* \gr_{\mot}^* \THH(\ku) \ar[r]^-{\varphi_2}
	& \olV(2)_* \gr_{\mot}^* \THH(\ku)^{tC_2} \,,
}
\]
we see that $\varphi_2(\varepsilon_2)$ and~$\varphi_2(\mu)$ in the real
case must be detected in the same, or higher, Tate filtration as the
detecting classes in the complex case, namely $u_1 t^{-4}$ and~$t^{-4}$.
There are no classes of higher Tate filtration in the same total degrees,
so the only possibility is that $\varphi_2(\varepsilon_2)$ is detected
by $u_1 t^{-4}$ and $\varphi_2(\mu)$ is detected by $t^{-4}$, also in
the real case.

In particular, this shows that $u_1 = t^4 \cdot u_1 t^{-4}$ is a permanent
cycle in the spectral sequence~\eqref{eq:5.3}.  By naturality with
respect to the map~$F^t$, we deduce that we have the same differentials
in the $C_2$-Tate spectral sequence as in the $\bT$-Tate spectral
sequence~\eqref{eq: v2 Tate}, listed in Theorem~\ref{thm: prismatic}.
This leaves
\[
\hat E^\infty(C_2) = \bF_2\{1, t^2 \lambda'_1, (t^2 \lambda'_1)^2,
	\lambda'_1\} \otimes \Lambda(u_1) \otimes \bF_2[t^{\pm4}] \,.
\]
It is clear that the $\bS$-algebra map $\varphi_2$ takes $\eta$ to
$\eta$, which we saw is detected by $t^2 \lambda'_1$.  The relation
$(\lambda'_1)^2 = \eta^2 \mu$ now shows that $\varphi_2(\lambda'_1)^2$
must be detected by $(t^2 \lambda'_1)^2 \cdot t^{-4} = (\lambda'_1)^2
\ne 0$, which can only happen if $\varphi_2(\lambda'_1)$ is detected
by~$\lambda'_1$.  Summarizing, we have an isomorphism
\[
\olV(2)_* \gr_{\mot}^* \THH(\ko)^{tC_2} \cong
	\bF_2\{1, \eta, \eta^2, \lambda'_1\} \otimes \Lambda(u_1)
	\otimes \bF_2[t^{\pm4}] \,,
\]
where $\varphi_2$ is given by $\varepsilon_2 \mapsto u_1 t^{-4}$, $\eta
\mapsto \eta$, $\lambda'_1 \mapsto \lambda'_1$ and $\mu \mapsto t^{-4}$.
The claim about localization then amounts to the isomorphism
\[
\mu^{-1} \frac{ \bF_2[\eta, \lambda'_1, \mu] }
	{ (\eta \lambda'_1, (\lambda'_1)^2 = \eta^2\mu) }
\cong \bF_2\{1, \eta, \eta^2, \lambda'_1\} \otimes \bF_2[\mu^{\pm1}]
	\,.  \qedhere
\]
\end{proof}

\begin{proposition} \label{prop:5.4}
There is an isomorphism
\[
(\olV(2) \otimes \olC\eta)_* \gr_{\mot}^* \THH(\ko)^{tC_2}
	\cong \Lambda(\lambda'_1) \{1, \lambda_2\}
	\otimes \Lambda(u_1) \otimes \bF_2[t^{\pm4}] \,,
\]
where $1$, $\lambda'_1$, $\lambda_2$, $\lambda'_1 \lambda_2$ and $t^{\pm 4}$ 
are detected by classes with the same names.  Under this
correspondence, the cyclotomic structure map
\[
\varphi_2 \: (\olV(2) \otimes \olC\eta)_* \gr_{\mot}^* \THH(\ko) \longto (\olV(2) \otimes \olC\eta)_* \gr_{\mot}^* \THH(\ko)^{tC_2}
\]
is given by $\varepsilon_2 \mapsto u_1 t^{-4}$, $\lambda'_1 \mapsto
\lambda'_1$, $\lambda_2 \mapsto \lambda_2$ and $\mu \mapsto t^{-4}$,
hence can be identified with the localization homomorphism
\[
(\olV(2) \otimes \olC\eta)_* \gr_{\mot}^* \THH(\ko) \longto \mu^{-1} (\olV(2) \otimes \olC\eta)_* \gr_{\mot}^* \THH(\ko)
\]
that inverts~$\mu$.
\end{proposition}

\begin{proof}
Naturality with respect to~$i \: \olV(2) \to \olV(2) \otimes
\olC\eta$ shows that $u_1$ is a permanent cycle in the spectral
sequence~\eqref{eq:5.4}.  When combined with the differentials
in~\eqref{eq: v2 Ceta Tate}, listed in Theorem~\ref{thm: prismatic mod
eta}, this shows that
\[
\hat E^4(C_2) = \Lambda(\lambda'_1) \{1, \lambda_2\}
	\otimes \Lambda(u_1) \otimes \bF_2[t^{\pm1}]
\]
and
\[
\hat E^\infty(C_2) = \Lambda(\lambda'_1) \{1, \lambda_2\}
	\otimes \Lambda(u_1) \otimes \bF_2[t^{\pm4}] \,.
\]
The detection results then follow from those in Theorem~\ref{thm:
prismatic mod eta}.  The evaluation of $\varphi_2$ on $\varepsilon_2$,
$\lambda'_1$ and~$\mu$ follows from that in Proposition~\ref{prop:
C2 Tate computation} by comparison along the same map~$i$.

To show that $\varphi_2(\lambda_2)$ is detected by $\lambda_2$, we note
that by naturality with respect to $j_2 \: \olV(2) \otimes \olC\eta
\to \Sigma^{7,-1} \olA(1)$ it cannot be detected by $u_1 t^{-4}$.
On the other hand, by naturality along $j \: \olV(2) \otimes \olC\eta
\to \Sigma^{2,0} \olV(2)$ it is nonzero, since we saw in the proof
of Proposition~\ref{prop: differential on t2lambda1prime} that $j \:
\lambda_2 \mapsto \Sigma^{2,0} \lambda'_1$.  Hence $\varphi_2(\lambda_2)$
must be the class detected by~$\lambda_2$.
\end{proof}

\begin{remark}
The computations in Theorems~\ref{thm: prismatic} and~\ref{thm:
prismatic mod eta} and Propositions~\ref{prop: C2 Tate computation}
and~\ref{prop:5.4} are consistent with isomorphisms
\begin{align*}
\olV(2)_* \gr_{\mot}^* \TP(\ko)
	&\cong \olV(1)_* \gr_{\mot}^* \THH(\ko)^{tC_2} \\
(\olV(2) \otimes \olC\eta )_* \gr_{\mot}^* \TP(\ko)
	&\cong \olA(1)_* \gr_{\mot}^* \THH(\ko)^{tC_2} \,,
\end{align*}
in analogy with \cite{HRW}*{Theorem~6.4.1}.
\end{remark}

\begin{theorem}[Syntomic cohomology modulo $(2, v_1, v_2)$ of~$\ko$]
	\label{thm:TC-V2}
We have an algebra 
isomorphism
\[
 \olV(2)_* \gr_{\mot}^* \TC(\ko)
	\cong \bF_2[\eta] \{ 1, \eta^4 \varepsilon_2 \}
	\oplus \bF_2\{\partial, \nu, \lambda'_1,
		\partial \lambda'_1, \nu^2, (\lambda'_1)^2\} \,,
\]
with generators in bidegrees $\|\partial\| = (-1,1)$, $\|\eta\| =
(1,1)$, $\|\nu\| = (3,1)$, $\|\lambda'_1\| = (5,1)$ and $\|\eta^4
\varepsilon_2\| = (11,3)$.  See Figure~\ref{fig:V2TCko} for a view of
the algebra structure of the right-hand side.
\end{theorem}

\begin{proof}
To calculate the effect in $\olV(2)$-homotopy of $\can \: \TC^{-}(\ko)
\to \TP(\ko)$, we use the map of spectral sequences from~\eqref{eq:TC-v2}
to~\eqref{eq: v2 Tate}, described in Proposition~\ref{prop: 5.1} and
Theorem~\ref{thm: prismatic}, given at the $E^2$-terms by inverting~$t$.
To calculate the effect of $\varphi_2^{h\bT} \: \TC^{-}(\ko) \to
(\THH(\ko)^{tC_2})^{h\bT}$ we appeal to Corollary~\ref{mult-TP} and Proposition~\ref{prop: C2 Tate
computation} to see that there is a $\mu$-inverted $\bT$-homotopy fixed point spectral
sequence
\begin{equation} \label{eq:5.5}
\begin{aligned}
\mu^{-1} E^2(\bT) &= \olV(2)_* \gr_{\mot}^* \THH(\ko)^{tC_2} \, [t] \\
	&= \Lambda(\varepsilon_2) \otimes \bF_2\{1, \eta, \eta^2, \lambda'_1\}
	\otimes \bF_2[\mu^{\pm1}] \otimes \bF_2[t] \\
	&\Longrightarrow \olV(2)_* \gr_{\mot}^* (\THH(\ko)^{tC_2})^{h\bT}
\,,
\end{aligned}
\end{equation}
and $\varphi_2^{h\bT}$ is calculated by the map of spectral sequences
from~\eqref{eq:TC-v2} to~\eqref{eq:5.5} that is given at the $E^2$-terms
by inverting~$\mu$.  The differentials
\[
d^2(\varepsilon_2) = t \mu
	\qquad\text{and}\qquad
d^2(\mu) = t \eta \mu
\]
carry over from the proof of Proposition~\ref{prop: 5.1}, leaving
\[
\mu^{-1} E^4(\bT) = \mu^{-1} E^\infty(\bT)
	= \bF_2\{1, \eta, \eta^2, \lambda'_1\}
	\otimes \bF_2[\bar\mu^{\pm1}] \,,
\]
concentrated on the vertical axis.  As before, $\bar\mu
= \mu + \eta\varepsilon_2$.  

As recalled in Definition~\ref{GM}, we know a priori that the natural
map 
\[
G \: \gr_\mot^* \TP(\ko)
	\overset{\simeq}\longto \gr_\mot^* (\THH(\ko)^{tC_2})^{h\bT}
\]
is an equivalence. 
%We know a priori that $G \:
%\TP(\ko) \to (\THH(\ko)^{tC_2})^{h\bT}$ is an equivalence, by
%\cite{BBLNR14}*{Proposition~3.8}, (cf.~\cite{NS18}*{Lemma~II.4.2}).
The $\olV(2)$-homotopy isomorphism
\[
\bF_2\{1, \eta, \eta^2, \lambda'_1\} \otimes \bF_2[t^{\pm4}]
\overset{\cong}\longto
\bF_2\{1, \eta, \eta^2, \lambda'_1\} \otimes \bF_2[\bar\mu^{\pm1}]
\]
induced by the equivalence~$G$ can then only be given by $\eta
\mapsto \eta$, $\lambda'_1 \mapsto \lambda'_1$ and $t^{\pm4} \mapsto
\bar\mu^{\mp1}$.

We claim that the map $\can - \varphi$ (which is short for
$G \circ \can - \varphi_2^{h\bT}$) induces isomorphisms
\begin{align}
\label{iso 1} \bF_2[t^4] \{t^4\}
	\otimes \bF_2\{1, t^2 \lambda'_1, (t^2 \lambda'_1)^2, \lambda'_1\}
&\overset{\cong}\longto \bF_2[\bar\mu^{-1}]\{\bar\mu^{-1}\}
	\otimes \bF_2\{1, \eta, \eta^2, \lambda'_1) \\
 \label{iso 2} \bF_2[\bar\mu] \{\bar\mu\}
	\otimes \bF_2\{1, \eta, \eta^2, \lambda'_1\}
&\overset{\cong}\longto \bF_2[\bar\mu] \{\bar\mu\}
	\otimes \bF_2\{1, \eta, \eta^2, \lambda'_1\} \\
\label{iso 3} \bF_2\{t^2 \lambda'_1, (t^2 \lambda'_1)^2\}
&\overset{\cong}\longto \bF_2\{\eta, \eta^2\}
\end{align}
and the zero homomorphism
\begin{align}\label{eq:zero-homomorphism}
\bF_2[\eta]\{1, \eta^4 \varepsilon_2\}
	\oplus \bF_2\{t \lambda'_1, \lambda'_1,
		(t \lambda'_1)^2, (\lambda'_1)^2\}
\overset{0}\longto \bF_2\{1, \lambda'_1\} \,.
\end{align}
The isomorphism~\eqref{iso 1} occurs in horizontal degrees %(= filtrations) 
where inverting~$t$ (or~$t^4$) is an isomorphism, and $\varphi_2^{h\bT}$
is zero. The isomorphism~\eqref{iso 2} occurs in vertical degrees
where inverting~$\mu$ (or $\bar\mu$) is an isomorphism, and $\can$
is zero. 
The isomorphism~\eqref{iso 3} uses
that $t^2 \lambda'_1$ and $(t^2 \lambda'_1)^2$ in~\eqref{eq:TC-v2} map to $\eta$
and~$\eta^2$, which are detected in~\eqref{eq: v2 Tate}, but map to zero in~\eqref{eq:5.5}. 
The homomorphisms $G \circ \can$ and~$\varphi_2^{h\bT}$ agree on classes
coming from $\olV(2)_* \gr_{\ev}^* \bS$, such as $1$, $\eta$ and~$\nu$,
hence their difference is zero on $\bF_2[\eta] \{1\}$ and $\bF_2\{t
\lambda'_1, (t \lambda'_1)^2\}$.  Both $G \circ \can(\lambda'_1)$ and
$\varphi_2^{h\bT}(\lambda'_1)$ are detected by $\lambda'_1$, hence agree
in $\olV(2)$-homotopy since there are no other classes in the same total
degree, which implies that $G \circ \can - \varphi_2^{h\bT}$ is zero on
$\lambda'_1$ and its square.  Both $G \circ \can$ and~$\varphi_2^{h\bT}$
take $\eta^4 \varepsilon_2 = \eta^3 \bar\mu$ to zero, so their difference
is zero on $\bF_2[\eta] \{\eta^4 \varepsilon_2\}$.
  
Hence we have an isomorphism
\[
\olV(2)_* \gr_{\mot}^* \TC(\ko)
	\cong \bF_2[\eta] \{1, \eta^4\varepsilon_2\}
\oplus \bF_2\{\partial, \partial \lambda'_1, t\lambda'_1,
	\lambda'_1, (t \lambda'_1)^2, (\lambda'_1)^2\} \,.
\]
The classes $t\lambda'_1$ and $(t\lambda'_1)^2$ detect $\nu$ and~$\nu^2$,
respectively.  Regarding
the multiplicative structure, the subgroup~$\bF_2\{\partial, \partial
\lambda'_1\}$ is the image of the target of~\eqref{eq:zero-homomorphism} under the connecting
homomorphism~$\partial$.  The quotient algebra by this square-zero ideal
maps isomorphically to the source of~\eqref{eq:zero-homomorphism}.  It has an associated graded
that is isomorphic to a subalgebra of the $E^\infty$-term displayed in
Figure~\ref{fig:TC-v2}.  Due to the sparsity of the situation, the only possible
hidden multiplicative extension would be from $\eta$ times $\lambda'_1$
to $(t \lambda'_1)^2$, but in fact $\eta \lambda'_1 = 0$, as we show in
Proposition~\ref{prop:BSS-from-V(2)-to-V(1)} below.
\end{proof}

\begin{figure}
\centering
\resizebox{.8\textwidth}{!}{ \input{V2TCko.inp} }
\caption{$\olV(2)_* \gr_{\mot}^* \TC(\ko)$, with lines of slope~$-1$,
	$1$ and~$1/3$ indicating multiplication by~$\partial$, $\eta$
	and~$\nu$, respectively \label{fig:V2TCko}}
\end{figure}

Next, we compute the $v_2$-Bockstein spectral sequence
\begin{equation} \label{eq:BSS-from-V(2)-to-V(1)}
E_1 = \olV(2)_* \gr_{\mot}^* \TC(\ko) \, [v_2]
	\Longrightarrow \olV(1)_* \gr_{\mot}^* \TC(\ko) \,.
\end{equation}

\begin{proposition} \label{prop:BSS-from-V(2)-to-V(1)}
In the spectral sequence~\eqref{eq:BSS-from-V(2)-to-V(1)} there is a
$d_1$-differential
\[
	d_1(\eta^4 \varepsilon_2) = v_2 \eta^4 \,,
\]
together with its various $\eta$- and $v_2$-power multiples.
This produces an algebra isomorphism
\[
\olV(1)_* \gr_{\mot}^* \TC(\ko) 
\cong \frac{
	\Lambda(\partial) \otimes \bF_2[\eta, \nu, \lambda'_1, v_2]
	}{
	(\partial \eta, \partial \nu, \eta \nu, \eta \lambda'_1,
	\nu \lambda'_1, \nu^3 = v_2 \eta^3 = \partial (\lambda'_1)^2,
	(\lambda'_1)^3 = c\cdot v_2^2 \eta^3)
	} \,,
\]
where $c \in \bF_2$ (and we have not resolved this indeterminacy).
\end{proposition}

\begin{proof}
The unit map $\bS \to \TC(\ko)$ induces a map of $v_2$-Bockstein
spectral sequences, from
\begin{equation} \label{eq:BSS-I3-to-I2}
E_1 = \Ext_{\BP_*\BP}(\BP_*, \BP_*/I_3) \, [v_2]
	\Longrightarrow \Ext_{\BP_*\BP}(\BP_*, \BP_*/I_2)
\end{equation}
shown in Figure~\ref{fig:Ext-BP/I2} to~\eqref{eq:BSS-from-V(2)-to-V(1)}
shown in Figure~\ref{fig:v2BockTCko}.  Since $v_2^2 h_{10}^4 = 0$ in the
abutment of the former, we must have that $v_2^2 \eta^4$ is a boundary in
the latter.  Considering bidegrees and $v_2$-powers, this can only happen
if $d_1(v_2 \eta^4 \varepsilon_2) = v_2^2 \eta^4$.  Hence $d_1(v_2^i \eta^j
\varepsilon_2) = v_2^{i+1} \eta^j$ for all $i\ge0$ and $j\ge4$, as claimed.
There is no room for other $v_2$-Bockstein differentials, so $E_2 =
E_\infty$ in~\eqref{eq:BSS-from-V(2)-to-V(1)}.

The relation $v_2 h_{11}^3 = v_2^2 h_{10}^3$ in the abutment
of~\eqref{eq:BSS-I3-to-I2}, see~\eqref{eq:v2h113=v22h103},
also implies that $v_2 \nu^3 = v_2^2 \eta^3$ in the abutment
of~\eqref{eq:BSS-from-V(2)-to-V(1)}.  Hence we have hidden
$\nu$-extensions from $v_2^i \nu^2$ to $v_2^{i+1} \eta^3$ for all $i\ge0$,
as shown by dashed lines of slope~$1/3$ in Figure~\ref{fig:v2BockTCko}.

The products $\partial^2$ and $\partial \eta$ lie in trivial groups.
The well-known relation $\eta \nu = 0$ implies the vanishing of $\partial
\nu$, $\eta \lambda'_1$ and $\nu \lambda'_1$.  We postpone the proof
that $\partial (\lambda'_1)^2$ is equal to $\nu^3 = v_2 \eta^3$ to
Remark~\ref{rem:nu3}.  We have not determined whether $(\lambda'_1)^3
\in \bF_2\{v_2^2 \eta^3\}$ is zero or not.
\end{proof}

\begin{figure}
\centering
\resizebox{\textwidth}{!}{ \input{v2BockTCko.inp} }
\caption{$v_2$-Bockstein $E_1 \Longrightarrow
	\olV(1)_* \gr_{\mot}^* \TC(\ko)$ \label{fig:v2BockTCko}}
\end{figure}

\begin{figure}
\centering
\resizebox{\textwidth}{!}{ \input{V2CetaTCko.inp} }
\caption{$\bF_2[v_2]$-basis for $\olA(1)_* \gr_{\mot}^* \TC(\ko)$
	\label{fig:V2CetaTCko2}}
\end{figure}

\begin{proposition} \label{prop: syntomic cohomology}
We have an isomorphism
\[
\olA(1)_* \gr_{\mot}^* \TC(\ko)
	\cong \bF_2[v_2] \otimes {} 
\left( \Lambda(\partial, \lambda'_1) \{1, \lambda_2\}
	\oplus \bF_2\{t^2 \lambda'_1,
	t \lambda'_1, t \lambda_2, t^3 \lambda'_1 \lambda_2,
	t^2 \lambda'_1 \lambda_2, t \lambda'_1 \lambda_2\} \right)
\]
of finitely generated and free $\bF_2[v_2]$-modules, where $\|v_2\| = (6,0)$,
$\|\partial\| = (-1,1)$, $\|\lambda'_1\| = (5,1)$, $\|\lambda_2\| =
(7,1)$ and $\|t\| = (-2,0)$.  See Figure~\ref{fig:V2CetaTCko2}.
\end{proposition}

\begin{proof}
This proof is similar to that of Theorem~\ref{thm:TC-V2}, to which we
refer for a more elaborate review of some of the notations.  To calculate
the effect of~$\can$ in $\olV(2) \otimes \olC\eta$-homotopy we use the
map of spectral sequences from~\eqref{eq:TC-v2Ceta} to~\eqref{eq: v2 Ceta
Tate}, described in Proposition~\ref{prop: TN} and Theorem~\ref{thm:
prismatic mod eta}, given at the $E^2$-terms by inverting~$t$.
To calculate the effect of $\varphi_2^{h\bT}$ we use Corollary~\ref{cor:
A(1) coeff} and Proposition~\ref{prop:5.4} to see that there is a
$\mu$-inverted $\bT$-homotopy fixed point spectral sequence
\begin{equation} \label{eq:V2CetaTHHkotC2hT}
\begin{aligned}
\mu^{-1} E^2(\bT) &= (\olV(2) \otimes \olC\eta)_*
	\gr_{\mot}^* \THH(\ko)^{tC_2} \, [t] \\
	&= \Lambda(\varepsilon_2) \otimes \Lambda(\lambda'_1) \{1, \lambda_2\}
	\otimes \bF_2[\mu^{\pm1}] \otimes \bF_2[t] \\
	&\Longrightarrow (\olV(2) \otimes \olC\eta)_*
	\gr_{\mot}^* (\THH(\ko)^{tC_2})^{h\bT} \,,
\end{aligned}
\end{equation}
and $\varphi_2^{h\bT}$ is given by the map of spectral sequences
from~\eqref{eq:TC-v2Ceta} to~\eqref{eq:V2CetaTHHkotC2hT} that is given at
the $E^2$-terms by inverting~$\mu$.  The differential $d^2(\varepsilon_2)
= t \mu$ carries over from the proof of Proposition~\ref{prop: TN},
leaving
\[
\mu^{-1} E^4(\bT) = \mu^{-1} E^\infty(\bT)
	= \Lambda(\lambda'_1) \{1, \lambda_2\} \otimes \bF_2[\mu^{\pm1}]
\]
concentrated on the vertical axis.  The $\olV(2) \otimes \olC\eta$-homotopy
isomorphism
\[
\Lambda(\lambda'_1) \{1, \lambda_2) \otimes \bF_2[t^{\pm4}]
	\overset{\cong}\longto
\Lambda(\lambda'_1) \{1, \lambda_2) \otimes \bF_2[\mu^{\pm1}]
\]
induced by the equivalence~$G$ must thus be given by $\lambda'_1 \mapsto \lambda'_1$, $\lambda_2 \mapsto \lambda_2$ and $t^{\pm4} \mapsto \mu^{\mp1}$.

The map $G \circ \can - \varphi_2^{h\bT}$ induces isomorphisms
\begin{align*}
\bF_2[t^4]\{t^4\} \otimes \Lambda(\lambda'_1) \{1, \lambda_2\}
	&\overset{\cong}\longto
\Lambda(\lambda'_1) \{1, \lambda_2\} \otimes \bF_2[\mu^{-1}]\{\mu^{-1}\} \\
\bF_2[\mu]\{\mu\} \otimes \Lambda(\lambda'_1) \{1, \lambda_2\}
	&\overset{\cong}\longto
\Lambda(\lambda'_1) \{1, \lambda_2\} \otimes \bF_2[\mu]\{\mu\}
\end{align*}
and the zero homomorphism
\[
\Lambda(\lambda'_1) \{1, \lambda_2\} \oplus \bF_2\{t^2 \lambda'_1,
	t \lambda'_1, t \lambda_2, t^3 \lambda'_1 \lambda_2,
	t^2 \lambda'_1 \lambda_2, t \lambda'_1 \lambda_2\}
\overset{0}\longto
\Lambda(\lambda'_1) \{1, \lambda_2\} \,,
\]
by the same arguments as in the proof of Theorem~\ref{thm:TC-V2}.
Hence we have an isomorphism
\[
(\olV(2) \otimes \olC\eta)_* \gr_{\mot}^* \TC(\ko)
	\cong \Lambda(\partial, \lambda'_1) \{1, \lambda_2\} 
\oplus \bF_2\{t^2 \lambda'_1,
	t \lambda'_1, t \lambda_2, t^3 \lambda'_1 \lambda_2,
	t^2 \lambda'_1 \lambda_2, t \lambda'_1 \lambda_2\} \,.
\]
There is no room for differentials in the $v_2$-Bockstein spectral
sequence
\[
E_1 = (\olV(2) \otimes \olC\eta)_* \gr_{\mot}^* \TC(\ko) \, [v_2]
	\Longrightarrow \olA(1)_* \gr_{\mot}^* \TC(\ko)
	\,. \qedhere
\]
\end{proof}

\begin{lemma} \label{lem:Hurewicz-image}
The unit map
\[
\pi_* \olA(1) \longto \olA(1)_* \gr_{\mot}^* \TC(\ko)
\]
sends the classes $1$, $h_{11}$, $w$, $h_{11}^2 = h_{10} w$, $h_{11} w$
and~$h_{11}^2 w$ to classes that are detected by $1$, $t \lambda'_1$,
$t \lambda_2$, $t^3 \lambda'_1 \lambda_2 \mod \partial \lambda_2$, $t^2
\lambda'_1 \lambda_2$ and~$\partial \lambda'_1 \lambda_2$, respectively.
The product $h_{11} \lambda_2$ is detected by $t \lambda'_1 \lambda_2$.
\end{lemma}

\begin{proof}
The $\olV(1)$-module cofiber sequence~\eqref{eq:V1-A1-cofibseq} induces
a long exact sequence
\[
\dots \overset{\eta}\longto
	\olV(1)_* \gr_{\mot}^* \TC(\ko)
\overset{i}\longto
	\olA(1)_* \gr_{\mot}^* \TC(\ko) 
\overset{j}\longto
	\Sigma^{2,0} \olV(1)_* \gr_{\mot}^* \TC(\ko)
\overset{\eta}\longto \dots
\]
of $\bF_2[v_2]$-modules, see Figures~\ref{fig:v2BockTCko}
and~\ref{fig:V2CetaTCko2}.  Having chosen $\lambda'_1 \in \olV(1)_*
\gr_{\mot}^* \TC(\ko)$ we choose $\lambda'_1, \lambda_2 \in \olA(1)_*
\gr_{\mot}^* \TC(\ko)$ so that $i(\lambda'_1) = \lambda'_1$ and
$j(\lambda_2) = \Sigma^{2,0} \lambda'_1$.  By exactness, $i$ is then
given by
\begin{align*}
1 &\longmapsto 1 \\
\partial &\longmapsto \partial \\
\nu &\longmapsto t \lambda'_1 \\
\lambda'_1 &\longmapsto \lambda'_1 \\
\partial \lambda'_1 &\longmapsto \partial \lambda'_1 \\
\nu^2 &\longmapsto t^3 \lambda'_1 \lambda_2 \mod \partial \lambda_2 \\
(\lambda'_1)^2 &\longmapsto t \lambda'_1 \lambda_2 \mod v_2 \partial \lambda'_1 \,,
\end{align*}
while $j$ is given by
\begin{align*}
t^2 \lambda'_1 &\longmapsto \Sigma^{2,0} \partial \\
t \lambda_2 &\longmapsto \Sigma^{2,0} \nu \\
\lambda_2 &\longmapsto \Sigma^{2,0} \lambda'_1 \\
\partial \lambda_2 &\longmapsto \Sigma^{2,0} \partial \lambda'_1 \\
t^2 \lambda'_1 \lambda_2 &\longmapsto \Sigma^{2,0} \nu^2 \\
\lambda'_1 \lambda_2 &\longmapsto \Sigma^{2,0} (\lambda'_1)^2 \\
\partial \lambda'_1 \lambda_2 &\longmapsto \Sigma^{2,0} \nu^3 \,.
\end{align*}
The formulas for~$i$ imply the claims for $1$, $\nu = h_{11}$ and~$\nu^2
= h_{11}^2$.  We know from Corollary~\ref{cor:detection} that $w$ is
detected by $t \lambda_2$, so the formulas for~$j$ imply the claims for
$\nu w = h_{11} w$ and $\nu^2 w = h_{11}^2 w$.

The $\olV(1)$-module action on $\olA(1)$ shows that $\nu \cdot \lambda_2
= h_{11} \lambda_2$ is detected by $t \lambda'_1 \cdot \lambda_2$, since
the latter product is nonzero in $\olA(1)_* \gr_{\mot}^* \TC^{-}(\ko)$.
\end{proof}

\begin{remark} \label{rem:nu3}
We can now complete the unfinished business in the proof
of Proposition~\ref{prop:BSS-from-V(2)-to-V(1)}.  Since $\nu^2 w$
is detected by $\partial \lambda'_1 \lambda_2$, and $j$ maps~$w$
to~$\Sigma^{2,0} \nu$ and $\lambda_2$ to $\Sigma^{2,0} \lambda'_1$, it
follows that $\Sigma^{2,0} \nu^3$ is detected by $\Sigma^{2,0} \partial
(\lambda'_1)^2$, so $\partial (\lambda'_1)^2$ is equal to $\nu^3 =
v_2 \eta^3$ in $\olV(1)_* \gr_{\mot}^* \TC(\ko)$.
\end{remark}

\begin{lemma} \label{lem:class-detection}
Let $\varsigma \in \olA(1)_* \gr_{\mot}^* \TC(\ko)$ be the class in bidegree $(1,1)$ detected by $t^2 \lambda'_1$.
Then $\varsigma \nu$ is the class in bidegree $(4,2)$ detected by $\partial \lambda'_1$.
\end{lemma}

\begin{proof}
By \cite{BHM93}*{Theorem~5.17}, \cite{Rog02}*{Corollary~1.21} there
is a $2$-complete equivalence $\TC(\bS) \simeq \bS \oplus \Sigma \bC
P^\infty_{-1}$, and by \cite{BM94}*{Proposition~10.9}, \cite{Dun97}*{Main
Theorem} the $3$-connected map $\bS \to \ko$ induces a $4$-connected
map $\TC(\bS) \to \TC(\ko)$.  For each $i\ge-1$ let $\Sigma\beta_i
\in H_{2i+1}(\Sigma \bC P^\infty_{-1})$ denote the generator.
The Atiyah--Hirzebruch spectral sequence
\[
E^2 = H_*(\Sigma \bC P^\infty_{-1}; \pi_* A(1))
	\Longrightarrow A(1)_*(\Sigma \bC P^\infty_{-1})
\]
has nonzero differentials $d^4(\Sigma \beta_1) = \nu \Sigma \beta_{-1}$
and $d^6(\Sigma \beta_2) = w \Sigma \beta_{-1}$.  This follows from
\cite{Mos68}*{Proposition~5.2, Proposition~5.4}, using that $w \in \<
\nu, \eta, \iota \>$ in $\pi_* A(1)$, where $\iota$ is the class of $\bS
\to A(1)$.  Hence
\[
A(1)_* \TC(\bS) \cong \bF_2\{\Sigma \beta_{-1}, \iota,
	\Sigma \beta_0, \nu \iota, \nu \Sigma \beta_0\}
\]
in stems $-1 \le * \le 4$, mapping isomorphically to
\[
A(1)_* \TC(\ko) \cong \bF_2\{\partial, 1, t^2 \lambda'_1,
	t \lambda'_1, \partial \lambda'_1\}
\]
in this range.  It follows that $\Sigma \beta_0$ maps to the class
$\varsigma$ detected by $t^2 \lambda'_1$ and $\nu \Sigma \beta_0$ to the
class detected by $\partial \lambda'_1$, which must therefore be equal
to $\varsigma \nu$.
\end{proof}

\begin{theorem}[Syntomic cohomology modulo $(2, \eta, v_1)$ of~$\ko$]
	\label{thm: syntomic A(1)}
We have an isomorphism
\begin{multline*}
\olA(1)_* \gr_{\mot}^* \TC(\ko) \cong \bF_2[v_2] \otimes \bigl(
	\bF_2\{1, \partial, \nu, w, \nu^2 = \eta w, \nu w,
		\lambda'_1 \lambda_2,
		\nu^2 w = \partial \lambda'_1 \lambda_2\} \\
	\oplus \bF_2\{\varsigma, \lambda'_1,
		\varsigma \nu = \partial \lambda'_1\}
	\oplus \bF_2\{\lambda_2, \partial \lambda_2, \nu \lambda_2\}
	\bigr)
\end{multline*}
of $\olV(1)_* \gr_{\mot}^* \TC(\ko)$-modules, where the (stem,
motivic filtration) bidegrees and detecting classes of the
$\bF_2[v_2]$-module generators are as in Table~\ref{tab:A1TCkogens}.
See also Figure~\ref{fig:A1TCko}.
\end{theorem}

\begin{proof}
This summarizes the results of Proposition~\ref{prop: syntomic cohomology}
and Lemmas~\ref{lem:Hurewicz-image} and~\ref{lem:class-detection}.
The lift of $t^3 \lambda'_1 \lambda_2$ over $\pi \: \TC(\ko) \to
\TC^{-}(\ko)$ is only defined modulo $\partial \lambda_2$ in the image
under $\partial \: \Sigma^{-1} \TP(\ko) \to \TC(\ko)$, but the image of
$\nu^2$ specifies one such choice of lift.
\end{proof}

\begin{table}[ht!]
\[
\begin{tabular}{ >{$}c<{$} >{$}c<{$} >{$}c<{$} }
\hline
\text{generator} & \text{bidegree} & \text{detecting class} \\
\hline
\hline
1 & (0,0) & 1 \\
\partial & (-1,1) & \partial \\
\varsigma & (1,1) & t^2 \lambda'_1 \\
\nu & (3,1) & t \lambda'_1 \\
w & (5,1) & t \lambda_2 \\
\lambda'_1 & (5,1) & \lambda'_1 \\
\lambda_2 & (7,1) & \lambda_2 \\
\varsigma \nu & (4,2) & \partial \lambda'_1 \\
\partial \lambda_2 & (6,2) & \partial \lambda_2 \\
\nu^2 & (6,2) & t^3 \lambda'_1 \lambda_2 \mod \partial \lambda_2 \\
\nu w & (8,2) & t^2 \lambda'_1 \lambda_2 \\
\nu \lambda_2 & (10,2) & t \lambda'_1 \lambda_2 \\
\lambda'_1 \lambda_2 & (12,2) & \lambda'_1 \lambda_2 \\
\nu^2 w & (11,3) & \partial \lambda'_1 \lambda_2 \\
\hline
\end{tabular}
\]
\caption{Bidegrees and detecting classes for the $\bF_2[v_2]$-module
	generators of $\olA(1)_* \gr_{\mot}^* \TC(\ko)$
	\label{tab:A1TCkogens}}
\end{table}

\section{Topological cyclic homology and algebraic $K$-theory} \label{sec:tc}

We now use the motivic spectral sequence
\begin{equation} \label{eq:motivicss}
E_2 = \olA(1)_* \gr_{\mot}^* \TC(\ko) \Longrightarrow A(1)_* \TC(\ko)
\end{equation}
to compute the $A(1)$-homotopy of the topological cyclic homology
of~$\ko$.  The $E_2$-term, given in Theorem~\ref{thm: syntomic A(1)},
is concentrated in even total degrees and motivic filtrations $0 \le *
\le 3$, so the only possibly nonzero differentials are
\[
d_3(v_2^i) \, \in \, \bF_2\{ v_2^{i-2} \nu^2 w\}
\]
for $i\ge2$.  We show that some, but not all, of these differentials are
nonzero.  This contrasts with the motivic spectral sequence converging
to $V(1)_* \TC(\ell)$ at odd primes~$p$, which was shown to collapse at
the $E_2$-term by Hahn--Raksit--Wilson in \cite{HRW}*{Corollary~1.4.3}.

\begin{notation} \label{not:FilGr}
We equip $A(1) \otimes \TC(\ko)$ with the relative convolution 
filtration
\[
\fil_{\mot}^{\star} (A(1) \otimes \TC(\ko))
	:= \fil_{\ev}^{\star} A(1) \otimes_{\fil_{\ev}^{\star} \bS}
	\fil_{\mot}^{\star} \TC(\ko) \,,
\]
with associated graded $\gr_{\mot}^* (A(1) \otimes \TC(\ko)) \simeq
\olA(1) \otimes \gr_{\mot}^* \TC(\ko)$.  This filtration is complete and exhaustive, since $\fil_\ev^\star A(1)$ is a finite cell $\fil_\ev^\star\bS$-module.
The motivic spectral sequence~\eqref{eq:motivicss} is the associated
spectral sequence, converging to $\pi_*(A(1) \otimes \TC(\ko))
= A(1)_* \TC(\ko)$. 

We write
\begin{align*}
\Fil_{\mot}^w A(1)_* \TC(\ko)
	&= \im( \pi_* \fil_{\mot}^w (A(1) \otimes \TC(\ko))
		\longto A(1)_* \TC(\ko)) \\
\Gr_{\mot}^w A(1)_* \TC(\ko)
	&= \Fil_{\mot}^w A(1)_* \TC(\ko) / \Fil_{\mot}^{w+1} A(1)_* \TC(\ko)
\end{align*}
for the induced (algebraic) filtration on $A(1)_* \TC(\ko)$ and its
associated graded, so that
\[
E_\infty \cong \Gr_{\mot}^* A(1)_* \TC(\ko) \,.
\]
In each stem~$n$ the $E_\infty$-term contains at most two nonzero groups,
in motivic filtrations $s \in \{0,2\}$ or $s \in \{1,3\}$, according to
the parity of~$n$.
\end{notation}

Bhattacharya--Egger--Mahowald \cite{BEM17}*{Main Theorem} proved
for each version of~$A(1)$ that there exists a $v_2^{32}$ self-map
$\Sigma^{192} A(1) \to A(1)$.  We noted in Lemma~\ref{lem:d3v24zero}
that $\id \: A(1) \to A(1)$ has additive exponent~$4$.  Hence there is
a natural $\bZ/4[v_2^{32}]$-module structure on~\eqref{eq:motivicss}
and its abutment.  This factors through a finitely generated and free
$\bF_2[v_2^4]$-module structure on the associated graded.

\begin{theorem} \label{thm:A1htpyTCko}
The motivic spectral sequence~\eqref{eq:motivicss} has nonzero
differentials
\[
d_3(v_2^i) = v_2^{i-2} \nu^2 w
\]
for $i \equiv 2, 3 \mod 4$.  The remaining differentials are zero.
Hence
\begin{align*}
\Gr_{\mot}^* A(1)_* \TC(\ko)
	&= \bF_2\{v_2^i \mid i \equiv 0, 1 \mod 4\} \\
	&\qquad\oplus \bF_2[v_2] \{\partial, \varsigma, \nu,
		\lambda'_1, w, \lambda_2\} \\
	&\qquad\oplus \bF_2[v_2] \{\varsigma \nu, \nu^2, \partial \lambda_2,
		\nu w, \nu \lambda_2, \lambda'_1 \lambda_2\} \\
	&\qquad\oplus \bF_2\{v_2^j \nu^2 w \mid j \equiv 2, 3 \mod 4\}
\end{align*}
is a finitely generated and free $\bF_2[v_2^4]$-module of rank~$52$.
Here $|\partial| = -1$, $|\varsigma| = 1$, $|\nu| = 3$, $|\lambda'_1|
= |w| = 5$, $|v_2| = 6$ and $|\lambda_2| = 7$.
\end{theorem}

\begin{proof}
The unit map $\bS \to \TC(\ko)$ induces a map from the Novikov
spectral sequence for~$A(1)$, as discussed in Lemma~\ref{lem:
Nov A1}, to the motivic spectral sequence~\eqref{eq:motivicss}.
By Lemma~\ref{lem:Hurewicz-image} this map of $E_2$-terms sends $v_2^i$
to $v_2^i$ and $v_2^i h_{11}^2 w$ to $v_2^i \partial \lambda'_1 \lambda_2$
for each $i\ge0$.  Since $d_3(1) = d_3(v_2) = 0$, $d_3(v_2^2) = h_{11}^2
w$ and $d_3(v_2^3) = v_2 h_{11}^2 w$ in the Novikov spectral sequence,
we must have $d_3(1) = d_3(v_2) = 0$, $d_3(v_2^2) = \partial \lambda'_1
\lambda_2$ and $d_3(v_2^3) = v_2 \partial \lambda'_1 \lambda_2$
in the motivic spectral sequence.  This handles the cases $0 \le i
< 4$.  By Corollary~\ref{cor:d3isv24periodic}, we know that these
$d_3$-differentials propagate $v_2^4$-periodically, as claimed.

It follows that all classes in motivic filtrations $1$ and~$2$ survive
to $E_\infty$.  In filtrations $0$ and~$3$, only the classes $v_2^i$
with $0 \le i \equiv 0, 1 \mod 4$ and $v_2^{i-2} \nu^2 w$ with $2 \le i
\equiv 0, 1 \mod 4$ survive.  Setting $0 \le j = i-2$ gives the asserted
formula.
\end{proof}

\begin{remark} \label{rem:addext}
The additive extensions
\[
0 \to E_\infty^{n,s+2} \longto A(1)_n \TC(\ko) \longto E_\infty^{n,s} \to 0
\]
(with $s=0$ for $n$ even, $s=1$ for $n$ odd) are sometimes nontrivial.
For example, we see from Figure~\ref{fig:Ext-A-A1abcd} that $2 \cdot
v_2 = \nu^2$ in $\pi_6 A(1)[ij]$ if (and only if) $[ij] \in \{[10],
[11]\}$, which implies that $A(1)_6 \TC(\ko) \cong \bZ/4\{v_2\} \oplus
\bZ/2\{\partial \lambda_2\}$ in these two cases.  We have not carried
out a complete analysis of these extension problems.
\end{remark}

Theorem~\ref{thm:A1htpyTCko} allows us to determine the $A(1)$-homotopy of
the algebraic $K$-theory spectra of~$\ko$ and $\ko^\wedge_2$.  We begin
with the $2$-complete case.

\begin{theorem} \label{thm:k2-complete}
There is an exact sequence of $\bZ/4[v_2^{32}]$-modules
\[
0 \to \Sigma^1 \bF_2 \oplus \Sigma^3 \bF_2
	\longto A(1)_* \K(\ko^\wedge_2) 
	\overset{\trc}\longto A(1)_* \TC(\ko)
	\longto \bF_2\{\partial, \varsigma\} \to 0 \,,
\]
with $|\partial| = -1$ and $|\varsigma| = 1$.
\end{theorem}

\begin{proof}
Let $\bZ_2 = \pi_0(\ko^\wedge_2)$ denote the $2$-adic integers.
By \cite{HM97}*{Theorem~D} and \cite{Dun97}*{Main Theorem}
(cf.~\cite{DGM13}*{Theorem~7.3.1.8}) applied to the $1$-connected
$\bE_\infty$-ring map $\ko^\wedge_2 \to H\bZ_2$ there is a
cofiber sequence
\[
\K(\ko^\wedge_2)^\wedge_2 \overset{\trc}\longto \TC(\ko)^\wedge_2
	\overset{p}\longto \Sigma^{-1} H\bZ_2 \,.
\]
The associated long exact sequence in $A(1)$-homotopy breaks up into
four-term exact sequences, as above.

In more detail, the $3$-connected map $A(1) \to H = H\bF_2$
identifies $A(1)_* H\bZ_2$ with 
\[
\bF_2\{1, \xi_1^2,
\bar\xi_2, \xi_1^2 \bar\xi_2\} \subset H_* H\bZ_2 \subset
\cA^\vee\,.
\]  
By \cite{BM94}*{Proposition~10.9}, $\K(\ko^\wedge_2)
\to \K(\bZ_2)$ is $2$-connected, where $\K_0(\bZ_2) =
\bZ$ and $\K_1(\bZ_2) = \bZ_2^\times$, so that
$A(1)_0 \K(\ko^\wedge_2)\cong A(1)_0 \K(\bZ_2) = \bZ/2$
and 
\[A(1)_1 \K(\ko^\wedge_2) \cong A(1)_1\K(\bZ_2) =
\bZ_2^\times/(\pm (\bZ_2^\times)^2) \cong \bZ/2\,,\]
generated by any $u \in \bZ_2^\times$ congruent to $3$
or~$5$ modulo~$8$.  This uses that $\eta \in \pi_1(\bS)$ maps to $-1
\in \bZ_2^\times \cong \K_1(\bZ_2)$.  By exactness,
we know $p \: \partial \mapsto \Sigma^{-1} 1$ and $p \: \varsigma
\mapsto \Sigma^{-1} \xi_1^2$.  Multiplication by~$\nu$ acts trivially
on $H\bZ_2$, so $p \: \varsigma \nu \mapsto 0$ does not hit
$\Sigma^{-1}\xi_1^2 \bar\xi_2$.  There is no class in degree~$2$ that $p$
could map to $\Sigma^{-1} \bar\xi_2$.  Hence these two classes instead
appear as $\Sigma^{-2} \bar\xi_2$ and $\Sigma^{-2} \xi_1^2 \bar\xi_2$
in $A(1)_*\K(\ko^\wedge_2)$, in degrees $1$ and~$3$, respectively.
\end{proof}

The proof in the integral case relies on the proven Lichtenbaum--Quillen
conjecture for $\bZ[1/2]$, cf.~\cite{Voe03} and~\cite{RW00}.

\begin{theorem} \label{thm:k}
There is an exact sequence of $\bZ/4[v_2^{32}]$-modules
\[
0 \to \Sigma^3 \bF_2
	\longto A(1)_* \K(\ko)
	\overset{\trc}\longto A(1)_* \TC(\ko)
	\longto \bF_2\{\partial, \varsigma\} \to 0 \,,
\]
with $|\partial| = -1$ and $|\varsigma| = 1$.
\end{theorem}

\begin{proof}
By \cite{Rog02}*{Theorem~3.13} there are two cofiber sequences
\begin{align*}
\K(\ko)^\wedge_2 &\overset{\trc}\longto \TC(\ko)^\wedge_2
	\overset{q}\longto X \\
\Sigma^{-2} \ku^\wedge_2 &\overset{\delta}\longto \Sigma^4 \ko^\wedge_2
	\longto X
\end{align*}
with equivalent third terms.  Passing to $A(1)$-homotopy, the second
cofiber sequence ensures that $A(1)_* X = \bF_2\{x_{-1}, x_1, x_4\}$,
where $|x_i| = i$.  The long exact sequence associated to the first
cofiber sequence then breaks up into four-term exact sequences, as shown.

This time, the details are as follows.  The $3$-connected $\bE_\infty$-ring map $\bS \to \ko$ induces a $4$-connected map $\K(\bS) \to \K(\ko)$,
where
\[
\K(\bS) \simeq \bS \oplus \Wh^{\Diff}(*) \,.
\]
Here $\Wh^{\Diff}(*)$ is $2$-connected with $\pi_3 \Wh^{\Diff}(*) =
\bZ/2$, cf.~\cite{Rog02}*{Theorem~5.8}.  Hence we have
$A(1)_0 \K(\ko) \cong
A(1)_0 \K(\bS) = \bZ/2\{1\}$,
$A(1)_1 \K(\ko) \cong A(1)_1 \K(\bS) =0$,
$A(1)_2 \K(\ko) \cong A(1)_2 \K(\bS) = 0$
and $A(1)_3 \K(\ko) \cong
A(1)_3 \K(\bS) = \bZ/2\{\nu\} \oplus \bZ/2$.  By exactness, we know $q \:
\partial \mapsto x_{-1}$ and $q \: \varsigma \mapsto x_1$, while $x_4$
must contribute to $A(1)_3 \K(\ko)$ and cannot be in the image of~$q$.
(It follows that $\nu x_1 = 0 \ne x_4$.)
\end{proof}

\begin{corollary} \label{cor:nomaptotmf}
The unit map $\bS \to \tmf$ does not factor through~$\K(\ko)$.
\end{corollary}

\begin{proof}
In fact, the unit map $A(1) \to A(1) \otimes \tmf$ does not factor
through~$A(1) \otimes \K(\ko)$, since $\pi_{20} A(1) \to A(1)_{20}
(\tmf) \cong (\bZ/2)^2$ is surjective, as can be seen using
Bruner's {\tt ext} program or from~\cite{Pha22}*{Figure~16},
while
\[
A(1)_{20} \K(\ko) \cong A(1)_{20} \TC(\ko) \cong \bZ/2\{ v_2^2 \nu w \}
\]
by Theorems~\ref{thm:A1htpyTCko} and~\ref{thm:k},
cf.~Figure~\ref{fig:A1TCko}.
\end{proof}

The proof by Hahn--Raksit--Wilson~\cite{HRW} of the height~$2$
telescope conjecture for $\TC(\ku)$ can be adapted to prove the
corresponding statement for~$\TC(\ko)$, using our Proposition~\ref{prop:
evenly free map} and Theorem~\ref{thm: A(1) coeff}.  However, as was
kindly pointed out to us by Ishan Levy, this is also a direct consequence
of the descent result of Clausen--Mathew--Naumann--Noel~\cite{CMNN20},
as we summarize below. 

\begin{theorem} \label{thm:telescope}
For each $X \in \{\K(\ko), \K(\ko^\wedge_2), \TC(\ko)\}$ the canonical map
$L^f_2 X \to L_2 X$ is an equivalence.  In other words, these spectra
all satisfy the height~$2$ telescope conjecture (at the prime~$2$).
\end{theorem}

\begin{proof}
According to~\cite{Dun97}, \cite{HM97} and~\cite{RW00} there are
equivalences
\[
L_{T(2)} \K(\ku) \simeq L_{T(2)} \K(\ku^\wedge_2)
	\simeq L_{T(2)} \TC(\ku) \,.
\]
By~\cite{HRW}*{Theorem~6.6.4}, $L^f_2 \TC(\ku) \simeq L_2
\TC(\ku)$, which by~\cite{Hov95}*{Corollary~2.2} implies that
$L_{T(2)} \TC(\ku) \simeq L_{K(2)} \TC(\ku)$ is $K(2)$-local.  Applying
descent~\cite{CMNN20}*{Theorem~1.8, Theorem~1.10} along $\ko \to \ku$ or $\ko^\wedge_2
\to \ku^\wedge_2$, for $E = \K$ or $\TC$, it follows that $L_{T(2)}
\K(\ko)$, $L_{T(2)} \K(\ko^\wedge_2)$ and $L_{T(2)} \TC(\ko)$ are all
limits of $K(2)$-local spectra, hence are $K(2)$-local.  In particular, there is an equivalence
$L_{T(2)} X \simeq L_{K(2)} X$ in each case.  Standard telescopic and
chromatic fracture squares~\cite{DFHH14}*{Proposition~6.2.2}, and the
known validity of the height~$1$ telescope conjecture~\cite{Mah81},
\cite{Bou79}*{Proposition~4.2}, then imply that $L^f_2 X \simeq L_2 X$
in each case.
\end{proof} 

Our calculations also show that $\TC(\ko)$ has fp-type~$2$ in
the sense of Mahowald--Rezk~\cite{MR99}, with the following consequence.

\begin{theorem} \label{thm:LQ}
For $X \in \{\K(\ko), \K(\ko^\wedge_2), \TC(\ko)\}$ and $Y \in \{X_{(2)},
X^\wedge_2\}$, the canonical map $Y \to L_2^f Y$ is an equivalence in
all sufficiently large degrees.
\end{theorem}

\begin{proof}
Theorems~\ref{thm:k}, \ref{thm:k2-complete} and~\ref{thm:A1htpyTCko} show,
respectively, that $(A(1)/(v_2^{32}))_* X^\wedge_2$ is finite for each of
the three choices for~$X$.  This implies that $X^\wedge_2$ has fp-type~$2$
in the sense of~\cite{MR99}*{p.~5}, by \cite{MR99}*{Proposition~3.2}.
According to~\cite{MR99}*{Theorem~8.2}, this implies that the
Brown--Comenetz dual spectrum $IC_2^f X^\wedge_2$ is bounded below and,
consequently, that the fiber $C_2^f X^\wedge_2$ of the map $X^\wedge_2\to
L_2^f X^\wedge_2$ is bounded above (cf.~\cite{HW22}*{Theorem~3.1.3}).
Using the pullback square
\[
\xymatrix{
	X_{(2)} \ar[r] \ar[d] & X^\wedge_2 \ar[d] \\
	X_{(2)}[1/2] \ar[r] & X^\wedge_2[1/2] \,,
}
\]
and the fact that $X_{(2)}[1/2]$ and $X^\wedge_2[1/2]$ are $L_2^f$-local,
it also follows that $X_{(2)} \to L_2^f X_{(2)}$ is an equivalence in
all sufficiently large degrees.
\end{proof}

\end{document}